\documentclass{lms-mod}

\usepackage[T1]{fontenc}
\usepackage[latin1]{inputenc}
\usepackage[english]{babel}
\usepackage{lmodern}

\usepackage{amssymb,amsfonts,amsmath,latexsym, mathtools}
\usepackage{enumerate, caption}
\usepackage[unicode]{hyperref}

\usepackage[margin=1.1in]{geometry}

\newtheorem{theoreme}{Theorem}[section]
\newtheorem{coro}[theoreme]{Corollary}
\newtheorem{lemme}[theoreme]{Lemma}
\newtheorem{prop}[theoreme]{Proposition}
\newtheorem{thmext}{Theorem}

\newunnumbered{remarque}{Remark}

\newcommand{\cond}{{\rm cond}}
\newcommand{\prim}{\text{ primitive}}

\renewcommand{\bar}{\overline}

\newcommand{\ssum}[1]{\sum_{\substack{#1}}}

\newcommand{\e}{{\rm e}}
\newcommand{\dd}{{\rm d}}
\newcommand{\cA}{{\mathcal A}}
\newcommand{\cB}{{\mathcal B}}
\newcommand{\cE}{{\mathcal E}}
\newcommand{\cH}{{\mathcal H}}

\newcommand{\cC}{{\mathcal C}}
\newcommand{\cW}{{\mathcal W}}

\newcommand{\tphi}{{\tilde \phi}}
\newcommand{\dphi}{{\dot \phi}}
\newcommand{\cphi}{{\check \phi}}

\renewcommand{\hat}{\widehat}

\newcommand{\cD}{{\mathcal D}}

\newcommand{\cM}{{\mathcal M}}

\newcommand{\cJ}{{\mathcal J}}

\newcommand{\cT}{{\mathcal T}}
\newcommand{\cS}{{\mathcal S}}
\newcommand{\cL}{{\mathcal L}}
\newcommand{\cK}{{\mathcal K}}
\newcommand{\cX}{{\mathcal X}}
\newcommand{\cR}{{\mathcal R}}

\newcommand{\ca}{{\mathfrak a}}
\newcommand{\cb}{{\mathfrak b}}
\newcommand{\cc}{{\mathfrak c}}
\newcommand{\cu}{{\mathfrak u}}

\newcommand{\ee}{{\varepsilon}}
\newcommand{\qt}{\tilde{q}}
\newcommand{\chit}{\tilde{\chi}}

\newcommand{\bfH}{{\mathbf H}}
\newcommand{\bfN}{{\mathbf N}}
\newcommand{\bfZ}{{\mathbf Z}}

\newcommand{\bfR}{{\mathbf R}}
\newcommand{\bfC}{{\mathbf C}}
\newcommand{\bfUn}{{\mathbf 1}}
\newcommand{\Id}{{\rm Id}}

\newcommand{\bsl}{\backslash}

\newcommand{\BFI}{Bombieri--Friedlander--Iwaniec}
\newcommand{\Milicevic}{Mili{\'c}evi{\'c}}

\DeclareMathOperator{\li}{li}

\newcommand{\vth}{{\vartheta}}
\newcommand{\vphi}{{\varphi}}
\renewcommand{\tilde}{\widetilde}

\newcommand{\card}{{\rm card\ }}

\newcommand{\reg}{\text{\rm reg}}
\newcommand{\exc}{\text{\rm exc}}

\newcommand{\floor}[1]{{\left\lfloor {#1} \right\rfloor}}

\renewcommand{\mod}[1]{\ ({\rm mod\ }#1)}

\renewcommand\Re{\operatorname{\mathfrak{Re}}}
\renewcommand\Im{\operatorname{\mathfrak{Im}}}

\newcommand\numberthis{\addtocounter{equation}{1}\tag{\theequation}}

\classno{11L07 (primary), 11F30, 11N75, 11N13}

\numberwithin{equation}{section}

\title[Error term in the dispersion method]{Sums of Kloosterman sums in arithmetic progressions, \\ and the error term in the dispersion method}
\author{Sary Drappeau}
\date{\today}

\begin{document}

\begin{abstract}
We prove a bound for quintilinear sums of Kloosterman sums, with congruence conditions on the ``smooth'' summation variables. This generalizes classical work of Deshouillers and Iwaniec, and is key to obtaining power-saving error terms in applications, notably the dispersion method.

As a consequence, assuming the Riemann hypothesis for Dirichlet~$L$-functions, we prove power-saving error term in the Titchmarsh divisor problem of estimating~$\sum_{p\leq x}\tau(p-1)$. Unconditionally, we isolate the possible contribution of Siegel zeroes, showing it is always negative. Extending work of Fouvry and Tenenbaum, we obtain power-saving in the asymptotic formula for~$\sum_{n\leq x}\tau_k(n)\tau(n+1)$, reproving a result announced by Bykovski{\u\i} and Vinogradov by a different method. The gain in the exponent is shown to be independent of~$k$ if a generalized Lindelöf hypothesis is assumed.
\end{abstract}

\maketitle

\tableofcontents


\section{Introduction}

Understanding the joint multiplicative structure of pairs of neighboring integers such as~$(n, n+1)$ is an outstanding problem in multiplicative number theory. A quantitative way to look at this question is to try to estimate sums of the type
\begin{equation}
\sum_{n\leq x} f(n) g(n+1)\label{eq:question-corre}
\end{equation}
when~$f, g : \bfN\to\bfC$ are two functions that are of multiplicative nature -- multiplicative functions for instance, or the characteristic function of primes. In this paper we are motivated by two instances of the question~\eqref{eq:question-corre}: the Titchmarsh divisor problem, and correlation of divisor functions.

In what follows,~$\tau(n)$ denotes the number of divisors of the integer~$n$, and more generally,~$\tau_k(n)$ denotes the number of ways one can write~$n$ as a product of~$k$ positive integers. Studying the function~$\tau_k$ gives some insight into the factorisation of numbers\footnote{There are a number of formulas relating the characteristic function of primes to linear combination of divisor-like functions, for instance Heath-Brown's identity~\cite{HB3}.}, which is deeper but more difficult to obtain as~$k$ grows.

\subsection{The Titchmarsh divisor problem}

One would like to be able to evaluate, for~$k\geq 2$, the sum
\begin{equation}
\sum_{p\leq x}\tau_k(p-1)\label{eq:somme-TM}
\end{equation}
where~$p$ denotes primes. \textit{A priori}, this would require understanding primes up to~$x$ in arithmetic progressions of moduli up to~$x^{1-1/k}$. The case~$k\geq 3$ seems far from reach of current methods, so we consider~$k=2$.

In place of~\eqref{eq:somme-TM}, one may consider
\[ T(x) := \sum_{1<n\leq x}\Lambda(n)\tau(n-1) \]
where~$\Lambda$ is the von Mangoldt function~\cite[formula~(1.39)]{IK}. In 1930, Titchmarsh~\cite{Titchmarsh} first considered the problem, and proved~$T(x)\sim C_1x\log x$ for some constant~$C_1>1$ under the assumption that the Riemann hypothesis holds for all Dirichlet $L$-functions. This asymptotic was proved unconditionally by Linnik~\cite{Linnik} using his so-called dispersion method. Simpler proofs were later given by Rodriquez~\cite{Rodriquez} and Halberstam~\cite{Halberstam} using the theorems of Bombieri-Vinogradov and Brun-Titchmarsh. Finally the most precise known estimate was proved independently by \BFI~\cite{BFI} and Fouvry~\cite{FouvryTit}. To state their result, let us denote
\[ C_1 := \prod_p\Big(1+\frac1{p(p-1)}\Big), \qquad C_2 := \sum_p\frac{\log p}{1+p(p-1)} .\]
\begin{thmext}[Fouvry~\cite{FouvryTit}, \BFI~\cite{BFI}]\label{thmext:TM}
For all~$A>0$ and all~$x\geq 3$,
\[ T(x) = C_1x\big\{\log x + 2\gamma-1 - 2C_2\big\} + O_A\big(x/(\log x)^A\big) .\]
\end{thmext}
In this statement,~$\gamma$ denotes Euler's contant. See also~\cite{Felix,Fiorilli2} for generalizations in arithmetic progressions; and~\cite{ABR} for an analogue in function fields.

The error term in Theorem~\ref{thmext:TM} is due to an application of the Siegel-Walfisz theorem~\cite[Corollary~5.29]{IK}. One could wonder whether assuming the Riemann Hypothesis generalized to Dirichlet $L$-functions (GRH) allows for power-saving error term to be obtained (as is the case for the prime number theorem in arithmetic progressions~\cite[Corollary~13.8]{MontgomeryVaughan}). The purpose of this paper is to prove that such is indeed the case.
\begin{theoreme}\label{thm:TM-cond}
Assume GRH. Then for some~$\delta>0$ and all~$x\geq 2$,
\[ T(x) = C_1x\big\{\log x + 2\gamma-1-2C_2\big\} + O(x^{1-\delta}) .\]
\end{theoreme}

Unconditionally, we quantify the influence of hypothetical Siegel zeroes. Define, for~$q\geq 1$,
$$ C_1(q) := \frac1{\vphi(q)}\prod_{p\nmid q} \Big( 1 + \frac1{p(p-1)}\Big), \qquad
C_2(q) := \sum_{p\nmid q}\frac{\log p}{1+p(p-1)} $$
where~$\vphi$ is Euler's totient function. Note that~$C_1 = C_1(1)$ and~$C_2 = C_2(1)$.
\begin{theoreme}\label{thm:TM-uncond}
There exist~$b>0$ and~$\delta>0$ such that
\begin{align*}
T(x) = &\ C_1x\big\{\log x + 2\gamma-1-2C_2 \big\} \\
&\ - C_1(q)\frac{x^\beta}{\beta}\big\{\log\big(\frac x{q^2}\big) +  2\gamma-\frac1{\beta}-2C_2(q)\big\} + O\big(x\e^{-\delta\sqrt{\log x}}\big).
\end{align*}
The second term is only to be taken into account if there is a primitive character~$\chi\mod{q}$ with~$q\leq \e^{\sqrt{\log x}}$ whose Dirichlet $L$-function has a real zero~$\beta$ with~$\beta \geq 1-b/\sqrt{\log x}$.
\end{theoreme}
By partial summation, one deduces
\begin{coro}\label{coro:TM-partial}
In the same notation as Theorem~\ref{thm:TM-uncond},
$$ \sum_{p\leq x}\tau(p-1) = C_1\{x+2\li(x)(\gamma-C_2)\} - C_1(q)\{\frac{x^\beta}{\beta}+2\li(x^\beta)(\gamma-\log q - C_2(q))\} + O(x\e^{-\delta\sqrt{\log x}}). $$
\end{coro}
The method readily allows for more general shifts~$\tau(p-a)$, $0<|a|\leq x^\delta$ (\textit{cf.}~\cite[Corollary~3.4]{Fiorilli} for results on the uniformity in~$a$). The contribution of the exceptional character in Corollary~\ref{coro:TM-partial} would then have a twist by~$\chi(a)$. Since~$\chi$, if it exists, is a real character, then~$\chi(a)=1$ whenever~$a$ is a perfect square (for instance~$a=1$), in which case we have an unconditional inequality.
\begin{coro}\label{coro:TM-onesided}
With an effective implicit constant, we have
$$ \sum_{p\leq x}\tau(p-1) \leq C_1\{x + 2\li(x)(\gamma-C_2)\} + O(x\e^{-\delta\sqrt{\log x}}). $$
\end{coro}

We conclude our discussion of the Titchmarsh divisor problem by mentioning the important work of Pitt~\cite{Pitt}, who proves~$\sum_{p\leq x}a(p-1)\ll x^{1-\delta}$ for the sequence~$(a(n))$ of Fourier coefficients of an integral weight holomorphic cusp form (which is a special case of~\eqref{eq:question-corre} when the~$(a(n))$ are Hecke eigenvalues). It is a striking feature that power-saving can be proved \emph{unconditionally} in this situation.

\subsection{Correlation of divisor functions}

Another instance of the problem~\eqref{eq:question-corre} is the estimation, for integers~$k, \ell\geq 2$, of the quantity
\[ \cT_{k, \ell}(x) := \sum_{n\leq x}\tau_k(n)\tau_{\ell}(n+1) .\]
The conjectured estimate is of the shape
$$ \cT_{k, \ell}(x) \sim C_{k, \ell}x(\log x)^{k+\ell-2} $$
for some constants~$C_{k, \ell}>0$. The case~$k=\ell$ is of particular interest when one looks at the~$2k$-th moment of the Riemann~$\zeta$ function~\cite[§7.21]{TM-zeta} (see also~\cite{CG}): in that context, the size of the error term is a non-trivial issue, as well as the uniformity with which one can replace~$n+1$ above by~$n+a$, $a\neq 0$. Current methods are ineffective when~$k, \ell\geq 3$, so we focus on the case~$\ell=2$. Let us denote
$$ \cT_k(x) := \sum_{n\leq x}\tau_k(n)\tau(n+1). $$
There has been several works on the estimation of~$\cT_k(x)$. There are nice expositions of the history of the problem in the papers of Heath-Brown~\cite{HB} and Fouvry-Tenenbaum~\cite{FT-Piltz}. The latest published results may be summarized as follows.
\begin{thmext}
There holds:
\begin{align*}
\cT_2(x) =&\ xP_2(\log x) + O_\ee(x^{2/3+\ee}), && (\text{\cite{DI-BinDiv}}), \\
\cT_3(x) =&\ xP_3(\log x) + O(x^{1-\delta}), && (\text{\cite{Deshouillers}}, \text{\cite{Topacogullari}}), \\
\cT_k(x) =&\ xP_k(\log x) + O_k(x\e^{-\delta\sqrt{\log x}}) \quad \text{for fixed $k\geq 4$,} && (\text{\cite{FT-Piltz}}). \numberthis\label{eq:estim-Tk-4}
\end{align*}
Here~$\ee>0$ is arbitrary, $\delta>0$ is some constant depending on~$k$, and~$P_k$ is an explicit degree~$k$ polynomial.
\end{thmext}
The error term of~\eqref{eq:estim-Tk-4} resembles that in the distribution of primes in arithmetic progressions, where it is linked to the outstanding problem of zero-free regions of~$L$-functions. However there is no such process at work in~\eqref{eq:estim-Tk-4}, leaving one to wonder if power-saving can be achieved. In~\cite{BV}, Bykovski{\u\i} and Vinogradov announce results implying
\begin{equation}
\cT_k(x) = \ xP_k(\log x) + O_k(x^{1-\delta/k}) \qquad (k\geq 4, x\geq 2) \label{eq:estim-Tk}
\end{equation}
for some absolute~$\delta>0$, and sketch ideas of a proof. The proposed argument, in a way, is dual to the method adopted in~\cite{FT-Piltz}\footnote{In~\cite{Motohashi,FT-Piltz}, the authors study the distribution of~$\tau_k(n)$ in progressions of moduli up to~$x^{1/2}$, while in~\cite{BV} the authors address the distribution of~$\tau(n)$ in progressions of moduli up to~$x^{1-1/k}$.} (which is related to earlier work of Motohashi~\cite{Motohashi}). Here we take up the method of~\cite{FT-Piltz} and prove an error term of the same shape.
\begin{theoreme}\label{thm:Tk}
For some absolute~$\delta>0$, the estimate~\eqref{eq:estim-Tk} holds.
\end{theoreme}
In view of~\cite{BV}, Theorem~\ref{thm:Tk} is not new. However the method is somewhat different. In the course of our arguments, the analytic obstacle to obtaining an error term~$O_k(x^{1-\delta})$ ($\delta$ independent of~$k$) in the estimate~\eqref{eq:estim-Tk} will appear clearly: it lies in the estimation of sums of the shape~$\sum_{n\leq x}\tau_k(n)\chi(n)$ for Dirichlet characters~$\chi$ of small conductors. This issue is know to be closely related to the growth of Dirichlet~$L$-functions inside the critical strip~\cite{FrIw2}.
\begin{theoreme}\label{thm:Tk-cond}
Assume that Dirichlet $L$-functions satisfy the Lindelöf hypothesis, meaning~$L(\tfrac12+it, \chi) \ll_\ee (q(|t|+1))^\ee$ for~$t\in\bfR$ and~$\chi\mod{q}$. Then for some \emph{absolute}~$\delta>0$,
\begin{equation}
\cT_k(x) = xP_k(\log x) + O_k(x^{1-\delta}) \qquad (k\geq 4,\ x\geq 2)\label{eq:estim-Tk-cond}
\end{equation}
\end{theoreme}
The standard conjecture for the error term in the previous formula is~$O_{k,\ee}(x^{1/2+\ee})$. We have not sought optimal values for~$\delta$ in Theorems~\ref{thm:Tk} and~\ref{thm:Tk-cond}. In the case of~\eqref{eq:estim-Tk}, the method of~\cite{BV} seems to yield much better numerical results.

Our method readily allows to replace the shift~$n+1$ in Theorem~\ref{thm:Tk} by~$n+a$, $0<|a|\leq x^\delta$ with~$\delta$ \textit{independent of~$k$}. We give some explanations in Section~\ref{sec:remark-uniformity-a} below regarding this point.

\begin{acknowledgements}
This work was done while the author was a CRM-ISM Postdoctoral Fellow at Université de Montréal. The author is indebted to R. de la Bretèche, É. Fouvry, V. Blomer, D. {\Milicevic}, S. Bettin, G. Tenenbaum, B. Topacogullari and A. Granville for valuable discussions and comments, and to an anonymous referee for helpful remarks and careful reading of the manuscript. The author is particularly grateful to V. Blomer for making a preprint of~\cite{BM} available, and for making him aware of the reference~\cite{BV}; and finally to B. Topacogullari for correcting a significant oversight in an earlier version.
\end{acknowledgements}


\section{Overview}

The method at work in Theorems~\ref{thm:TM-cond}, \ref{thm:TM-uncond} and~\ref{thm:Tk} is the dispersion method, which was pioneered by Linnik~\cite{Linnik} and studied intensively in groundbreaking work of Bombieri, Fouvry, Friedlander and Iwaniec~\cite{Fouvry82,FI,BFI} on primes in arithmetic progressions. It has received a large publicity recently with the breakthrough of Zhang~\cite{Zhang} (see also~\cite{Polymath}), giving the first proof of the existence of infinitely many bounded gaps between primes (which was shown later by Maynard~\cite{Maynard} and Tao (unpublished) not to require such strong results).

In our case, by writing~$\tau(n)$ as a convolution of the constant function~$1$ with itself, the problem is reduced to estimating the mean value of~$\Lambda(n)$ or~$\tau_k(n)$ when~$n\leq x$ runs over arithmetic progressions~$\mod{q}$, with an average over~$q$. It is crucial that the uniformity be good enough to average over~$q\leq\sqrt{x}$. In the case of~$\Lambda(n)$, that is beyond what can currently be done for individual moduli~$q$, even assuming the GRH. The celebrated theorem of Bombieri-Vinogradov~\cite[Theorem~17.1]{IK} allows to exploit the averaging over~$q$, but if one wants error terms at least as good as~$O(x/(\log x)^2)$ for instance, it barely fails to be useful.

Linnik's dispersion method~\cite{Linnik}, which corresponds at a technical level to an acute use of the Cauchy--Schwarz inequality, offers the possibility for such results, on the condition that one has good bounds on some types of exponential sums related to Kloosterman sums. One then appeals to Weil's bound~\cite{Weil}, or to the more specific but stronger bounds of Deshouillers-Iwaniec~\cite{DI} which originate from the theory of modular forms through Kuznetsov's formula.

The Deshouillers-Iwaniec bounds apply to exponential sums of the following kind:
$$ \ssum{c, d, n, r, s \\ (rd, sc)=1 } b_{n, r, s} g(c, d) \e\Big(n\frac{\bar{rd}}{sc}\Big) $$
where~$c, d, n, r, s$ are integers in specific intervals, $(b_{n, r, s})$ is a generic sequence, and~$g(c, d)$ depends in a smooth way on~$c$ and~$d$. Here and in what follows, $\e(x)$ stand for~$\e^{2\pi i x}$, and~$\bar{rd}$ denotes the multiplicative inverse of~$rd\mod{sc}$ (since~$\e(x)$ is of period~$1$, the above is well-defined). It is crucial that the variables~$c$ and~$d$ are attached to a smooth weight~$g(c, d)$: for the variable~$d$, in order to reduce to complete Kloosterman sums~$\mod{sc}$; and for the variable~$c$, because the object that arises naturally in the context of modular forms is the average of Kloosterman sums over moduli (with smooth weight).

In the dispersion method, dealing with largest common divisors (appearing through the Cauchy--Schwarz inequality) causes some issues. The most important of these is that the phase function that arises in the course of the argument takes a form similar to
\begin{equation}
\e\Big( n\frac{\bar{rd}}{sc}  + \frac{\bar{cd}}{q}\Big)\label{eq:phases-2}
\end{equation}
rather than the above. Here~$q$ can be considered small and fixed, but even then, the second term oscillates chaotically.

Previous works avoided the issue altogether by using a sieve beforehand in order to reduce to the favourable case~$q=1$ (see Lemma 4 and Section~3 of~\cite{FouvryTit}, and Lemma 4 and Theorem~5* of~\cite{BFI}). Two error terms are then produced, which take the form
$$ \e^{-\delta(\log x)/\log z} + z^{-1} $$
where~$z\leq x$ is a parameter. Roughly speaking, the first term corresponds to sieving out prime factors smaller than~$z$, with the consequence that the ``bad'' variable~$q$ above is either~$1$ or larger than~$z$. The second term corresponds to a trivial bound on the contribution of~$q>z$. The best error term one can achieve in this way is~$\e^{-\delta\sqrt{\log x}}$, whence the estimate~\eqref{eq:estim-Tk-4}.

By contrast, in the present paper, we transpose the work of Deshouillers-Iwaniec in a slightly more general context, which allows to encode phases of the kind~\eqref{eq:phases-2}. More specifically, whereas Deshouillers and Iwaniec worked with modular forms with trivial multiplier system, we find that working with multiplier systems defined by Dirichlet characters allows one to encode congruence conditions~$\mod{q}$ on the ``smooth'' variables~$c$ and~$d$. This is partly inspired by recent work of Blomer and {\Milicevic}~\cite{BM}. The main result, which extends~\cite[Theorem~12]{DI} and has potential for applications beyond the scope of the present paper, is the following.
\begin{theoreme}\label{thm:quintilin}
Let~$C, D, N, R, S \geq 1$, and~$q, c_0, d_0\in\bfN$ be given with~$(c_0d_0, q)=1$. Let~$(b_{n,r,s})$ be a sequence supported inside ~$(0, N]\times(R, 2R]\times(S, 2S]\cap\bfN^3$.
Let~$g :\bfR_+^5\to\bfC$ be a smooth function compactly supported in~$]C, 2C]\times]D, 2D]\times (\bfR_+^*)^3$, satisfying the bound
\begin{equation}
\frac{\partial^{\nu_1+\nu_2+\nu_3+\nu_4+\nu_5}g} {\partial c^{\nu_1}\partial d^{\nu_2}\partial n^{\nu_3}\partial r^{\nu_4} \partial s^{\nu_5}} (c,d, n, r, s)
\ll_{\nu_1, \nu_2, \nu_3, \nu_4, \nu_5} \{c^{-\nu_1}d^{-\nu_2}n^{-\nu_3}r^{-\nu_4}s^{-\nu_5}\}^{1-\ee_0}\label{eq:cond-derg-quintilin}
\end{equation}
for some small~$\ee_0>0$ and all fixed~$\nu_j\geq 0$. Then
\begin{equation}
\begin{aligned}
\underset{\substack{c \equiv c_0\text{ and } d\equiv d_0 \mod{q} \\ (qrd, sc)=1}}{\sum_c\sum_d\sum_n\sum_r\sum_s} b_{n, r, s}& g(c, d, n, r, s) \e\Big(n\frac{\bar{rd}}{sc}\Big) \\ &\ll_{\ee,\ee_0} (qCDNRS)^{\ee+O(\ee_0)}q^{3/2} K(C,D,N,R,S) \|b_{N,R,S}\|_2,\label{eq:majo-quintilin}
\end{aligned}
\end{equation}
where~$\|b_{N, R, S}\|_2 = \big(\sum_{n, r, s}|b_{n,r,s}|^2\big)^{1/2}$ and
\[ K(C,D,N,R,S)^2 = qCS(RS+N)(C+RD)+C^2DS\sqrt{(RS+N)R}+D^2NRS^{-1}. \]
\end{theoreme}
We have made no attempt to optimize the dependence in~$q$. In all of the applications considered here, we only apply the estimate~\eqref{eq:majo-quintilin} for small values of~$q$, say~$q=O((CDNRS)^{\ee_1})$ for some small~$\ee_1>0$. Such being the case, the reader might still wonder why the bound tends to grow with~$q$. The main reason is that upon completing the sum over~$d$, we obtain a Kloosterman sum to modulus~$scq$, which grows with~$q$.

In the footsteps of previous work~\cite{Drappeau}, for the proof of our equidistribution results, we separate from the outset of the argument the contribution of characters of small conductors (which is typically well-handled by complex-analytic methods). We only apply the dispersion method to the contribution of characters of large conductors. There is considerable simplification coming from the fact that no ``Siegel-Walfisz''-type hypothesis is involved in the latter, which allows us to focus on the combinatorial aspect of the method\footnote{It is more straightforward to study the mean value of~$\tau_k(n)$ in arithmetic progressions of small moduli, than a~$k$-fold convolution of slowly oscillating sequences, each supported on a dyadic interval.}.

\medskip

In Section~\ref{sec:lemmas}, we state a few useful lemmas. In Section~\ref{sec:sums-kloost-sums}, we adapt the arguments of~\cite{DI} to prove Theorem~\ref{thm:quintilin}. In Section~\ref{sec:convo-bin}, we employ a variant of the dispersion method to obtain equidistribution for binary convolutions in arithmetic progressions. In Sections~\ref{sec:appl-titchm-divis} and~\ref{sec:appl-corr-divis}, we derive Theorems~\ref{thm:TM-cond},~\ref{thm:TM-uncond},~\ref{thm:Tk} and~\ref{thm:Tk-cond}.

\subsubsection*{Notations}

We use the convention that the letter~$\ee$ denotes a positive number that can be chosen arbitrarily small and whose value may change at each occurence. The letter~$\delta>0$ will denote a positive number whose value may change from line to line, and whose dependence on various parameters will be made clear by the context.

We define the Fourier transform~$\hat f$ of a function~$f$ as
$$ {\hat f}(\xi) = \int_\bfR f(t) \e(-\xi t)\dd t .$$
If~$f$ is smooth and compactly supported, the above is well-defined and there holds
$$ f(t) = \int_\bfR {\hat f}(\xi)\e(\xi t)\dd\xi. $$
If moreover~$f$ is supported inside~$[-M, M]$ for some~$M\geq 1$ and~$\|f^{(j)}\|_{\infty} \ll M^{-j}$ for~$j\in\{0, 2\}$, then we have
$$ {\hat f}(\xi) \ll \frac{M}{1 + (M\xi)^2}. $$


\section{Lemmas}\label{sec:lemmas}

In this section we group a few useful lemmas. The first is the Poisson summation formula, which is very effective at estimating the mean value of a smooth function along arithmetic progressions.
\begin{lemme}[{\cite[Lemma~2]{BFI}}]\label{lemme:poisson}
Let~$M\geq 1$ and~$f:\bfR\to\bfC$ be a smooth function supported on an interval~$[-M, M]$ satisfying~$\|f^{(j)}\|_\infty \ll_j M^{-j}$ for all~$j\geq 0$. For all~$q\geq 1$ and~$(a, q)=1$, with~$H := q^{1+\ee}/M$, we have
$$ \sum_{m\equiv a\mod{q}} f(m) = \frac1q\sum_{|h|\leq H}{\hat f}\Big(\frac hq\Big)\e\Big(\frac{ah}q\Big) + O_\ee\Big(\frac1q\Big). $$
\end{lemme}
The next lemma is quoted from work of Shiu~\cite[Theorem~2]{Shiu}, and gives an upper bound of the right order of magnitude for sums of~$\tau_k(n)$ in short intervals and arithmetic progressions of large moduli. It is an analogue of the celebrated Brun-Titchmarsh inequality~\cite[Theorem~6.6]{IK}.
\begin{lemme}[{\cite[Theorem~2]{Shiu}}]\label{lemme:shiu}
For~$k\geq 2$,~$x\geq 2$, $x^{1/2}\leq y\leq x$, $(q, a)\in\bfN$ with~$(a, q)=1$ and~$q\leq y^{3/4}$,
$$ \ssum{x-y < n \leq x \\ n\equiv a\mod{q}} \tau_k(n) \ll_k \frac{y}{q}\big(\frac{\vphi(q)}{q}\log x\big)^{k-1}. $$
\end{lemme}
Note that such a result could also be deduced from earlier work of Barban and Vekhov~\cite{BarbanVekhov}; see also~\cite{Henriot} for the most recent results on this topic.

The next lemma is the classical form of the multiplicative large sieve inequality~\cite[Theorem~7.13]{IK}.
\begin{lemme}\label{lemme:GC}
Let~$(a_n)$ be a sequence of numbers, and~$N, M, Q\geq 1$. Then
$$ \sum_{q\leq Q}\frac{q}{\vphi(q)}\ssum{\chi\mod{q}\\\chi\prim}\Big|\sum_{M<n\leq M+N} a_n\chi(n)\Big| \leq (Q^2+N-1)\sum_{N<n\leq N+M}|a_n|^2. $$
\end{lemme}
We quote from~\cite[Number Theory Result~1]{Harper-AB} the following version of the P{\'o}lya-Vinogradov inequality with an explicit dependence on the conductor.
\begin{lemme}\label{lemme:PV}
Let~$\chi\mod{q}$ be a character of conductor~$1\neq r|q$, and~$M, N\geq 1$. Then
$$ \sum_{M<n\leq M+N} \chi(n) \ll \tau(q/r)\sqrt{r}\log r. $$
\end{lemme}


\section{Sums of Kloosterman sums in arithmetic progressions}\label{sec:sums-kloost-sums}

Theorem~\ref{thm:quintilin} is proved by a systematic use of the Kuznetsov formula, which establishes a link between sums of Kloosterman sums and Fourier coefficients of holomorphic and Maa{\ss} cusp forms. There is numerous bibliography about this theory; we refer the reader to the books~\cite{Iwaniec-Book-Spectral,Iwaniec-Book-Automorphic} and to chapters 14--16 of~\cite{IK} for references.

Most of the arguments in~\cite{DI} generalizes without the need for substantial new ideas. We will introduce the main notations, and of course provide the required new arguments; but we will refer to~\cite{DI} for the parts of the proofs that can be transposed~\textit{verbatim}.

\subsection{Setting}\label{sec:DI-contexte}
\subsubsection{Kloosterman sums}

Let~$q\geq 1$. The setting is the congruence subgroup
$$ \Gamma = \Gamma_0(q) := \Big\{ \begin{pmatrix} a & b \\ c & d \end{pmatrix}\in SL_2(\bfZ), c\equiv 0\mod{q} \Big\}. $$
Let~$\chi$ be a character modulo~$q_0|q$, and~$\kappa\in\{0, 1\}$ such that~$\chi(-1)=(-1)^\kappa$. We warn the reader that the variable~$q$ has a different meaning in Sections~\ref{sec:DI-contexte} and~\ref{sec:DI-gc}, than in the statement of Theorem~\ref{thm:quintilin} (where it corresponds to~$qrs$). The character~$\chi$ induces a multiplier (\textit{i.e.} here, a multiplicative function) on~$\Gamma$ by
$$ \chi\Big(\begin{pmatrix}a & b \\ c & d\end{pmatrix}\Big) = \chi(d) .$$
The \emph{cusps} of~$\Gamma$ are~$\Gamma$-equivalence classes of elements~$\bfR\cup\{\infty\}$ that are parabolic, \textit{i.e.} each of them is the unique fixed point of some element of~$\Gamma$. They correspond to cusps on a fundamental domain. A set of representatives is given by rational numbers~$u/w$ where~$1\leq w$, $w|q$, $(u, w)=1$ and $u$ is determined~$\mod{(w, q/w)}$.

For each cusp~$\ca$, let~$\Gamma_\ca$ denote the stabilizer of~$\ca$ for the action of~$\Gamma$. A \emph{scaling matrix} is an element~$\sigma_\ca\in SL_2(\bfR)$ such that~$\sigma_{\ca}\infty = \ca$ and
$$  \Big\{\sigma_\ca\begin{pmatrix} 1&b\\0&1\end{pmatrix}\sigma_\ca^{-1}, b\in\bfZ\Big\}  = \Gamma_\ca .$$
Whenever~$\ca = u/w$ with~$u\neq 0$, $(u, w)=1$ and $w|q$, one can choose
\begin{equation}
\sigma_\ca = \begin{pmatrix} \ca\sqrt{[q, w^2]} & 0 \\ \sqrt{[q, w^2]} & (\ca\sqrt{[q, w^2]})^{-1} \end{pmatrix}.\label{eq:scaling-matrix}
\end{equation}
A cusp~$\ca$ is said to be \emph{singular} if~$\chi(\gamma)=1$ for any~$\gamma\in\Gamma_\ca$. When~$\ca = u/w$ with~$u$ and~$w$ as above, then this merely means that~$\chi$ has conductor dividing~$q/(w, q/w)$. The point at infinity is always a singular cusp, with stabilizer
$$ \Gamma_\infty = \Big\{\begin{pmatrix}1&\ast\\0&1\end{pmatrix}\Big\} .$$

For any pair of singular cusps~$\ca, \cb$ and any associated scaling matrices~$\sigma_\ca, \sigma_\cb$, define the set of moduli
$$ \cC(\ca, \cb) := \Big\{c\in\bfR_+^* :\ \exists a, b, d \in\bfR, \begin{pmatrix} a & b \\ c & d \end{pmatrix} \in \sigma_\ca^{-1}\Gamma\sigma_\cb \Big\} .$$
This set actually only depends on~$\ca$ and~$\cb$. For all~$c\in\cC(\ca, \cb)$, let~$\cD_{\ca\cb}(c)$ be the set of real numbers~$d$ with~$0<d\leq c$, such that
$$  \begin{pmatrix} a & b \\ c & d \end{pmatrix} \in \sigma_\ca^{-1}\Gamma\sigma_\cb $$
for some~$a, b\in\bfR$. For each such~$d$,~$a$ is uniquely determined~$\mod{c}$.

For any integers~$m, n\geq 0$, and any~$c\in\cC(\ca, \cb)$, the Kloosterman sum is defined as (see formula~(3.13) and Chapter~4 of~\cite{Iwaniec-Book-Topics})
$$ S_{\sigma_\ca\sigma_\cb}(m, n ; c) = \sum_{d\in\cD_{\ca\cb}(c)}\bar{\chi}(\sigma_\ca \big(\begin{smallmatrix}a & \ast \\ c & d\end{smallmatrix}\big) \sigma_\cb^{-1})\e\Big(\frac{am+dn}{c}\Big) $$
where~$\big(\begin{smallmatrix}a & \ast \\ c & d\end{smallmatrix}\big)$ denotes any matrix~$\gamma$ having lower row~$(c, d)$ such that~$\sigma_\ca\gamma\sigma_{\cb}^{-1}\in \Gamma$. This is well-defined by our hypotheses that~$\ca$ and~$\cb$ are singular. This definition allows for a great deal of generality. We quote from~\cite[section 2.1]{DI} the remark that the Kloosterman sums essentially depend only on the cusps~$\ca$, $\cb$, and only mildly on the scaling matrices~$\sigma_\ca$ and~$\sigma_\cb$, in the following sense. If~$\tilde{\ca}$ and~$\tilde{\cb}$ are two cusps respectively~$\Gamma$-equivalent to~$\ca$ and~$\cb$, with respective scaling matrices~$\tilde{\sigma_\ca}$ and~$\tilde{\sigma_\cb}$, then there exist real numbers~$t_1$ and~$t_2$, independent of~$m$ or~$n$, such that
$$ S_{\sigma_\ca\sigma_\cb}(m, n ; c) = \e(mt_1 + nt_2)S_{\tilde{\sigma_\ca}\tilde{\sigma_\cb}}(m, n ; c) .$$
Moreover, the converse fact holds, that for any reals~$t_1, t_2$, any cusps~$\ca$ and~$\cb$, and any scaling matrices~$\sigma_\ca$ and~$\sigma_\cb$, there exist scaling matrices~$\tilde{\sigma_\ca}$ and~$\tilde{\sigma_\cb}$ associated to~$\ca$ and~$\cb$ such that the equality above holds. This rather simple fact is of tremendous help because all of the results obtained through the Kuznetsov formula are uniform with respect to the scaling matrices, so that one can encode oscillating factors depending on~$m$ and~$n$ at no cost (it is crucial for separation of variables). Whenever the context is clear enough, we write
$$ S_{\ca\cb}(m, n ; c) $$
without reference to the scaling matrices.

\bigskip

The first example is~$\ca=\cb=\infty$ and~$\sigma_\ca=\sigma_\cb= 1$. Then~$\cC(\infty, \infty)=q\bfN$ and
\begin{equation}
S_{\infty\infty}(m, n ; c) = S_{\chi}(m, n ; c) = \ssum{d\mod{c}^\times}\bar{\chi}(d)\e\Big(\frac{\bar{d}m+dn}{c}\Big) \qquad (c\in q\bfN)\label{eq:somme-kloo-infini}
\end{equation}
is the usual (twisted) Kloosterman sum. Here and in the rest of the paper, we write~$\mod{c}^\times$ to mean a primitive residue class~$\mod{c}$.

The next example that we need is the case~$\ca=\cb$. The following is an extension of~\cite[Lemma~2.5]{DI}. It is proven in an identical way, so we omit the details.
\begin{lemme}
Assume~$\ca = u/w$ is a cusp with~$(u, w)=1$, $w|q$ and~$u\neq 0$. Assume that~$\ca$ is singular. Choose the scaling matrix as in~\eqref{eq:scaling-matrix}. Then~$\cC(\ca, \ca)=\frac{q}{(w, q/w)}\bfN $, and if~$c=\gamma q/(w, q/w)$ for some~$\gamma\in\bfN$,
\begin{equation}
S_{\ca\ca}(m, n ; c) = \e\Big((w, q/w)\frac{m-n}{uq}\Big)\underset{\delta\mod{c}}{{\sum}^\ast} \bar{\chi}\Big(\alpha +u\frac{\alpha\delta-1}{\gamma}\Big) \e\Big(\frac{m\alpha + n\delta}{c}\Big),\label{eq:expr-Saa}
\end{equation}
where, in the sum~${\sum}^\ast$, $\delta$ runs over the solutions~$\mod{c}$ of
\begin{equation}
(\delta, \gamma q/w)=1, \quad (\gamma+u\delta, w)=1, \quad \delta(\gamma+u\delta)\equiv u\mod{(w, q/w)},\label{eq:Saa-cond-delta}
\end{equation}
and~$\alpha$ is determined~$\mod{c}$ by the equations
\begin{equation}
\alpha\delta \equiv 1\mod{\gamma q/w}, \quad \alpha\equiv \gamma'\bar{u'}+u'\bar{(\gamma'+u'\delta)}\mod{w\gamma'}\label{eq:Saa-cond-alpha}
\end{equation}
where~$\gamma' = \gamma/(\gamma, u)$ and~$u'=u/(\gamma, u)$.
\end{lemme}
The sums~$S_{\ca\ca}(m, n ; c)$ are expressed by means of the Chinese remainder theorem (twisted multiplicativity) as a product of similar sums for moduli~$c$ that are prime powers. When~$c=p^\nu$ and~$\nu\geq 2$, a bound is obtained by means of elementary methods as in~\cite[Section~12.3]{IK}. When~$c$ is prime, the Weil bound (\textit{cf.}~\cite[Theorem~9.3]{KL}) from algebraic geometry can be used. In the general case, one obtains
\begin{lemme}\label{sec:Saa-Weil}
For all~$c\in\cC(\ca, \ca)$, $m, n\in\bfZ$, we have
\[ S_{\ca\ca}(m, n ; c) \ll (m, n, c)^{1/2}\tau(c)^{O(1)}(cq_0)^{1/2} \]
where~$q_0$ is the modulus of~$\chi$.
\end{lemme}

Finally, we consider as in~\cite{DI} the following family of Kloosterman sums, which will be of particular interest to us.
\begin{lemme}\label{lemme:kloo-chi}
Assume that the level~$q$ is of the shape~$rs$, with~$q_0|r$, where~$q_0$ is the modulus of~$\chi$, and~$(r, s)=1$. The two cusps~$\infty$ and~$1/s$ are singular. Choose the scaling matrices
$$ \sigma_\infty = \Id, \qquad \sigma_{1/s} = \begin{pmatrix}\sqrt{r} & 0 \\ s\sqrt{r} & 1/\sqrt{r}\end{pmatrix}. $$
Then~$\cC(\infty, 1/s) = \{cs\sqrt{r}, c\in\bfN, (c, r)=1\}$, and for~$(c, r)=1$, we have
$$ S_{\infty,1/s}(m, n ; cs\sqrt{r}) = \bar{\chi}(c)\e\Big(\frac{n\bar{s}}{r}\Big)S(m\bar{r}, n ; sc) $$
where~$S(\ldots)$ in the right-hand side is the usual (untwisted) Kloosterman sum.
\end{lemme}
The main feature here is the presence of the character \emph{outside} the Kloosterman sums, as opposed to~\eqref{eq:somme-kloo-infini}. It is proven in a way identical to~\cite[page~240]{DI}, keeping track of an additional factor~$\bar{\chi}(D)$ in the summand.

\subsubsection{Normalization}

In order to state the Kuznetsov formula, we first fix the normalization. We largely borrow from~\cite{BHM}. We also refer to~\cite[Section~4]{DFI} for useful explanations on Maa{\ss} forms, and to~\cite{Proskurin} for a discussion in the case of general multiplier systems.

For each integer~$k>0$ with~$k\equiv \kappa\mod{2}$, we fix a basis~$\cB_k(q, \chi)$ of holomorphic cusp forms. It is taken orthonormal with respect to the weight~$k$ Petersson inner product:
$$ \langle f, g\rangle_k = \int_{\Gamma\bsl\bfH} y^k f(z)\bar{g(z)}\frac{\dd x\dd y}{y^2} \qquad (z=x+iy). $$

We let~$\cB(q, \chi)$ denote a basis of the space of Maa{\ss} cusp forms. In particular they are functions on~$\bfH$, are automorphic of weight~$\kappa\in\{0, 1\}$ (meaning they satisfy~\cite[formula~(5)]{Proskurin}), are square-integrable on a fundamental domain and vanish at the cusps (note that when~$\kappa=1$, they do not induce a function on~$\Gamma\bsl\bfH$). They are eigenfunctions of the $L^2$-extension of the Laplace-Beltrami operator
$$ \Delta = y^2\Big(\frac{\partial^2}{\partial x^2} + \frac{\partial^2}{\partial y^2}\Big) - i\kappa y\frac{\partial}{\partial x}. $$
This operator has pure point spectrum on the~$L^2$-space of cusp forms. For~$f\in\cB(q, \chi)$, we write~$(\Delta + s(1-s))f = 0$ with~$s=\tfrac12 + it_f$ and~$t_f\in\bfR\cup[-i/2, i/2]$. The~$(t_f)_{f\in\cB(q, \chi)}$ form a countable sequence with no limit point in~$\bfC$ (in particular, there are only finitely many~$t_f\in i\bfR$). We choose the basis~$\cB(q, \chi)$ orthonormal with respect to the weight zero Petersson inner product. Let
\begin{equation}
\theta := \sup_{f\in\cB(q, \chi)}|\Im t_f|,\label{eq:def-theta}
\end{equation}
then Selberg's eigenvalue conjecture is that~$\theta=0$ \textit{i.e.}~$t_f\in\bfR$ for all~$f\in\cB(q, \chi)$. Selberg proved that~$\theta\leq1/4$ (see~\cite[Theorem~4]{DI}), and the current best known result is~$\theta \leq 7/64$, due to Kim and Sarnak~\cite{Kim} (see~\cite{Sarnak} for useful explanations on this topic).

The decomposition of the space of square-integrable, weight~$\kappa$ automorphic forms on~$\bfH$ with respect to eigenspaces of the Laplacian contains the Eisenstein spectrum~$\cE(q, \chi)$ which turns out to be the orthogonal complement to the space of Maa{\ss} forms. It can be described explicitely by means of the Eisenstein series~$E_{\ca}(z ; \tfrac12+it)$ where~$\ca$ runs through singular cusps, and~$t\in\bfR$. Care must be taken because these are not square-integrable; see~\cite[Section~15.4]{IK} for more explanations.

\medskip

Let~$j(g, z) := cz+d$ where~$g=\begin{psmallmatrix} \ast & \ast \\ c & d\end{psmallmatrix}\in SL_2(\bfR)$. We write the Fourier expansion of~$f\in\cB_k(q, \chi)$ around a singular cusp~$\ca$ with associated scaling matrix~$\sigma_\ca$ as
\begin{equation}
f(\sigma_\ca z) j(\sigma_\ca, z)^{-k} = \sum_{n\geq 1} \rho_{f\ca}(n)(4\pi n)^{k/2}\e(nz).\label{eq:dvpt-fourier-holo}
\end{equation}
We write the Fourier expansion of~$f\in\cB(q, \chi)$ around the cusp~$\ca$ as
$$ f(\sigma_\ca z) \e^{-i\kappa \arg j(\sigma_\ca, z)} = \sum_{n\neq 0} \rho_{f\ca}(n) W_{\frac{|n|}{n}\frac{\kappa}{2}, it_f}(4\pi|n|y)\e(nx) $$
where the Whittaker function is defined as in~\cite[formula~(1.26)]{Iwaniec-Book-Spectral}. Finally, for every singular cusp~$\cc$, we write the Fourier expansion around the cusp~$\ca$ of the Eisenstein series associated with the cusp~$\cc$ as
$$ E_\cc(\sigma_\ca z, \tfrac12 + it)  \e^{-i\kappa \arg j(\sigma_\ca, z)} = c_{1,\cc}(t)y^{1/2+it} + c_{2,\cc}(t)y^{1/2-it} + \sum_{n\neq 0}\rho_{\cc\ca}(n, t)W_{\frac{|n|}{n}\frac{\kappa}{2}, it}(4\pi|n|y)\e(nx). $$

\subsubsection{The Kuznetsov formula}

Let~$\phi:\bfR_+\to\bfC$ be of class~$\cC^\infty$ and satisfy
\begin{equation}
\phi(0)=\phi'(0)=0, \qquad \phi^{(j)}(x)\ll (1+x)^{-2-\eta} \quad (0\leq j\leq 3)\label{eq:cond-phi-kuz}
\end{equation}
for some~$\eta>0$. In practice, the function~$\phi$ will be~$\cC^\infty$ with compact support in~$\bfR_+^*$. We define the integral transforms
\begin{align*}
\dphi(k) :=\ & 4i^k\int_0^\infty J_{k-1}(x)\phi(x)\frac{\dd x}x, \numberthis\label{eq:def-dphi} \\
\tphi(t) :=\ & \frac{2\pi it^\kappa}{\sinh(\pi t)}\int_0^\infty (J_{2it}(x)-(-1)^\kappa J_{-2it}(x))\phi(x)\frac{\dd x}x, \numberthis\label{eq:def-tphi} \\
\cphi(t) :=\ & 8i^{-\kappa} \cosh(\pi t)\int_0^\infty   K_{2it}(x)\phi(x)\frac{\dd x}x \numberthis\label{eq:def-cphi}
\end{align*}
where we refer to~\cite[Appendix~B.4]{Iwaniec-Book-Spectral} for the definitions and estimates on the Bessel functions. We have borrowed the normalization from~\cite{BHM-Burgess}, apart from a constant factor~$4$ which we included in the transforms. The sizes of these transforms is controlled by the following Lemma (we need only consider~$|t|\leq 1/4$ in the second estimate, by Selberg's theorem that~$\theta\leq 1/4$). The bounds we state are not the best that can be obtained, but they will be sufficient for our purpose.
\begin{lemme}[]\label{lemme-estim-Hankel}
If~$\phi$ is supported on~$x\asymp X$ with~$\|\phi^{(j)}\|_\infty\ll X^{-j}$ for~$0\leq j\leq 4$, then
\begin{align*}
|\dphi(t)| + \frac{|\tphi(t)|}{1+|t|^\kappa} + |\cphi(t)|&\ \ll \frac{1 + |\log X|}{1+X}\min\Big\{1, \Big(\frac{1+X^{3/2}}{1+|t|^3}\Big)\Big\} \qquad (t\in\bfR), \numberthis\label{eq:hankel-reg} \\
|\tphi(t)| + |\cphi(t)|&\ \ll \frac{1 + X^{-2|t|}}{1+X} \qquad\qquad\qquad (t\in [-i/4, i/4]).
\end{align*}
\end{lemme}
\begin{proof}
These bounds are analogues of \cite[Lemma~7.1]{DI} and \cite[Lemma~2.1]{BHM-Burgess}. Taking into account the factor~$t^\kappa$ in front of~$\tphi(t)$, the arguments there are easily adapted. The only non-trivial fact to check is that the decaying factor in~\eqref{eq:hankel-reg} only requires the hypotheses~$\|\phi^{(j)}\|_\infty\ll X^{-j}$ for~$j\leq 4$. This is seen by reproducing the proof of~\cite[Lemma~2.1]{BHM-Burgess} with the choices~$j=1$ and~$i=2$.
\end{proof}
Recall that~$\kappa$ is defined by~$\chi(-1) = (-1)^\kappa$. We are ready to state the Kuznetsov formula for Dirichlet multiplier system and general cusps.
\begin{lemme}\label{lemme-Kuz}
Let~$\ca$ and~$\cb$ be two singular cusps with associated scaling matrices~$\sigma_{\ca}$ and~$\sigma_{\cb}$, and~$\phi:\bfR_+\to\bfC$ as in~\eqref{eq:cond-phi-kuz}. Let~$m, n\in\bfN$. Then
\begin{align*}
\sum_{c\in\cC(\ca, \cb)}\frac1c S_{\ca\cb}(m, n ; c)\phi\Big(\frac{4\pi\sqrt{mn}}{c}\Big) =\ & \cH + \cE + \cM, \numberthis\label{eq:Kuz-plus} \\
\sum_{c\in\cC(\ca, \cb)}\frac1c S_{\ca\cb}(m, -n ; c)\phi\Big(\frac{4\pi\sqrt{mn}}{c}\Big) =\ & \cE' + \cM', \numberthis\label{eq:Kuz-moins}
\end{align*}
where~$\cH$,~$\cE$,~$\cM$ (``holomorphic'', ``Eisenstein'', ``Maa\ss{}'') are defined by
\begin{align*}
\cH :=\ &\ssum{k>\kappa\\k\equiv \kappa\mod{2}}\sum_{f\in\cB_k(q, \chi)}\dphi(k)\Gamma(k)\sqrt{mn}\overline{\rho_{f\ca}(m)}\rho_{f\cb}(n), \numberthis\label{Kuz-def-holo}\\
\cE :=\ &\sum_{\cc\text{ sing.}}\frac1{4\pi}\int_{-\infty}^\infty \tphi(t)\frac{\sqrt{mn}}{\cosh(\pi t)}\overline{\rho_{\cc\ca}(m, t)}\rho_{\cc\cb}(n, t)\dd t, \numberthis\label{Kuz-def-eisen} \\
\cM :=\ &\sum_{f\in\cB(q, \chi)}\tphi(t_f)\frac{\sqrt{mn}}{\cosh(\pi t_f)}\overline{\rho_{f\ca}(m)}\rho_{f\cb}(n), \numberthis\label{Kuz-def-maass} \\
\cE' :=\ &\sum_{\cc\text{ sing.}}\frac1{4\pi}\int_{-\infty}^\infty \cphi(t)\frac{\sqrt{mn}}{\cosh(\pi t)}\overline{\rho_{\cc\ca}(m, t)}\rho_{\cc\cb}(-n, t)\dd t, \numberthis\label{Kuz-def-eisen-m} \\
\cM' :=\ &\sum_{f\in\cB(q, \chi)}\cphi(t_f)\frac{\sqrt{mn}}{\cosh(\pi t_f)}\overline{\rho_{f\ca}(m)}\rho_{f\cb}(-n). \numberthis\label{Kuz-def-maass-m}
\end{align*}

\begin{proof}
For~$\ca=\cb=\infty$, the formula~\eqref{eq:Kuz-plus} and the case~$\kappa=0$ of~\eqref{eq:Kuz-moins} can be found in Section~2.1.4 of~\cite{BHM}. The extension to general cusps~$\ca$, $\cb$ is straightforward.

The case~$\kappa=1$ of~\eqref{eq:Kuz-moins} was obtained by B. Topacogullari (private communication). We restrict here to mentionning that it can be proved by reproducing the computations of page~251 of~\cite{DI} and Section~5 of~\cite{DFI}\footnote{Note that in the expression for~$h_p(t)$ given on page~518 of~\cite{DFI}, the term~$\Gamma(1-\tfrac k2 - ir)$ should read~$\Gamma(1-\tfrac k2+ir)$.}.
\end{proof}

\end{lemme}
The right-hand side of the Kuznetsov formula (the so-called spectral side) naturally splits into two contributions. The \emph{regular spectrum} consists in~$\cH$, $\cE$ and the contribution to~$\cM$ of those~$f\in\cB(q, \chi)$ with~$t_f\in\bfR$ ; the conjecturally inexistant \emph{exceptional spectrum} is the contribution to~$\cM$ of those~$f$ with~$t_f\in i\bfR^*$ (similarly with~$\cE'$ and~$\cM'$). The technical reason for this distinction is the growth properties of the integral transforms. Indeed, when~$X$ is small (\textit{i.e.} when the average over the moduli of the Kloosterman sums is long, since~$X\asymp \sqrt{mn}/c$), we see from Lemma~\ref{lemme-estim-Hankel} that while~$\dphi(t)$, $\tphi(t)$ and~$\cphi(t)$ are essentially bounded for~$t\in\bfR$, $\tphi(it)$ is roughly of size~$X^{-2|t|}$ when~$t\in[-1/2, 1/2]$.

We remark that in contrast with other works (\textit{e.g.}~\cite{BM2}), we do not make use of Atkin-Lehner's newform theory, nor of Hecke theory. In fact, we do not use any information about the Fourier coefficients~$\rho_{f\ca}(n)$ and~$\rho_{\cc\ca}(n, t)$ other than the fact that Kuznetsov's formula holds, so the reader unfamiliar with the subject can go through the following sections without knowing what they are. The main feature of the Kuznetsov formula which is used is the decay properties of the integral transforms~\eqref{eq:def-dphi}-\eqref{eq:def-cphi}, and the fact that it separates the variables~$m$ and~$n$ in a way that combines very nicely with the Cauchy--Schwarz inequality.

\subsection{Large sieve inequalities}\label{sec:DI-gc}
\subsubsection[Quadratic forms with Saa]{Quadratic forms with $S\sb{\ca\ca}$}

Given~$N\in\bfN$, $\vth\in\bfR_+^*$, $\lambda\geq 0$, a sequence~$(b_n)$ of complex numbers, a singular cusp~$\ca$ and~$c\in\cC(\ca, \ca)$, let
\[ B_{\ca}(\lambda, \vth ; c, N) := \sum_{N < m, n\leq 2N}b_m\bar{b_n}\e^{-\lambda \sqrt{mn}} S_{\ca\ca}(m, n, c)\e\Big(\frac{2\sqrt{mn}}{c}\vth\Big). \]
We also define
$$ \|b_N\|_2 := \Big(\sum_{N<n\leq2N} |b_n|^2\Big)^{1/2} .$$
The following extends~\cite[Proposition~3]{DI}.
\begin{lemme}[{\cite[Proposition~3]{DI}}]\label{lemme:Ba}
We have
\begin{align*}
|B_{\ca}(\lambda, \vth ; c, N)| &\ \leq \tau(c)^{O(1)}(q_0 c)^{1/2}N\|b_N\|^2, \numberthis\label{eq:Ba-q0} \\
|B_{\ca}(\lambda, \vth ; c, N)| &\ \ll (c + N + \sqrt{\vth c N})\|b_N\|^2, \\
|B_{\ca}(\lambda, \vth ; c, N)| &\ \ll_\ee \vth^{-1/2}c^{1/2}N^{1/2+\ee}\|b_N\|^2 \numberthis\label{eq:Ba-3}
\end{align*}
where the last bounds holds for~$\vth<2$ and~$c<N$.
\end{lemme}
\begin{proof}
Suppose~$\lambda=0$. The first bound is an immediate consequence of~Lemma~\ref{sec:Saa-Weil}. For the second bound, the proof given in~\cite[page~256]{DI} transposes without any change: after expanding out the sum~$S_{\ca\ca}(\dotsc)$, one uses the triangle inequality with the effect that the factors involving~$\chi$ are trivially bounded. For the last bound, the proof is adapted with the following modification: the Cauchy--Schwarz inequality yields
\begin{equation}
|B_\ca(0, \vth; c, N)|^2 \leq \|b_N\|_2^2 \ssum{N<m_1, m_2\leq 2N \\ \delta_1, \delta_2} b_{m_1} \bar{b_{m_2}}\bar{\chi(r_1)}\chi(r_2) \e\Big(\frac{m_1\delta_1-m_2\delta_2}{c}\Big)\sum_n f(n)\label{eq:Ba-apres-cauchy}
\end{equation}
where~$f(n)$ is defined as in~\cite[page~256]{DI}, $\delta_1$ and~$\delta_2$ run over residue classes modulo~$c$ satisfying~\eqref{eq:Saa-cond-delta}, and~$r_j := \delta_j^{-1} + u(\alpha_j\delta_j-1)/\gamma$ for~$j\in\{1, 2\}$, where~$\alpha_j$ is determined by~\eqref{eq:Saa-cond-alpha}. The only difference is the presence of the~$\chi$ factors. Upon using Poisson summation on the sum~$\sum_n f(n)$, the argument is split in two cases according to whether~$\alpha_1\equiv\alpha_2\mod{c}$ or not. If~$\alpha_1\not\equiv \alpha_2\mod{c}$, then one uses the triangle inequality on~\eqref{eq:Ba-apres-cauchy} so that the~$\chi$ factors do not intervene. If on the contrary~$\alpha_1\equiv\alpha_2\mod{c}$, then we deduce from~\eqref{eq:Saa-cond-alpha} that also~$\delta_1\equiv\delta_2\mod{c}$. The~$\chi$ factors cancel out and the rest of the argument carries through without change.

The case of arbitrary~$\lambda\geq 0$ reduces to the case~$\lambda=0$ by Mellin inversion
$$ \e^{-y} = \frac1{2\pi i}\int_{1-i\infty}^{1+i\infty} \Gamma(s) y^{-s} \dd s = 1 + \frac1{2\pi i}\int_{-1/2-i\infty}^{-1/2+i\infty} \Gamma(s)y^{-s}\dd s $$
at~$y=\lambda\sqrt{mn}$, using the first expression when~$\lambda N\geq 1$ and the second otherwise.
\end{proof}

\subsubsection{Large sieve inequalities for the regular spectrum}

We proceed to state the following large sieve-type inequalities, which extend~\cite[Proposition~4]{DI}.
\begin{prop}\label{prop:grand-crible}
Let~$(a_n)$ be a sequence of complex numbers, and~$\ca$ a singular cusp for the group~$\Gamma_0(q)$ and Dirichlet multiplier~$\chi\mod{q_0}$. Suppose~$T\geq 1$ and $N\geq 1/2$. Then each of the three quantities
\begin{align*}
& \ssum{\kappa<k\leq T\\k\equiv \kappa\mod{2}} \Gamma(k)\sum_{f\in\cB_k(q, \chi)}\Big|\sum_{N<n\leq2N} a_n \sqrt{n} \rho_{f\ca}(n)\Big|^2,   \numberthis\label{eq:GC-holo} \\
& \ssum{f\in\cB(q, \chi)\\ |t_f|\leq T} \frac{(1+|t_f|)^{\pm \kappa}}{\cosh(\pi t_f)}\Big|\sum_{N<n\leq 2N}a_n\sqrt{n} \rho_{f\ca}(\pm n)\Big|^2,   \numberthis\label{eq:GC-maass} \\
& \sum_{\cc\text{ sing.}} \int_{-T}^T \frac{(1+|t|)^{\pm \kappa} }{\cosh(\pi t)}\Big|\sum_{N<n\leq 2N}a_n\sqrt{n} \rho_{\cc\ca}(\pm n, t)\Big|^2\dd t,   \numberthis\label{eq:GC-eisen}
\end{align*}
is majorized by
\[ O_\ee\big((T^2 + q_0^{1/2}\mu(\ca)N^{1+\ee})\|a_N\|_2^2\big). \]
Here, if~$\ca$ is equivalent to~$u/w$ with~$w|q$ and~$(u, w)=1$, then~$\mu(\ca) := (w, q/w)/q$.
\end{prop}
\begin{proof}
These formulas are deduced from two summation formulas, namely the Petersson formula~\cite[Theorem~3.6]{Iwaniec-Book-Topics}
\begin{equation}\label{eq:Petersson}
\begin{aligned}
\bfUn_{m=n} + & 2\pi i^{-k}\sum_{c\in\cC(\ca, \ca)} \frac{1}{c}S_{\ca\ca}(m, n ; c)J_{k-1} \Big(\frac{4\pi\sqrt{mn}}{c}\Big) \\
&= 4\Gamma(k-1)\sqrt{mn} \sum_{f\in\cB_k(q, \chi)}\bar{\rho_{f\ca}(m)}\rho_{f\ca}(n),
\end{aligned}
\end{equation}
valid for~$k>1$, $k\equiv \kappa\mod{2}$, and a ``pre-Kuznetsov'' formula~\cite[Proposition~5.2]{DFI} which, for general cusps, is
\begin{equation}
\begin{aligned}
& \frac{\big|\Gamma(1\mp\tfrac \kappa 2 +ir)|^2}{4\pi^2}\Big\{\bfUn_{m=n} +  \sum_{c\in \cC(\ca, \ca)} \frac1c S_{\ca\ca}(\pm m, \pm n ; c) I_{\pm}\Big(\frac{4\pi \sqrt{mn}}{c}\Big)\Big\} \\
= \sum_{f\in \cB(q, \chi)}& \frac{\sqrt{mn}}{\cosh(\pi t_f)} H(t_f, r)\bar{\rho_{f\ca}(\pm m)}\rho_{f\ca}(\pm n) + \frac1{4\pi}\sum_{\cc\text{ sing.}}\int_{-\infty}^\infty \frac{\sqrt{mn}}{\cosh(\pi t)} H(t, r)\bar{\rho_{\cc\ca}(\pm m)}\rho_{\cc\ca}(\pm n) \dd t
\end{aligned}\label{eq:PreKuz}
\end{equation}
for all real~$r$ and positive integers~$m$, $n$. Here,
$$ H(t, r) = \frac{\cosh(\pi t)\cosh(\pi r)}{\cosh(\pi(t-r))\cosh(\pi(t+r))} \qquad (r, t\in\bfC, r\not\in \pm t+i/2+i\bfZ), $$
$$ I_{\pm}(x) = -2x\int_{-i}^i (-i v)^{\pm\kappa - 1} K_{2ir}(v x)\dd v \qquad (x>0), $$
where~$v$ varies on the half-circle~$|v|=1$, $\Re(v)\geq 0$ counter-clockwise. Note that by the complement formula
\begin{equation}\label{eq:complements}
\big|\Gamma(1 - \tfrac \epsilon 2 + ir)\big|^2 = \frac{\pi}{\cosh(\pi r)} \times \begin{cases} 1, & \epsilon = 1, \\ \tfrac 14 + r^2, & \epsilon = -1. \end{cases}
\end{equation}

Given the formulas~\eqref{eq:Petersson} and~\eqref{eq:PreKuz}, the arguments in~\cite[pages~258-261]{DI} are adapted as follows. When~$\kappa=0$, the details are strictly identical. Consider the case~$\kappa=1$ of~\eqref{eq:GC-holo}. We multiply both sides of~\eqref{eq:Petersson} by~$(k-1)\e^{-(k-1)/T}\bar{a_m}a_n$ and sum over~$k$, $m$ and~$n$. The analogue of the function~$E_K(x)$ defined in~\cite[page~258]{DI} is (up to a constant factor) the function
$$ E_T(x) = \sum_{\ell\geq 1} (-1)^\ell 2\ell \e^{-2\ell/T} J_{2\ell}(x) = -\frac12\sinh\big(\frac1T\big) \int_0^1 \frac{u^2xJ_1(ux)\dd u}{(\cosh(1/T)^2 - u^2)^{3/2}}, $$
as can be seen by reproducing the computations in~\cite[page~316]{Iwaniec-MV}\footnote{There is a slight convergence issue in the Fourier integral for~$yJ_1(y)$, which is resolved by changing~$b=\cosh(1/T)$ to~$b+i\ee$, $\ee>0$ and letting~$\ee\to0$.}. We then write (see~\cite[eq. 8.411.3, page~912]{GR})
$$ J_1(y) = \frac2\pi \int_0^{\pi/2} \cos \tau \sin(y\cos \tau) \dd\tau, $$
split the integral at~$\Delta\in(0, \pi/2]$ and deduce the bound~\eqref{eq:GC-holo} by following the steps in~\cite[page~259]{DI}.

Consider next the case~$\kappa=1$ and positive sign of~\eqref{eq:GC-maass} and~\eqref{eq:GC-eisen}. We multiply both sides of~\eqref{eq:PreKuz} by~$r^2\cosh(\pi r)\e^{-(r/T)^2}\bar{a_m}a_n$, integrate over~$r\in\bfR$ and sum over~$m$ and~$n$. The analogue of the function~$\Phi(x)$ of~\cite[page~260]{DI} is the function
$$ \Phi_+(x) = \int_{-\infty}^\infty r^2 \e^{-(r/T)^2} \int_{-i}^i K_{2ir}(xv) \dd v \dd r. $$
We use the expression~$K_{2ir}(y) = \int_0^\infty \e^{-y\cosh \xi} \cos(2r\xi) \dd\xi$ ($\Re y>0$). For~$x>0$, we obtain by integrations by parts
\begin{align*}
\Phi_+(x) = &\ -i\sqrt{\pi} T^3 \int_0^\infty \e^{-(\xi T)^2} \xi \tanh \xi \Big\{\cos(x\cosh \xi) - \frac12 \int_{-1}^1 \cos(x\vth \cosh \xi)\dd\vth\Big\} \dd\xi \\
= &\ i\sqrt{\pi} \frac{T^3}x \int_0^\infty \e^{-(\xi T)^2} (1 - 2(\xi T)^2) \sinh(x\cosh \xi)\frac{\dd\xi}{\cosh \xi},
\end{align*}
and from there, the bounds~\eqref{eq:GC-maass} and~\eqref{eq:GC-eisen} are obtained by reproducing the computations of~\cite[page~261]{DI}.

Consider finally the case of negative sign in~\eqref{eq:GC-maass} and~\eqref{eq:Petersson}. We multiply both sides of~\eqref{eq:PreKuz} by~$r^2\cosh(\pi r)/(\tfrac14+r^2)\e^{-(r/T)^2}\bar{a_m}a_n$. The analogue of the function~$\Phi(x)$ of~\cite[page~260]{DI} is now
$$ \Phi_-(x) = \int_{-\infty}^\infty r^2 \e^{-(r/T)^2} \int_{-i}^i K_{2ir}(xv) \frac{\dd v}{v^2} \dd r, $$
and we have by integration by parts
\begin{align*}
\Phi_-(x) = &\ i\sqrt{\pi} T^3 \int_0^\infty \e^{-(\xi T)^2} \xi \tanh \xi \Big\{ \cos(x \cosh\xi) - \frac1{2i} \int_{-i}^i \frac{\e^{-vx \cosh \xi}}{v^2}\dd v\Big\} \dd\xi \\
= &\ -i\sqrt{\pi} \frac{T^3}x \int_0^\infty \e^{-(\xi T)^2} (1 - 2(\xi T)^2) \Big\{ \sinh(x\cosh\xi) + \frac1i \int_{-i}^i \frac{\e^{-xv\cosh\xi}}{v^3}\dd v\Big\} \frac{\dd\xi}{\cosh\xi}.
\end{align*}
From there, it is straightforward to reproduce the computations of~\cite[page~261]{DI} using the bounds of Lemma~\ref{lemme:Ba}.

\end{proof}

\subsubsection{Weighted large sieve inequalities for the exceptional spectrum}

The objects we would like to bound now are of the shape
$$ E_{q, \ca}(Y, (a_n)) := \ssum{f\in \cB(q, \chi)\\t_f\in i\bfR} Y^{2|t_f|} \Big|\sum_{N<n\leq 2N} a_n n^{1/2}\rho_{f\ca}(n)\Big|^2 $$
where~$Y\geq 1$ is to be taken as large as possible while still keeping this quantity comparable to the bounds~$(1+\mu(\ca)N)\sum_n |a_n|^2$ coming from Proposition~\ref{prop:grand-crible}. The following is the analogue of~\cite[Theorem~5]{DI}.
\begin{lemme}\label{lemme:DIth5}
Assume that the situation is as in Proposition~\ref{prop:grand-crible}. Then for any~$Y\geq 1$,
$$ E_{q, \ca}(Y, (a_n)) \ll_\ee \big(1 + (\mu(\ca)NY)^{1/2}\big)\big(1 + (q_0\mu(\ca)N)^{1/2+\ee}\big)\|a_N\|_2^2. $$
\end{lemme}
The important aspect in this bound is that it is as good as those coming from the regular spectrum ({\it i.e.} the upper bound in Proposition~\ref{prop:grand-crible}) in the situation when~$\mu(\ca)=1/q$ (which will typically be the case), $N<q$ and~$Y\leq q/N$. Note also that the previous bound holds for any individual~$q$.
\begin{proof}
The arguments in~\cite[section~8.1, pages~270-271]{DI} transpose identically.\footnote{Note that in the last display of the proof~\cite[page 271]{DI}, $L(Y)$ should read~$L(Y^{-1})$.}.
\end{proof}

The next step is to produce an analogue of~\cite[Theorem~6]{DI}, which is concerned with the situation when an average over~$q$ is done. Deshouillers and Iwaniec make use of the very nice idea that with the choice~$\ca=\infty$ for each~$q$, the roles of~$q$ and~$c$ can be swapped in the Kuznetsov formula. Through an induction process, this enhances significantly the bounds obtained. This switching technique is specific to the choice~$\ca=\infty$ for all~$q$, with scaling matrices independent of~$q$.
\begin{lemme}\label{lemme:DIth6}
Assume the situation is as previously. Recall that~$\chi$ has modulus~$q_0\geq 1$. Then for all~$Y\geq 1$ and~$Q\geq q_0$,
$$ \ssum{q\leq Q \\ q_0|q} E_{q, \infty}(Y, (a_n)) \ll_\ee (QN)^\ee(Qq_0^{-1} + N + NY^{1/2})\|a_N\|_2^2, $$
where the scaling matrices are chosen independently of~$q$.
\end{lemme}
Note that now, in the situation when~$N\leq Q$, the parameter~$Y$ is allowed to be as large as~$(Q/N)^2$ while still yielding a bound of same quality as the regular spectrum. The final situation is the special case when~$(a_n)$ is the characteristic sequence of an interval of integers. Then Deshouillers and Iwaniec are able to provide an even stronger bound~\cite[Theorem~7]{DI}, by enhancing the initial step in the induction.
\begin{lemme}\label{lemme:DIth7}
Assume that the situation is as in Lemma~\ref{lemme:DIth6}. Assume moreover that~$(a_n)_{N<n\leq 2N}$ is the characteristic sequence of an interval of integers. Then
$$ \ssum{q\leq Q \\ q_0|q} E_{q, \infty}(Y, (a_n)) \ll_\ee (QN)^\ee(Qq_0^{-1} + N + (NY)^{1/2})N. $$
\end{lemme}
In the situation when~$N\leq Q$, the parameter~$Y$ can then be taken as large as~$Q^2/N$ while still yielding an acceptable bound.

We now proceed to justify Lemmas~\ref{lemme:DIth6} and~\ref{lemme:DIth7}. For the rest of this section, we rename~$q$ into~$q_0q$, so that now~$q$ runs over intervals. The object of interest is
\[ S(Q, Y, N, s) := \ssum{Q<q\leq 16Q} \ssum{f\in\cB(q_0q, \chi)\\t_f \in i\bfR} Y^{2|t_f|}\Big|\sum_{N<n\leq2N}a_n n^{s+1/2} \rho_{f\infty}(n)\Big|^2. \]
\begin{lemme}
Let~$N, Y, Q\geq 1$ and a sequence~$(a_n)$ be given. Then
\begin{equation}
\begin{aligned}
S(Q, Y, N, 0) \ll_\ee \int_{-\infty}^\infty &S\big(\frac{\pi NY}{q_0Q}, Y, N, it\big)\frac{\dd t}{t^4+1} \\
&+ (QYN)^\ee\big(Q + \frac{N}{q_0^{1/2}} + \frac{NY}{q_0^{1/2}Q}\big)\|a_N\|_2^2.\label{eq:SQNY-rec}
\end{aligned}
\end{equation}
Moreover, if~$(a_n)$ is the characteristic sequence of an interval, then
\begin{equation}
S(Q, Y, N, 0) \ll_\ee (NY)^\ee (Q + N + Y)N\label{eq:SQNY-lisse}
\end{equation}
\end{lemme}
\begin{proof}[of~\eqref{eq:SQNY-rec}]
The arguments in~\cite[pages~272-273]{DI} are adapted with minimal effort; however we take the opportunity to justify more precisely one of the claims made there. Fix a smooth function~$\Phi:\bfR\to[0, 1]$ supported inside~$[1/2, 5/2]$ and majorizing~$\bfUn_{[1, 2]}$. Letting~$g(q)=\Phi(q/Q)$ and~$\phi(x)=\Phi(Yx)$ (these kind of homotheties of~$\Phi$ we refer to as \emph{test functions}) we have
$$ S(Q, Y, N, 0) \ll |\cS_1|, $$
$$ \cS_1 := \sum_{q\geq 1}g(q)\ssum{f\in \cB(q_0q, \chi)\\t_f\in i\bfR} \frac{\tphi(t_f)}{\cosh(\pi t_f)}\Big|\sum_n a_n n^{1/2} \rho_{f\infty}(n)\Big|^2. $$
This is seen by approximating the Bessel function in the definition of~$\tphi$ by its first order term, as in~\cite[formula~(8.1)]{DI}. Opening the squares in~$\cS_1$ and applying the Kuznetsov formula and the large sieve estimates (Lemma~\ref{lemme-Kuz} and Proposition~\ref{prop:grand-crible}), one gets
$$ \cS_1 = \sum_{m, n}\bar{a_m}a_n\cS_2(m, n) + O_\ee\big((QNY)^\ee \big(Q + \frac{N}{q_0^{1/2}}\big)\sum_n|a_n|^2\big), $$
$$ \cS_2(m, n) := \sum_{q, c\geq 1}\frac{g(q)}{q_0qc}\phi\big(\frac{4\pi\sqrt{mn}}{q_0qc}\big)S_{\infty\infty}(m, n ; qc), $$
Letting~$h(x) = h_{m, n, c}(x) = \phi(x) g\big(\frac{4\pi\sqrt{mn}}{q_0cx}\big)$, one applies the Kuznetsov formula for the group~$\Gamma_0(q_0c)$ (which requires that the scaling matrices be independent of~$q$) and obtains
$$ \cS_1 \ll |\cS_3| + O_\ee\big((QNY)^\ee \big(Q + \frac{N}{q_0^{1/2}} + \frac{NY}{q_0^{1/2}Q}\big)\sum_n|a_n|^2\big), $$
$$ \cS_3 := \sum_{m, n}\bar{a_m}a_n\sum_{C < c\leq 16 C}\ssum{f\in \cB(q_0c, \chi)\\t_f\in i\bfR}\frac{{\tilde h}(t_f)}{\cosh(\pi t_f)}\sqrt{mn}\bar{\rho_{f\infty}(m)}\rho_{f\infty}(n) .$$
Note that~$h(t_f) = h_{m, n, c}(t_f) = 0$ unless~$C<c\leq 16C$, where~$C = \pi NY / (q_0 Q)$. Let
$$ \cK_{\kappa, t}(x) := \frac{2\pi i t^\kappa}{\sinh(\pi t)}\big(J_{2it}(x) - (-1)^\kappa J_{-2it}(x)\big), $$
and~${\breve g}(s) := \int_0^\infty g(x) x^{s-1}\dd x$ be the Mellin transform of~$g$. Then
$$ {\tilde h}(t) = \frac1{2\pi}\int_{-\infty}^\infty {\breve g}(i\tau)\big(\frac{q_0c}{4\pi\sqrt{mn}}\big)^{i\tau}\int_0^\infty \cK_{\kappa, t}(x) x^{i\tau}\phi(x) \dd x\dd \tau .$$
Inserting into the definition of~$\cS_3$ and using the triangle inequality, we obtain
\begin{align*}
\cS_3 \ll \int_{-\infty}^\infty |{\breve g}(i\tau)| \sum_{C<c\leq 16C}\ssum{f\in \cB(q_0c, \chi)\\t_f\in i\bfR} &
\Big|\sum_m a_m m^{(1+i\tau)/2}\rho_{f\infty}(m)\Big|
\Big|\sum_n a_n n^{(1-i\tau)/2}\rho_{f\infty}(m)\Big| \times \\
&\times \Big|\int_0^\infty \cK_{\kappa, t}(x) x^{i\tau}\phi(x)\dd x\Big|\dd \tau.
\end{align*}
From there, the arguments in~\cite[page~273]{DI} apply and yield
$$ \Big|\int_0^\infty \cK_{\kappa, t}(x) x^{i\tau}\phi(x)\dd x\Big| \ll_\ee Y^{2|t_f|} + Y^{\ee} $$
from which the claimed bound follows in the same way as~\cite[page~273]{DI}.
\end{proof}
\begin{proof}[of~\eqref{eq:SQNY-lisse}]
Assume that~$(a_n)_{N<n\leq 2N}$ is the characteristic sequence of the integers inside~$(N, N_1]$ for some~$N_1\leq 2N$. We proceed as in~\cite[page~276]{DI}. By applying the Kuznetsov formula and the large sieve inequalities, one obtains
\begin{align*}
S(Q, N, Y, 0) \ll_\ee \sum_{Q<q\leq 16Q}\sum_{c\geq 1} & \frac1{q_0qc}\Big|\sum_{N\leq m, n \leq
N_1}\phi\big(\frac{4\pi\sqrt{mn}}{q_0qc}\big)S_{\infty\infty}(m, n ; qq_0c)\Big| \\
& + \big(Q+\frac{N^{1+\ee}}{q_0^{1/2}}\big)N
\end{align*}
for a test function~$\phi$ supported inside~$[1/(2Y), 5/(2Y)]$. Here one may restrict summation to~$C/4<c\leq 8C$ for~$C := \pi NY/(q_0Q)$. Let~$k := q_0qc$. The first term above is majorized by
$$ T := (q_0QC)^{-1+\ee} \ssum{k \asymp q_0QC\\q_0|k} \Big|\sum_{N<m, n\leq N_1}\phi\big(\frac{4\pi\sqrt{mn}}{k}\big) S_{\infty\infty}(m, n ; k) \Big|. $$
Let~$\phi(x) = \tfrac1{2\pi}\int_{-\infty}^\infty{\breve \phi}(it)x^{-it}\dd t$, where the Mellin transform~${\breve \phi}(s) = \int_0^\infty \phi(x) x^{s-1}\dd x$ satisfies~${\breve\phi}(it) \ll (1+t^4)^{-1}$, so that (after reinterpreting~$t$ by~$2t$)
$$ T \ll (q_0QC)^{-1+\ee} \int_{-\infty}^\infty \frac1{t^4+1} \ssum{k \asymp q_0QC\\q_0|k} \Big|\sum_{N<m, n\leq N_1}(mn)^{-it} \e((m-n)\vth)S_{\chi}(m, n ; k) \Big| \dd t $$
for some~$\vth\in[0, 1)$ (depending on the scaling matrix). By~$m^{-it} = N_1^{-it}+it\int_m^{N_1}u^{-it-1}\dd u$, we obtain
$$ T \ll (q_0QC)^{-1+\ee} \sup_{N\leq N', M' \leq N_1} \ssum{k \asymp q_0QC\\q_0|k} U_1(k, M', N'), $$
$$ U_1(M', N') :=  \Big|\ssum{m\leq M' \\ n\leq N'} \e((m-n)\vth) S_{\chi}(m, n ; k) \Big| .$$
Opening the summation in~$S_\chi$, we have
$$ U_1(k, M', N') \leq U_2(k, M', N') := \ssum{\delta\mod{k}^\times} \Big| \sum_{m\leq M'} \e\Big(\frac{\delta m}{k}+m\vth\Big)\Big| \Big|\sum_{n\leq N'}\e\Big(\frac{\bar{\delta}n}{k}-n\vth\Big)\Big| .$$
It is crucial to note that the quantity on the RHS also exists for~$k$ not multiple of~$q_0$, so trivially
$$ T \ll (q_0QC)^{-1+\ee} \sup_{N\leq M', N' \leq N_1} \ssum{k \asymp q_0QC} U_2(k, M', N'), $$
From there on, the calculations in~\cite[page~276]{DI} apply and yield, in the notation of~\cite[Lemma~8.2]{DI},
$$ U_2(k, M', N') \ll \sum_{m, n\in \bfZ} {\hat f}_{M'}(m)\e(m\vth) {\hat f}_{N'}(n)\e(-n\vth) S(m, n ; k). $$
The proof of Theorem~14 of~\cite{DI} follows through, and yields for all~$K\geq 1$,
$$ \sum_{k\leq K} U_2(k, M', N')  \ll_\ee (KMN)^\ee K(K+MN). $$
Taking~$K\asymp q_0QC$, we conclude that
$$ T \ll_\ee (q_0QC)^{\ee} (q_0QC + N^2). $$
The rest of the arguments in~\cite[page~277]{DI} applies and yields
$$ S(Q, N, Y, 0) \ll_\ee (NY)^\ee (Q + N + Y)N $$
as claimed.
\end{proof}

\begin{proof}[of Lemmas~\ref{lemme:DIth6} and~\ref{lemme:DIth7}]
In addition to the recurrence relation~\eqref{eq:SQNY-rec}, we have the properties
\begin{align*}
S(Q, Y, N, 0) \leq&\ (Y/Z)^{1/2} S(Q, Z, N, 0) \qquad (1\leq Z\leq Y), \\
S(Q, 1, N, 0) \ll_\ee &\ (QN)^\ee \big(Q + \frac{N}{q_0^{1/2}}\big)\|a_N\|_2^2.
\end{align*}
The second one follows from Proposition~\ref{prop:grand-crible}. Having these at hand, the induction arguments in~\cite[page 274]{DI} and~\cite[page 277]{DI} are easily reproduced. It is useful to notice that~$q_0$ appears only with negative powers in the error terms, and that its presence in the denominator of~$\pi NY/(q_0Q)$ in~\eqref{eq:SQNY-rec} is beneficial for the induction.
\end{proof}

\begin{remarque}
The previous three lemmas used only Selberg's theorem that~$\theta\leq 1/4$ (recall the definition~\eqref{eq:def-theta}). One could make the bounds explicit in terms of~$\theta$ and thus benefit from recent progress towards the Ramanujan-Selberg conjecture. It is straightforward to check that Lemmas~\ref{lemme:DIth5},~\ref{lemme:DIth6} and~\ref{lemme:DIth7} hold with the right-hand sides replaced by
$$ (1 + (\mu(\ca)NY)^{2\theta})(1 + q_0^{1/2}(\mu(\ca)N)^{1-2\theta+\ee})\|a_N\|_2^2, $$
$$ (QN)^\ee (Qq_0^{-1} + N + Y^{2\theta}N^{4\theta}Q^{1-4\theta})\|a_N\|_2^2, $$
$$ (QN)^\ee (Qq_0^{-1} + N + Y^{2\theta}N^{2\theta}Q^{1-4\theta})N $$
respectively (compare with~\cite[Proposition~16.10]{IK}). We refrain from doing so because it would not impact the applications considered here.
\end{remarque}

\subsection{Proof of Theorem~\ref{thm:quintilin}}
\subsubsection{Estimates for sums of generalized Kloosterman sums}

We begin by the following statement regarding the generalized Kloosterman sums~$S_{\ca, \cb}(m, n ; c)$. For the sake of simplifying the presentation of the bound obtained, we discard powers of the modulus~$q$. This does not have consequences on our applications.

\begin{prop}\label{prop:estim-genkloo}
Let the real numbers~$M, N, R, S\geq 1$, $X>0$ and the integer~$q\geq 1$ be given, let~$\chi$ be a character modulo~$q$, let~$\phi$ be a smooth function supported on the interval~$[X, 2X]$ such that~$\|\phi^{(j)}\|_\infty\ll X^{-j}$ for~$0\leq j \leq 4$, and let~$(a_m)$ and~$(b_{n,r,s})$ be sequences of complex numbers supported on~$M<m\leq 2M$, $N<n\leq 2N$, $R<r\leq 2R$ and~$S<s\leq 2S$. Assume that~$(a_m)$ is the characteristic sequence of an interval of integers. Then
\begin{equation}\begin{aligned}
 \ssum{m, n, r, s \\ (s, rq)=1} a_m& b_{n,r,s} \sum_{c\in\cC(\infty, 1/s)} \frac1c \phi\Big(\frac{4\pi\sqrt{mn}}{c}\Big)S_{\infty, 1/s}(m, \pm n ; c) \\ \ll_\ee &\ (q(X+X^{-1})RSMN)^\ee\big\{L_{\reg} + L_{\exc}\big\},  \label{eq:somme-genkloo-bnrs}
\end{aligned}
\end{equation}
$$ L_{\reg} := \big( 1+ X + \sqrt{\frac{N}{RS}}\big)\big( 1 + X + \sqrt{\frac{M}{RS}}\big)\frac{\sqrt{RS}}{1+X}\sqrt{M} \|b_{N,R,S}\|_2, $$
$$ L_{\exc} := \big(1 + \sqrt{\frac{N}{RS}}\big) \sqrt{\frac{1+X^{-1}}{RS}}\big(\frac{MN}{RS+N}\big)^{1/4}\frac{\sqrt{RS}}{1+X}\sqrt{M}\|b_{N,R,S}\|_2. $$
where the Kloosterman sum is defined with respect to the congruence group~$\Gamma(qrs)$ with multiplier induced by~$\chi$, with scaling matrices~$\sigma_\infty$ and~$\sigma_{1/s}$ that are both independent of~$m$ and~$n$, with~$\sigma_\infty$ independent of~$r$ and~$s$ as well.
\end{prop}
\begin{remarque}
If~$(a_m)$ is not the characteristic sequence of an interval, then the bound~\eqref{eq:somme-genkloo-bnrs} still holds with~$L_\exc$ replaced by~$M^{1/4}L_\exc$ (see~\cite[Theorems~10 and~11]{DI}).
\end{remarque}
\begin{proof}
This estimate is deduced from Proposition~\ref{prop:grand-crible} and Lemmas~\ref{lemme:DIth5} and~\ref{lemme:DIth7} by following the computations of Section~9.1 of~\cite{DI}. It is useful to notice that the bounds of Lemmas~\ref{lemme:DIth5},~\ref{lemme:DIth7} and Proposition~\ref{prop:grand-crible} (for~$\ca \in\{\infty, 1/s\}$) decrease with~$q_0$.

\end{proof}

\subsubsection{Estimates for the complete Kloosterman sums twisted by a character}

We now justify the transition from Proposition~\ref{prop:estim-genkloo} to an estimate for twisted sums of usual Kloosterman sums~$S(m, n ; c)$.
\begin{prop}\label{prop:bd-twisted}
Let the real numbers~$M, N, R, S, C\geq 1$, and the integer~$q\geq 1$ be given, let~$\chi$ be a character modulo~$q$, let~$g$ be a smooth function supported on~$[C, 2C]\times[M, 2M]\times (\bfR_+^*)^3$ such that
\begin{equation}
\frac{\partial^{\nu_0+\nu_1+\nu_2+\nu_3+\nu_4} g}{\partial c^{\nu_0} \partial m^{\nu_1} \partial n^{\nu_2} \partial r^{\nu_3} \partial s^{\nu_4}}(c, m, n, r, s) \ll C^{-\nu_0}M^{-\nu_1}N^{-\nu_2}R^{-\nu_3}S^{-\nu_4} \label{eq:hypo-der-g}
\end{equation}
for~$0\leq \nu_j\leq 12$. Let~$(b_{n,r,s})$ be a sequence of complex numbers supported on~$N<n\leq 2N$, $R<r\leq 2R$ and~$S<s\leq 2S$. Then uniformly in~$t \in[0, 1)$,
\begin{equation}\begin{aligned}
\ssum{c, m, n, r, s \\ (sc, rq)=1} b_{n,r,s} &\ \bar{\chi}(c)g(c, m, n, r, s)\e(mt)S(n\bar{r}, \pm m\bar{q} ; sc) \\
&\ll_\ee (CRSMNq)^\ee q^{3/2} \big\{K_{\reg} + K_{\exc}\big\}\sqrt{M}\|b_{N,R,S}\|_2 ,  \label{eq:somme-complete-bnrs}
\end{aligned}
\end{equation}
\begin{align*}
&K_\reg^2 := RS\frac{(C^2S^2R + MN + C^2SN)(C^2S^2R + MN + C^2SM)}{C^2S^2R + MN}, \\
&K_\exc^2 := C^3S^2\sqrt{R(N+ RS)}.
\end{align*}
\end{prop}

\begin{proof} We present the proof in the case where there is a~$+$ sign in the Kloosterman sums. The complementary case is similar. The main issue is separation of variables, as explained in~\cite[page~269]{DI}. The nuisance is mainly notational. We write
$$ g(c, m, n, r, s) = \int_{\bfR^4} \tfrac1{sc\sqrt{rq}}G(\tfrac{4\pi\sqrt{mn}}{sc\sqrt{rq}}, \xi_1, \xi_2, \xi_3, \xi_4)\e(-m\xi_1 - n\xi_2 - r\xi_3 - s\xi_4) \prod_{j=1}^4 \dd \xi_j, $$
by Fourier inversion, where for all~$(x, \xi_1, \dotsc, \xi_4) \in \bfR_+^* \times \bfR^4$, 
$$ G(x, \xi_1, \dotsc, \xi_4) := \int_{\bfR^4}g_*(x, x_1, \dotsc, x_4)\e(x_1\xi_1 + \dotsb + x_4\xi_4) \prod_{j=1}^4 \dd x_j ,$$
$$ g_*(x, x_1, \dotsc, x_4) := \frac{4\pi\sqrt{x_1x_2}}{x} g\big(\frac{4\pi\sqrt{x_1x_2}}{xx_4\sqrt{x_3q}}, x_1, \dotsc, x_4\big) .$$
By integration by parts, for any non-negative integers~$(\ell, \ell_1, \dotsc, \ell_4)$ with~$\ell\leq 4$ and~$\ell_j\leq 2$,
\begin{align*}
\frac{\partial^\ell G}{\partial x^\ell}(x, \xi_1, \dotsc, \xi_4) = \prod_j (2\pi i \xi_j)^{-\ell_j} \int_{\bfR^4} &\ \big(\frac{\partial^{\ell+\ell_1+\dotsb+\ell_4}}{\partial x^\ell \partial x_1^{\ell_1} \dotsb \partial x_4^{\ell_4}}g_* (x, x_1, \dotsc, x_4)\big) \times \\
&\times \e(x_1\xi_1 + \dotsb + x_4\xi_4) \prod_{j} \dd x_j
\end{align*}
assuming~$\xi_j\neq 0$ if~$\ell_j>0$. The derivatives are estimated using~\eqref{eq:hypo-der-g}. Choose~$\ell_1=0$ or~$\ell_1=2$ according to whether~$|\xi_1|M<1$ or not, and similarly for~$\ell_2$, $\ell_3$, $\ell_4$. Then
$$ \frac{\partial^\ell G}{\partial x^\ell}(x, \xi_1, \dotsc, \xi_4) \ll \frac{MNRS^2C\sqrt{qR} (\sqrt{MN}/(CS\sqrt{qR}))^{-\ell}}{(1+(\xi_1 M)^2)(1+(\xi_2 N)^2)(1+(\xi_3 R)^2)(1+(\xi_4 S)^2)}. $$
We abbreviate further
$$ \phi(x) = \phi_{\xi_1, \dotsc, \xi_4}(x) := \frac{(1+(\xi_1 M)^2)(1+(\xi_2 N)^2)(1+(\xi_3 R)^2)(1+(\xi_4 S)^2)}{MNRS^2C\sqrt{qR}}G(x, \xi_1, \dotsc, \xi_4). $$
This function satisfies the hypotheses of Proposition~\ref{prop:estim-genkloo}, with\footnote{Note that in~\cite[page~278]{DI}, some occurences of~$X$ should read~$X^{-1}$.}~$X = \sqrt{MN}/(CS\sqrt{qR})$, uniformly in~$\xi_j$. Define
$$ {\tilde b}_{n, r, s} := b_{n, r, s} \e(-n(\xi_2 + \bar{s}/(rq)) - r\xi_3 - s\xi_4\big). $$
Finally, by Lemma~\ref{lemme:kloo-chi} with the scaling matrices
$$ \sigma_{\infty} = \begin{pmatrix} 1 & \xi_1-t \\ 0 & 1 \end{pmatrix}, \qquad \sigma_{1/s} = \begin{pmatrix} \sqrt{rq} & 0 \\ s\sqrt{rq} & 1/\sqrt{rq} \end{pmatrix}, $$
we have
$$ \bar{\chi}(c) S(n\bar{r}, m\bar{q} ; sc)\e(m(t-\xi_1)+n\bar{s}/(rq)) = S_{\infty, 1/s}(m, n ; sc\sqrt{rq}). $$
Proposition~\ref{prop:estim-genkloo} can therefore be applied and yields
\begin{align*}
\ssum{m, n, r, s\\(s, rq)=1}{\tilde b}_{n, r, s}\sum_{(c, rq)=1} \frac1{cs\sqrt{rq}}&\ \phi\big(\frac{4\pi\sqrt{mn}}{sc\sqrt{rq}}\big)S_{\infty, 1/s}(m, n ; sc\sqrt{rq}) \\
&\ll_\ee \frac{q^{3/2}(CMNRS)^{\ee}}{CS\sqrt{qR}}(W_\reg + W_\exc)\sqrt{M}\|b_{N, R, S}\|_2,
\end{align*}
with
\begin{align*}
&W_\reg^2 = RS\frac{(C^2S^2R + MN + C^2SN)(C^2S^2R + MN + C^2SM)}{C^2S^2R + MN}, \\
&W_\exc^2 = C^3S^2\sqrt{R(N+ RS)}.
\end{align*}
From the definitions of~$\phi$ and~$G$, we deduce the claimed bound.
\end{proof}

\subsubsection{Bounds for incomplete Kloosterman sums}\label{sec:completion-sums}
\label{sec:bounds-incompl-kloos}

In this section, we prove Theorem~\ref{thm:quintilin}. As a first reduction, we remark that it suffices to prove the result when the sequence~$b_{n, r, s}$ is supported on~$N<n\leq 2N$, by summing dyadically over~$N$ and by concavity of~$\sqrt{\cdot}$ (losing a factor~$(\log N)^{1/2}$ in the process). Secondly, we let~$s_0\mod{q}^\times$ be fixed and assume without loss of generality that
\begin{equation}
b_{n,r,s}=0 \text{ unless } s\equiv s_0\mod{q}.\label{eq:hypo-s0}
\end{equation}
We will recover the full bound~\eqref{eq:majo-quintilin} by summing over~$s_0\mod{q}^\times$ (losing a factor~$q^{1/2}$ in the process by concavity). Let
\begin{equation}
\ddot{g}(c, m, n, r, s) := \int_{-\infty}^\infty g(c, \xi, n, r, s)\e(\xi m)\dd\xi.\label{eq:def-ddg}
\end{equation}
By Poisson summation, we write the left-hand side of~\eqref{eq:majo-quintilin} as
\begin{align*}
&\ \ssum{c, n, r, s\\(qr,sc)=1\\c\equiv c_0\mod{q}}b_{n,r,s}\ssum{\delta\mod{sc}\\(\delta,sc)=1}
\e\Big(n\frac{\bar{r\delta}}{sc}\Big)\ssum{d\equiv\delta\mod{sc}\\d\equiv d_0 \mod{q}}g(c,d,n,r,s)\\
= &\ \ssum{c, n, r, s\\(qr,sc)=1\\c\equiv c_0\mod{q}}\frac{b_{n,r,s}}{scq}\sum_{(\delta,sc)=1}
\e\Big(n\frac{\bar{r\delta}}{sc}\Big)\sum_m\ddot{g}(c, m/sqc, n, r, s)
\e\Big(-\frac{md_0\bar{sc}}q-\frac{m\delta\bar{q}}{sc}\Big) \\
= &\ \ssum{c, m, n, r, s\\(qr,sc)=1\\c\equiv c_0\mod{q}}\frac{b_{n,r,s}}{scq}\ddot{g}(c, m/scq, n, r, s)
\e\Big(\frac{-md_0\bar{s_0c_0}}q\Big)S(n\bar{r}, -m\bar{q} ; sc) \numberthis \label{eq:ql-apres-ipp}
\end{align*}
where~$S(\ldots)$ is the usual Kloosterman. Let~$M>0$ be a parameter. We write~\eqref{eq:ql-apres-ipp} as~$\cA_0 + \cA_\infty + \cB$, where~$\cA_0$ is the contribution of~$m=0$, $\cA_\infty$ is the contribution of indices~$m$ such that~$|m|>M$, and~$\cB$ is the contribution of indices~$m$ with~$0<|m|\leq M$. By the bound for Ramanujan sums~\cite[formula~(3.5)]{IK},
\[ \cA_0 \ll \frac1q\ssum{c, n, r, s\\(qr, sc)=1\\c\equiv c_0\mod{q}}\frac{|b_{n,r,s}|}{sc}|\ddot{g}(c, 0, n, r, s)|(n, sc)
\ll q^{-2}(\log S)^2 D\{NR/S\}^{1/2}\|b_{N,R,S}\|_2 .\]

By repeated integration by parts in the integral~\eqref{eq:def-ddg}, for fixed~$k\geq 1$ and~$m\neq 0$ we have
\[ \ddot{g}(c, m/(scq), n, r, s) \ll_k D^{1-k(1-\ee_0)}\Big(\frac{scq}{|m|}\Big)^k .\]
Taking~$k\asymp 1/\ee_0$, we have that there is a choice of~$M \asymp (SCqD)^{\ee+O(\ee_0)}SCq/D$ such that the bound
$$ \ddot{g}(c, m/(scq), n, r, s)\ll_{\ee} 1/m^2 \qquad (|m|>M) $$
holds. Bounding trivially the Kloosterman sum in~\eqref{eq:ql-apres-ipp} by~$sc$, we obtain
\begin{equation}
\cA_\infty \ll_\ee (SCqD)^{\ee+O(\ee_0)} q^{-2} D\{NR/S\}^{1/2}\|b_{N,R,S}\|_2 \label{eq:borne-cA}
\end{equation}
which is also acceptable (if~$\ee_0$ is small enough, the factor~$q^{-2+\ee+O(\ee_0)}$ is bounded).

There remains to bound~$\cB$; we may assume that~$M\geq 1$ for otherwise~$\cB$ is void. By dyadic decomposition,
$$ |\cB| \ll \log 2M \sup_{1/2\leq M_1\leq M} |\cB(M_1)|, $$
where
$$ \cB(M_1) := \ssum{c, m, n, r, s\\(qr,sc)=1\\M_1<|m|\leq 2M_1\\c\equiv c_0\mod{q}}\frac{b_{n,r,s}}{scq}\ddot{g}(c, m/scq, n, r, s)
\e\Big(\frac{-md_0\bar{s_0c_0}}q\Big)S(n\bar{r}, -m\bar{q} ; sc) .$$
We insert the definition of~$\ddot{g}$ after having changed variables~$\xi\to\xi scq/m$, to obtain
$$ |\cB(M_1)| \ll \frac{DM_1}{SCq} \sup_{\xi \asymp DM_1/(SCq)} |\cB'(M_1, \xi)|, $$
where
\begin{equation}
\cB'(M_1, \xi) := \ssum{c, m, n, r, s\\(qr, sc)=1\\M_1<|m|\leq 2M_1\\c\equiv c_0\mod{q}}\frac{b_{n, r, s}}m g(c, \xi scq/m, n, r, s)\e\Big(\frac{-md_0\bar{s_0c_0}}q\Big)S(n\bar{r}, -m\bar{q} ; sc).\label{eq:ql-apres-completion}
\end{equation}
By orthogonality of multiplicative characters, we have
$$ \cB'(M_1, \xi) = \frac1{M_1\vphi(q)}\sum_{\chi\mod{q}}\chi(c_0)\cS(M_1, \xi, \chi), $$
where
$$ \cS(M_1, \xi, \chi) := \ssum{r, s\\(s, qr)=1}\ssum{m, n\\|m|\asymp M_1} b_{n, r, s}\sum_{(c, rq)=1}\bar{\chi(c)}g_1(c, m, n, r, s)\e\Big(\frac{-md_0\bar{s_0c_0}}q\Big)S(n\bar{r}, -m\bar{q} ; sc), $$
$$ g_1(c, m, n, r, s) := M_1m^{-1}g(c, \xi scq/m, n, r, s) . $$
Proposition~\ref{prop:bd-twisted} can be applied to the sums~$\cS(M_1, \xi, \chi)$, at the cost of enlarging the bound by a factor~$O((CDNRS)^{60\ee_0})$ in order for the derivative conditions~\eqref{eq:hypo-der-g} to be satisfied. We obtain
$$ \cS(M_1, \xi, \chi) \ll_\ee q^{3/2}(CDNRSq)^{\ee+O(\ee_0)}\big\{L_\reg + L_\exc\big\}\sqrt{M_1}\|b_{N, R, S}\|_2, $$
$$ L_\reg^2 := RS\frac{(C^2S^2R + M_1N + C^2SN)(C^2S^2R + M_1N + C^2SM_1)}{C^2S^2R + M_1N}, $$
$$ L_\exc^2 := C^3S^2\sqrt{R(N + RS)}. $$
From there, computations identical to~\cite[page~282]{DI} allow to bound
$$ L_\reg^2 \ll RS\big(C^2S^2R + M_1N + \frac{C^2M_1N}{R} + C^2S(M_1+N)\big). $$
We deduce successively
$$ |\cB(M_1)|  \ll_\ee (CDNRSq)^{\ee+O(\ee_0)} \frac{q D\sqrt{M_1}}{SC} L^*(M_1) \|b_{N, R, S}\|_2, $$
$$ L^*(M_1)^2 := RS(C^2S^2R + M_1N + C^2M_1N/R + C^2S(M_1+N)) + C^3S^2\sqrt{R(N + RS)}, $$
and finally
\begin{equation}
\cB \ll_\ee (CDNRSq)^{\ee+O(\ee_0)} q \cK,\label{eq:borne-cB}
\end{equation}
$$ \cK^2 := qCS(N + RS)(C + RD) + C^2DS\sqrt{(N+RS)R}.  $$
Grouping our two bounds~\eqref{eq:borne-cA} and~\eqref{eq:borne-cB}, and summing over~$s_0\mod{q}^\times$, we obtain the claimed result.

\section{Convolutions in arithmetic progressions}\label{sec:convo-bin}

In this section, we proceed with an instance of the dispersion method, for convolutions of two sequences one of which is supported in~$[x^\eta, x^{1/3-\eta}]$ for some~$\eta>0$. This extends~\cite[Section~13]{BFI} and~\cite[Section~V]{FouvryTit}.

Given a parameter~$R\geq 1$, an integer~$q\geq 1$ and a residue class~$n\mod{q}$, we let
$$ \cX_q(R) := \{\chi\mod{q},\ \cond(\chi)\leq R\}, $$
and
\begin{equation}
\begin{aligned}
\cu_R(n ; q) :=&\ \bfUn_{n\equiv 1\mod{q}} - \frac1{\vphi(q)}\ssum{\chi\in\cX_q(R)} \chi(n) \\
=&\ \frac1{\vphi(q)}\ssum{\chi\mod{q}\\\cond(\chi)>R} \chi(n). 
\end{aligned}
\label{eq:def-cu-R}
\end{equation}
Note that this vanishes when~$q\leq R$ or~$(n, q)>1$. We have the trivial bound
\begin{equation}
|\cu_R(n ; q)| \ll \bfUn_{n\equiv 1\mod{q}} + \frac{R\tau(q)}{\vphi(q)}.\label{eq:tv-cuR}
\end{equation}
It will also be sometimes useful to write
\begin{equation}
\cu_R(n ; q) = \Big(\bfUn_{n\equiv 1\mod{q}} - \frac{\bfUn_{(n, q)=1}}{\vphi(q)}\Big) -  \frac1{\vphi(q)}\ssum{\chi\mod{q} \\ 1<\cond(\chi)\leq R} \chi(n).\label{eq:expr-cuR-cu1}
\end{equation}

\begin{theoreme}\label{thm:distrib-convo}
Let~$M$,~$N$,~$Q$,~$R\geq 1$ and~$\eta$ be given, with~$x:=MN$ and~$x^{1/4}\leq Q$. Then there exists~$\delta$ depending at most on~$\eta$ such that the following holds. Let two sequences~$(\alpha_m)$, $(\beta_n)$ supported in~$n\in(N, 2N]$ and~$m\in(M, 2M]$ be given, which satisfy for some~$A\geq 1$,
\begin{equation}
|\alpha_m|\leq \tau(m)^A, \qquad |\beta_n|\leq \tau(n)^A.\label{eq:cond-taille}
\end{equation}
Let~$a_1, a_2\in\bfZ\smallsetminus\{0\}$, and assume that
\begin{equation}\label{eq:cond-aN}
\left\{
\begin{aligned}
& x^\eta \leq N\leq Q^{2/3-\eta}, \\
& Q \leq x^{1/2+\delta}, \\
& R, |a_1|, |a_2| \leq x^\delta.
\end{aligned}
\right. 
\end{equation}
Then for small enough~$\eta$, we have
\begin{equation}
\ssum{Q<q\leq 2Q \\ (q, a_1a_2)=1} \ssum{n, m \\ (n, a_2)=1} \alpha_m \beta_n \cu_R(mn\bar{a_1}a_2 ; q) \ll x(\log x)^{O(1)}R^{-1}.\label{eq:thm-distrib-convo}
\end{equation}
The implicit constants depend on~$\eta$ and~$A$ at most.
\end{theoreme}

Introducing~$\cu_R(n ; q)$ is technically much more convenient than the usual
\begin{equation}
\cu_1(n ; q) = \bfUn_{n\equiv 1\mod{q}} - \frac{\bfUn_{(n, q)=1}}{\vphi(q)} .\label{eq:def-cu1}
\end{equation}
Indeed, there are no equidistribution assumptions on our sequences in Theorem~\ref{thm:distrib-convo}.

\subsection{Bombieri-Vinogradov range}

Before we embark on the dispersion method we need an estimate which is relevant to values of the moduli less than the threshold~$x^{1/2-\ee}$.

\begin{lemme}\label{lemme:motohashi}
Let~$M, N, R\geq 1$. Let~$x=MN$, and suppose we are given two sequences~$(\alpha_m)$ and~$(\beta_n)$ supported on the integers of~$(M, 2M]$ and~$(N, 2N]$ respectively, satisfying the bounds~\eqref{eq:cond-taille}. Suppose that~$Q\leq x^{1/2}/R$ and~$R\leq Q$. Then
$$ \sum_{Q < q\leq 2Q} \max_{\substack{0<a<q \\ (a, q)=1}}\Big|\sum_{m, n} \alpha_m \beta_n \cu_R(mn\bar{a} ; q) \Big| \ll x(\log x)^{O(1)}(R^{-1} + M^{-1/2} + N^{-1/2}). $$
\end{lemme}
\begin{proof}
See~\cite[Theorem~17.4]{IK}. Only the case~$r>R$ appears in our case.
\end{proof}

\subsection{First reductions}

First we apply two reductions, following Section V.2 of~\cite{FouvryTit} and Section 3 of~\cite{FI}. We replace the sharp cutoff for the sum over~$q$ by a smooth function~$\gamma(q)$ ; and we transfer the squareful part of~$n$ into the number~$a_2$, allowing us to assume that~$n$ is squarefree. The squarefreeness assumption on~$n$ will be useful when dealing with GCD's (in particular in equation~\eqref{eq:mod-1} below). Note also that the statement of Theorem~\ref{thm:distrib-convo} is monotonically weaker as~$\delta\to0$, so that whenever needed, we will take the liberty of reducing the value of~$\delta$ in a way that depends at most on~$\eta$.
\begin{prop}\label{prop:distrib-convo}
Let~$x, M, N, Q, R, \eta$ and the sequences~$(\alpha_m)$ and~$(\beta_n)$ be as in Theorem~\ref{thm:distrib-convo}. Assume that~$(\beta_n)$ is supported on squarefree integers. There exists~$\delta>0$ such that for any smooth function~$\gamma:\bfR_+\to[0, 1]$ with
\begin{equation}
\bfUn_{q\in(Q, 2Q)} \leq \gamma(q) \leq \bfUn_{q\in(Q/2, 3Q/2]}, \label{eq:cond-taille-q}
\end{equation}
and~$\|\gamma^{(j)}\|_\infty\ll_j Q^{-j+B\delta j}$ for some~$B\geq0$ and all fixed~$j\geq 0$, under the conditions~\eqref{eq:cond-aN}, we have
\begin{equation}
\ssum{q \\ (q, a_1a_2)=1}\gamma(q) \ssum{n, m \\ (n, a_2)=1} \alpha_m \beta_n \cu_R(mn\bar{a_1}a_2 ; q) \ll x(\log x)^{O(1)}R^{-1}.\label{eq:majo-distrib-convo}
\end{equation}
The implicit constants depend on~$\eta$, $A$ (in~\eqref{eq:cond-taille}),~$B$ and the function~$\gamma$ at most.
\end{prop}

\begin{proof}[that Proposition~\ref{prop:distrib-convo} implies Theorem~\ref{thm:distrib-convo}]
We replace the sharp cutoff~$Q<q\leq 2Q$ by a smooth weight~$\gamma(q)$ such that
$$ \bfUn_{q\in(Q, 2Q]} \leq \gamma(q) \leq \bfUn_{q\in(Q(1-Q^{-10\delta}), 2Q(1+Q^{-10\delta})]}. $$
We can pick~$\gamma$ such that~$\|\gamma^{(j)}\|_\infty \ll_j Q^{-j+10\delta j}$ for all fixed~$j\geq 0$. The error term in this procedure comes from the contribution of those integers~$q$ at the transition range~$2Q<q\leq 2Q(1+Q^{-10\delta})$ and~$Q(1-Q^{-10\delta})\leq q\leq Q$. It is bounded by the triangle inequality, using our trivial bound~\eqref{eq:tv-cuR} and following the reasonning of~\cite[page~219 and~240]{BFI}, choosing~$Q_0 = x^{10\delta}$ there. We obtain
\begin{equation}
\ssum{q \\ (q, a_1a_2)=1}(\bfUn_{Q<q\leq 2Q} - \gamma(q)) \ssum{n, m \\ (n, a_2)=1} \alpha_m \beta_n \cu_R(mn\bar{a_1}a_2 ; q) \ll
x R (\log x)^{O(1)} Q^{-10\delta} .\label{eq:transfer-smooth-q}
\end{equation}
Given our hypotheses~$R\leq x^\delta$ and~$Q\geq x^{1/4}$, this is an acceptable error term.

Let~$\cK$ denote the set of squareful numbers:
$$ \cK = \{k\in\bfN :\ p|k \Rightarrow p^2|k \}. $$
Factor each integer~$n$ as~$n = n'k$ with~$\mu(n')^2 = 1$,~$(n', k)=1$ and~$k\in\cK$, so that~$k\leq x^{1/3}$ and~$(k, a_2)=1$. Here~$\mu$ is the Möbius function. There are only~$O(K^{1/2})$ squareful numbers up to~$K$~\cite{ErdosSzekeres}, therefore
$$ \ssum{k\geq K \\ k\in\cK} \frac1k \ll K^{-1/2} \qquad (K\geq 1). $$
Proceeding as in~\cite[Section~V.2]{FouvryTit} and using the trivial bound~\eqref{eq:tv-cuR}, we deduce for any~$K\geq 1$,
\begin{equation}
\begin{aligned}
&\ssum{q \\ (q, a_1a_2)=1} \gamma(q) \ssum{n, m \\ (n, a_2)=1} \alpha_m \beta_n \cu_R(mn\bar{a_1}a_2 ; q) \\
= &\ \ssum{k \leq K \\ k\in\cK \\ (k, a_2)=1} \ssum{q \\ (q, a_1a_2)=1} \gamma(q) \ssum{n, m \\ (n, ka_2)=1} \alpha_m \mu(n)^2\beta_{kn} \cu_R(mnk\bar{a_1}a_2 ; q) \\
&\hspace{5em} + O(Rx(\log x)^{O(1)}K^{-1/2}).
\end{aligned}\label{eq:convo-reduc2}
\end{equation}
We are left to analyze, for~$k\in\cK$, $k\leq K$, $(k, a_2)=1$, the sum
$$ \ssum{q \\ (q, a_1a_2)=1} \gamma(q) \ssum{n, m \\ (n, ka_2)=1} \alpha_m \beta_{kn}\mu(n)^2\cu_R(mn\bar{a_1}ka_2 ; q) .$$
Assume~$K\leq x^{4\delta}$. For each fixed~$k$, the sequences~$(\alpha_m)_m$ and~$(k^{-\delta}\mu(n)^2\beta_{kn})_n$ are supported in~$m\in(M, 2M]$ and~$n\in(N/k, 2N/k]$, respectively. We apply Proposition~\ref{prop:distrib-convo} with~$\eta$ replaced by~$\eta/2$,~$N$ replaced by~$N/k$ and~$a_2$ replaced by~$ka_2$ (the factor~$k^{-\delta}$ ensures that the condition~\eqref{eq:cond-taille} holds for~$(k^{-\delta}\mu(n)^2\beta_{kn})_n$). If~$\delta$ is small enough in terms of~$\eta$, we obtain, uniformly for~$k\leq K$,
$$ \ssum{q \\ (q, a_1a_2)=1} \gamma(q) \ssum{n, m \\ (n, ka_2)=1} \alpha_m \beta_{kn}\mu(n)^2 \cu_R(mn\bar{a_1}ka_2 ; q) \ll k^{-1+\delta}x (\log x)^{O(1)}R^{-1}. $$
Note that the sum~$\sum_{k\in\cK} k^{-1+\delta}$ converges. Inserting in~\eqref{eq:convo-reduc2}, we obtain
\begin{align*}
\ssum{q \\ (q, a_1a_2)=1} \gamma(q) \ssum{n, m \\ (n, a_2)=1} \alpha_m \beta_n \cu_R(mn\bar{a_1}a_2 ; q) \ll x(\log x)^{O(1)}(R^{-1} + RK^{-1/2})
\end{align*}
and so we conclude by the choice~$K = R^4$.
\end{proof}

\subsection{Applying the dispersion method}

Let us prove Proposition~\ref{prop:distrib-convo}. Recall that the sequence~$(\beta_n)$ is assumed to be supported on squarefree integers. Let~$\cD$ denote the left-hand side of~\eqref{eq:majo-distrib-convo}. By the triangle inequality
\begin{equation}
|\cD| = \Big|\sum_{(q, a_1a_2)=1 }\gamma(q) \ssum{m, n \\ (n, a_2)=1} \alpha_m \beta_n \cu_R(mn\bar{a_1}a_2 ; q) \Big| \leq \sum_m \Big( |\alpha_m| \Big|\sum_q \sum_n \Big| \Big).\label{eq:predisp}
\end{equation}
Define a smooth and non-negative function~$\alpha(m)$ (not to be confused with our sequence~$\alpha_m$), with~$\alpha(m)\geq 1$ for~$M<m\leq2M$, supported inside~$[M/2, 3M]$ and such that~$\|\alpha^{(j)}\|_\infty\ll_j M^{-j}$. Note that~$|\alpha_m| \leq \tau(m)^A \alpha(m)$ by the hypothesis~\eqref{eq:cond-taille}. Therefore, the Cauchy--Schwarz inequality yields
\begin{align*}
|\cD| {}&\ll \Big(\sum_m \alpha(m)\tau(m)^A\Big)^{1/2} \Big( \sum_m \alpha(m) \Big|\sum_q \sum_n \Big| \Big)^{1/2} \\
{}& \ll (\log x)^{O(1)} M^{1/2} \big(\cS_1 - 2\Re \cS_2 + \cS_3\big)^{1/2}\numberthis\label{eq:dispersion}
\end{align*}
where
$$ \cS_1 = \sum_{(q_1q_2, a_1a_2)=1} \gamma(q_1)\gamma(q_2) \ssum{n_1, n_2\\(n_1n_2, a_2)=1} \beta_{n_1}\overline{\beta_{n_2}} \ssum{mn_1\equiv a_1\bar{a_2}\mod{q_1} \\ mn_2\equiv a_1\bar{a_2} \mod{q_2}} \alpha(m) $$
and~$\cS_2$ and~$\cS_3$ are defined similarly, replacing the sum over~$m$ by
$$ \frac1{\vphi(q_2)}\ssum{\chi_2\in\cX_{q_2}(R)}\chi(n_2\bar{a_1}a_2) \ssum{mn_1\equiv a_1\bar{a_2} \mod{q_1}} \alpha(m)\chi_2(m), $$
$$ \frac1{\vphi(q_1)\vphi(q_2)}\sum_{\chi_1\in\cX_{q_1}(R)}\sum_{\chi_2\in\cX_{q_2}(R)}\chi_1(n_1\bar{a_1}a_2)\bar{\chi_2(n_2\bar{a_1}a_2)} \ssum{(mn_1, q_1)=1 \\ (mn_2, q_2)=1} \alpha(m)\chi_1\bar{\chi_2}(m) $$
respectively. We will prove
\begin{equation}
\cS_1 - 2\Re \cS_2 + \cS_3 = O((\log x)^{O(1)}MN^2R^{-2}).\label{eq:obj-dispersion}
\end{equation}

\subsubsection[Evaluation of S3]{Evaluation of~$\cS_3$}
The term~$\cS_3$ is defined by
\begin{equation}
\cS_3 = \sum_{(q_1q_2, a_1a_2)=1}\frac{\gamma(q_1)\gamma(q_2)}{\vphi(q_1)\vphi(q_2)}\ssum{\chi_1\in\cX_{q_1}(R) \\ \chi_2\in\cX_{q_2}(R)}\ssum{n_1, n_2 \\ (n_j, q_ja_2)=1} \beta_{n_1}\bar{\beta_{n_2}} \sum_{(m, q_1q_2)=1}\alpha(m)\chi_1(mn_1\bar{a_1}a_2)\bar{\chi_2(mn_2\bar{a_1}a_2)}.\label{eq:def-S3}
\end{equation}
Let~$W := [q_1, q_2]$ and~$H := W^{1+\ee}/M$. By Poisson summation (Lemma~\ref{lemme:poisson}),
\begin{align*}
\ \sum_m \alpha(m) \chi_1\bar{\chi_2}(m)\ &= \frac{\hat{\alpha}(0)}{W}\ssum{b\mod{W}^\times}\chi_1\bar{\chi_2}(b)  \\
& + \frac1W\sum_{0<|h|\leq H}\hat{\alpha}\Big(\frac hW\Big)\ssum{b\mod{W}^\times}\e\Big(\frac{bh}{W}\Big)\chi_1\bar{\chi_2}(b) + O_\ee(W^\ee).
\end{align*}
The conductor of~$\chi_1\bar{\chi_2}$ is at most~$R$, so that~\cite[Lemma~3.2]{IK}\footnote{Note that in Lemma~3.2 of~\cite{IK},~$\tau(\chi)$ should read~$\tau(\chi^*)$ and an additional factor~$\chi^*(m/(dm^*))$ should appear in the summand.} yields
$$ \ssum{b\mod{W}^\times}\e\Big(\frac{bh}{W}\Big)\chi_1\bar{\chi_2}(b) \ll R^{1/2}\sum_{d|(h, W)}d. $$
We deduce
$$ \sum_m \alpha(m) \chi_1\bar{\chi_2}(m) = \frac{\hat{\alpha}(0)}{W}\ssum{b\mod{W}^\times}\chi_1\bar{\chi_2}(b) + O_\ee(W^\ee R^{1/2}). $$
The error term is~$O(x^{\delta})$ while the trivial bound is~$M\geq x^{2/3}$. We deduce
$$ \cS_3 = \hat{\alpha}(0)X_3 + O(MN^2x^{-1/2}), $$
where, having changed~$b$ to~$ba_1\bar{a_2}$,
$$ X_3 := \ssum{q_1, q_2 \\(q_1q_2, a_1a_2)=1}\frac{\gamma(q_1)\gamma(q_2)}{[q_1, q_2]\vphi(q_1)\vphi(q_2)}\ssum{\chi_1\in\cX_{q_1}(R) \\ \chi_2\in\cX_{q_2}(R)}\ssum{n_1, n_2 \\ (n_j, q_ja_2)=1} \beta_{n_1}\bar{\beta_{n_2}}\ssum{b\mod{W}^\times} \chi_1(bn_1)\bar{\chi_2(bn_2)}. $$
By orthogonality,
$$ \sum_{b\mod{W}^\times} \chi_1\bar{\chi_2}(b) = \vphi(W)\bfUn_{\chi_1 \sim \chi_2} $$
where by~$\chi_1\sim\chi_2$ we mean that~$\chi_1$ and~$\chi_2$ are induced by the same primitive character -- which necessarily has conductor dividing~$(q_1, q_2)$. Therefore,
$$ \ssum{\chi_1\in\cX_{q_1}(R) \\ \chi_2\in\cX_{q_2}(R)} \chi_1(n_1)\bar{\chi_2(n_2)}\bfUn_{\chi_1\sim\chi_2} = \sum_{\chi_0\in\cX_{(q_1, q_2)}(R)} \chi_0(n_1\bar{n_2}). $$
Since~$\vphi([q_1, q_2])=\vphi(q_1)\vphi(q_2)/\vphi((q_1, q_2))$, we deduce
\begin{equation}
X_3 = \sum_{(q_1q_2, a_1a_2)=1}\frac{\gamma(q_1)\gamma(q_2)}{[q_1, q_2]\vphi((q_1, q_2))}\ssum{\chi_0\in\cX_{(q_1, q_2)}(R)}\ssum{n_1, n_2 \\ (n_j, q_ja_2)=1} \beta_{n_1}\bar{\beta_{n_2}}\chi_0(n_1\bar{n_2}).\label{eq:def-X3}
\end{equation}

\subsubsection[Evaluation of S2]{Evaluation of~$\cS_2$}

The term~$\cS_2$ is defined by
\begin{equation}
\cS_2 = \sum_{(q_1q_2,a_1a_2)=1}\frac{\gamma(q_1)\gamma(q_2)}{\vphi(q_2)}\ssum{n_1, n_2\\(n_j, q_ja_2)=1} \beta_{n_1}\bar{\beta_{n_2}} \sum_{\chi_2\in\cX_{q_2}(R)} \sum_{m\equiv a_1\bar{a_2n_1}\mod{q_1}} \alpha(m)\chi_2(mn_2\bar{a_1}a_2).\label{eq:def-S2-uncond}
\end{equation}
As before, let~$W = [q_1, q_2]$ and~$H=W^{1+\ee}/M$. By Poisson summation,
\begin{equation}
\sum_{m\equiv a_1\bar{a_2n_1}\mod{q_1}} \alpha(m)\chi_2(m) = \frac{\hat{\alpha}(0)}W\ssum{b\mod{W}^\times \\ b\equiv a_1\bar{a_2n_1} \mod{q_1}} \chi_2(b) + O_\ee\Big(\cR_2 + W^\ee \Big),\label{eq:S2-Poisson-uncond}
\end{equation}
where
\begin{equation}
\cR_2 := \frac{M}{W}\sum_{0<|h|\leq H}\Big|\ssum{b\mod{W}^\times \\ b\equiv a_1\bar{a_2n_1}\mod{q_1}} \chi_2(b)\e\Big(\frac{bh}{W}\Big)\Big|.\label{eq:def-R2}
\end{equation}
We wish to express the sum over~$b$ as a complete sum over residues. We write~$W = [q_1, q_2] = q_1'q_2'$, where~$(q_2', q_1)=1$ and~$q_1'|q_1^\infty$ (meaning that~$p|q_1' \Rightarrow p| q_1$). Note that then~$q_1|q_1'$ and~$(q_1', q_2')=1$. Let
$$ \psi : (\bfZ/q_1'\bfZ) \times (\bfZ/q_2'\bfZ) \longrightarrow (\bfZ/W\bfZ) $$
denote the canonical ring isomorphism (so~$\psi^{-1}$ is the projection map). Note that
$$ b_2 \mapsto \chi_2(\psi(1, b_2)) $$
defines a character~$\mod{q_2'}$ of conductor at most~$R$. Finally, we have
$$ \frac1W \equiv \frac{\bar{q_1'}}{q_2'} + \frac{\bar{q_2'}}{q_1'} \mod{1}. $$
The sum over~$b$ in~\eqref{eq:def-R2} is in absolute values at most
\begin{equation}
\ssum{b_1 \mod{q_1'}^\times \\ b_1 \equiv a_1\bar{a_2}\bar{n_1}\mod{q_1}}\Big| \sum_{b_2\mod{q_2'}^\times} \chi_2(\psi(1, b_2))\e\Big(\frac{b_2h\bar{q_1'}}{q_2'}\Big)\Big|\label{eq:somme-b1b2}
\end{equation}
since~$\psi(b_1, b_2)\equiv b_1 \mod{q_1}$, and by factoring
$$ \chi_2(\psi(b_1, b_2)) = \chi_2(\psi(b_1, 1)) \chi_2(\psi(1, b_2)). $$
The sum over~$b_2$ in~\eqref{eq:somme-b1b2} is a Gauss sum; by~\cite[Lemma~3.2]{IK},
\begin{equation}
\Big| \sum_{b_2\mod{q_2'}^\times} \chi_2(\psi(1, b_2))\e\Big(\frac{b_2h\bar{q_1'}}{q_2'}\Big)\Big| \leq R^{1/2}\sum_{d|(h, q_2')} d .\label{eq:contrib-b2}
\end{equation}
Note that
\begin{equation}
\ssum{b_1 \mod{q_1'}^\times \\ b_1 \equiv a_1\bar{a_2n_1}\mod{q_1}} 1 = \frac{\vphi(q_1')}{\vphi(q_1)} = (q_2, q_1^\infty)\label{eq:contrib-b1}
\end{equation}
which is a shorthand for~$\prod_{p^\nu||q_2,\ p|q_1} p^\nu$. Multiplying~\eqref{eq:contrib-b2} with~\eqref{eq:contrib-b1} and summing over~$h$, we obtain
$$ \cR_2 \ll_\ee W^\ee \tau(q_2) (q_2, q_1^\infty) R^{1/2} .$$
Inserting this estimate into~\eqref{eq:S2-Poisson-uncond} then~\eqref{eq:def-S2-uncond}, the error term contributes
$$ \ll_\ee R^{3/2}N^2 W^\ee \sum_{q_1, q_2 \asymp Q} \frac{\tau(q_2)(q_2, q_1^\infty)}{q_2} \ll x^{3\delta/2+\ee}N^2Q. $$
In the last inequality we used standard facts about the kernel function~$k(n)=\prod_{p|n}p$, for which we refer to~\cite{deBruijn}. The error term above is acceptable, since
$$ x^{3\delta/2}Q\leq x^{1/2+3\delta} \leq x^{2/3-2\delta}\leq MR^{-2}x^{-\delta} $$
if~$\delta$ is small enough. We therefore have
$$ \cS_2 = \hat{\alpha}(0)X_2 + O(MN^2R^{-2}) $$
with (having changed~$b$ into~$ba_1\bar{a_2}$)
$$ X_2 = \sum_{(q_1q_2, a_1a_2)=1}\frac{\gamma(q_1)\gamma(q_2)}{[q_1, q_2]\vphi(q_2)} \ssum{n_1, n_2 \\ (n_j, q_ja_2)=1} \beta_{n_1}\bar{\beta_{n_2}} \sum_{\chi_2\in\cX_{q_2}(R)} \chi_2(n_2) \ssum{b\mod{W}^\times \\ b\equiv \bar{n_1} \mod{q_1}} \chi_2(b). $$
Fix~$\chi_2\in\cX_{q_2}(R)$ and let~$\chit_2\mod{\qt_2}$ be the primitive character inducing~$\chi_2$. If~$S$ denotes the sum over~$b$ above, then~$S=\chi_2(c)S$ for any~$c\mod{W}^\times$, $c\equiv 1\mod{q_1}$. Thus~$S=0$ if~$\chi_2$ is not~$q_1$-periodic, that is, if $\qt_2\nmid(q_1, q_2)$. If on the contrary~$\qt_2|(q_1, q_2)$, then~$S=\chit_2(\bar{n_1})\vphi(W)/\vphi(q_1)=\chit_2(\bar{n_1})\vphi(q_2)/\vphi((q_1, q_2))$. We therefore find
$$ \ssum{b\mod{W}^\times \\ b\equiv \bar{n_1}\mod{q_1}} \chi_2(b) = \frac{\vphi(q_2)}{\vphi((q_1, q_2))}\bfUn_{\qt_2|(q_1, q_2)}\chit_2(\bar{n_1}). $$
Summing over~$\chi_2\in\cX_{q_2}(R)$ and since~$(n_1n_2, (q_1, q_2))=1$, we obtain
$$ \sum_{\chi_2\in\cX_{q_2}(R)} \chi_2(n_2) \ssum{b\mod{W}^\times \\ b\equiv \bar{n_1}\mod{q_1}} \chi_2(b) = \frac{\vphi(q_2)}{\vphi((q_1, q_2))}\sum_{\chi_0\in\cX_{(q_1, q_2)}(R)}\chi_0(\bar{n_1}n_2), $$
and so~$X_2 = X_3$.

\subsection{Second reduction}

We now wish to evaluate
$$ \cS_1 := \sum_{(q_1q_2, a_1a_2)=1}\gamma(q_1)\gamma(q_2) \ssum{n_1, n_2 \\ (n_j, q_ja_2)=1 \\ n_1\equiv n_2 \mod{(q_1, q_2)}} \beta_{n_1}\bar{\beta_{n_2}} \ssum{m\equiv a_1\bar{a_2n_1} \mod{q_1} \\ m\equiv a_1\bar{a_2n_2} \mod{q_2}} \alpha(m). $$
The expected main term is~$\hat{\alpha(0)}X_1$, where
\begin{equation}
X_1 := \sum_{(q_1q_2, a_1a_2)=1}\frac{\gamma(q_1)\gamma(q_2)}{[q_1, q_2]} \ssum{n_1, n_2 \\ (n_j, q_ja_2)=1 \\ n_1\equiv n_2 \mod{(q_1, q_2)}} \beta_{n_1}\bar{\beta_{n_2}}.\label{eq:def-X1}
\end{equation}

For all integers~$q_0$, $n_0$ with~$(n_0, q_0)=1$, let~$\cS_1(q_0, n_0)$ denote the contribution to~$\cS_1$ of those integers satisfying~$(q_1, q_2)=q_0$ and~$(n_1, n_2)=n_0$. Then we have
\begin{align*}
|\cS_1(q_0, n_0)| \ll_\ee &\ x^\ee \ssum{q_1, q_2 \asymp Q/q_0 \\ (q_0q_2, a_2n_0)=1} \ssum{n_1, n_2 \asymp N/n_0 \\ n_1\equiv n_2\mod{q_0} \\ (n_2, q_0q_2)=1} \ssum{a_2n_0 n_2 m\equiv a_1\mod{q_0q_2} \\ q_1|ma_2n_0n_1-a_1} \alpha(m) \\
\ll_\ee &\ x^\ee \ssum{q_2 \asymp Q/n_0 \\(q_0q_2, a_2n_0)=1} \ssum{n_1, n_2 \asymp N/n_0 \\ n_1\equiv n_2\mod{q_0} \\ (n_2, q_0q_2)=1} \sum_{ma_2n_0 n_2 \equiv a_1\mod{q_0q_2}} \alpha(m)\tau(|ma_2n_0n_1-a_1|) \\
\ll_\ee &\ x^\ee\Big\{ \frac{MN^2}{n_0^2q_0^2} + \frac{MN}{n_0q_0} \Big\}
\end{align*}
where we used our hypotheses on~$M$ and~$|a_1|$ to justify that~$m|a_1$ cannot be satisfied. Therefore, for some~$\delta>0$ and all~$1\leq K\leq x^\delta$, we have
$$ \ssum{(q_0, n_0)= 1 \\ \max\{q_0, n_0\} > K} |\cS_1(q_0, n_0)| \ll_\ee x^\ee MN^2 K^{-1}. $$
Similarly, if~$X_1(q_0, n_0)$ denotes the contribution to~$X_1$ of indices with~$(q_1, q_2)=q_0$ and~$(n_1, n_2)=n_0$, we have
$$ \ssum{(q_0, n_0)=1 \\ \max\{q_0, n_0\} > K} |X_1(q_0, n_0)| \ll_\ee x^\ee \ssum{(q_0, n_0)=1 \\ \max\{q_0, n_0\} > K} \frac{N}{q_0n_0}\Big(\frac{N}{q_0n_0}+1\Big) \ll_\ee x^\ee N^2 K^{-1}. $$
By choosing~$K$ appropriately, it will therefore suffice to show that
$$ \cS_1(q_0, n_0) = \hat{\alpha}(0)X_1(q_0, n_0) + O(MN^2x^{-\delta}) \qquad (q_0, n_0 \leq x^\delta). $$

\subsection[Evaluation of S1]{Evaluation of~$\cS_1(q_0, n_0)$}

Let the integers~$q_0$, $n_0$ be coprime, at most~$x^\delta$, such that~$(q_0, a_1a_2)=(n_0, a_2)=1$. Let us rename~$q_1$ into~$q_0q_1$ and~$q_2$ into~$q_0q_2$, and similarly for~$n_1$ and~$n_2$. We wish to evaluate
$$ \cS_1(q_0, n_0) = \ssum{q_1, q_2 \\ (q_1q_2, a_1a_2) = (q_1, q_2) = 1} \gamma(q_0q_1)\gamma(q_0 q_2)\ssum{n_1, n_2 \\ (n_0n_j, q_0q_ja_2) = 1 \\ (n_1, n_2) = 1 \\ n_1\equiv n_2 \mod{q_0}} \beta_{n_0n_1}\bar{\beta_{n_0n_2}} \sum_{m\equiv a_1\bar{a_2n_0n_j}\mod{q_0q_j}} \alpha(m). $$
Using Poisson summation, we have
$$ \cS_1(q_0, n_0) = \hat{\alpha(0)}X_1(q_0, n_0) + \cR_1 + O_\ee(x^\ee\cR_2) $$
where, having put~$W = q_0q_1q_2$ and~$H := W^{1+\ee}M^{-1}$,
$$ \cR_1 = \underset{(q_1q_2, a_1a_2)=(q_1, q_2)=1}{\sum_{q_1, q_2}} \underset{\substack{(n_0n_j, q_0q_ja_2)=1 \\ (n_1, n_2)=1 \\ n_1\equiv n_2 \mod{q_0}}}{\sum_{n_1, n_2}} \gamma(q_0q_1)\gamma(q_0q_2) \beta_{n_0n_1}\bar{\beta_{n_0n_2}}  \sum_{0<|h|\leq H} \frac1W\hat{\alpha}\big(\frac{h}{W}\big)\e\Big(\frac{h\mu}{W}\Big), $$
$$ \cR_2 = {\sum_{q_1, q_2 \asymp Q/q_0}} \sum_{n_1, n_2 \asymp N/n_0} \frac1W \ll Q^\ee N^2, $$
and the residue class~$\mu\mod{W}$ satisfies
$$ \mu\equiv a_1\bar{a_2 n_0 n_j}\mod{q_0q_j} \qquad (j\in\{1, 2\}). $$
We seek an error term~$O(MN^2x^{-\delta})$. The contribution of~$\cR_2$ is acceptable.

We now focus on~$\cR_1$. Recall that~$\beta_n$ is non-zero only when~$n$ is squarefree (so that~$(n_0, n_1)=1$). We have the equality modulo~$1$
\begin{equation}
\frac{\mu}{q_0q_1q_2} \equiv \frac{a_1}{q_0q_1q_2a_2n_0n_1} + a_1\frac{n_1-n_2}{q_0}\frac{\bar{q_1a_2n_0n_2}}{n_1q_2} - a_1\frac{\bar{q_0q_1q_2n_1}}{a_2n_0} \mod{1}. \label{eq:mod-1}
\end{equation}
This is found following the steps in~\cite[p.208]{FI}, but can also be more easily verified by multiplying each side by~$q_0q_1q_2a_2n_0n_1$, and checking the resulting congruence modulo~$a_2n_0$, $n_1q_0$, $q_0q_1$ and~$q_0q_2$ respectively. Taking the exponential, we may approximate
$$ \e\big(\frac{ha_1}{q_0q_1q_2a_2n_0n_1}\big) = 1 + O\big(\frac{|h a_1|}{q_0q_1q_2|a_2|n_0n_1}\big). $$ Inserting in~$\cR_1$, the error term contributes a quantity
$$ \ll \frac{|a_1|q_0H}{|a_2|n_0Q^2N} \frac{Q^2}{q_0^2} \frac{N^2}{n_0} \ll x^\ee |a_1|NQ^2M^{-1} $$
which is clearly acceptable. We therefore evaluate
$$ \cR_1' := \sum_{q_1, q_2, n_1, n_2}\frac{\gamma(q_0q_1)\gamma(q_0q_1)}{q_0q_1q_2}\beta_{n_0n_1}\bar{\beta_{n_0n_2}}\hat{\alpha}\Big(\frac h{q_0q_1q_2}\Big) \e\Big(a_1h\frac{n_1-n_2}{q_0}\frac{\bar{q_1a_2n_0n_2}}{n_1q_2} - a_1h\frac{\bar{q_0q_1q_2n_1}}{a_2n_0}\Big). $$
Now we insert the definition of~$\hat{\alpha}$ as
$$ {\hat \alpha}\big(\frac{h}{q_0q_1q_2}\big) = q_0q_1q_2\int_\bfR \alpha(q_0q_1q_2\xi)\e(-h\xi)\dd \xi, $$
we detect the condition~$(a_1, q_1q_2)=1$ by Möbius inversion, and we split the sums over~$q_1$, $q_2$ into congruence classes modulo~$n_0a_2$. We obtain
\begin{equation}
|\cR_1'| \ll_\ee x^\ee (n_0|a_2|)^2 \frac{Mq_0}{Q^2} \sup_{\xi \asymp Mq_0/Q^2}\sup_{\substack{\delta_1, \delta_2 | a_1 \\ (\delta_1, \delta_2)=1 \\ (\delta_1\delta_2, n_0a_2)=1}} \sup_{\substack{\lambda_1, \lambda_2 \mod{n_0a_2}^\times}} \cR_1''\label{eq:majo-R1p}
\end{equation}
where
\begin{align*}
\cR_1'' := &\ \ssum{q_1, q_2 \\ (\delta_1q_1, \delta_2q_2)=1 \\ q_j \equiv \lambda_j\bar{\delta_j}\mod{n_0a_2}} \gamma(q_0\delta_1q_1) \gamma(q_0\delta_2q_2)
\ssum{n_1, n_2 \\ (n_0n_j, q_0\delta_jq_ja_2)=1 \\ (n_1, n_2)=1 \\ n_1\equiv n_2 \mod{q_0}} \beta_{n_0n_1}\bar{\beta_{n_0n_2}} \times \\
&\ \times \sum_{0<|h|\leq H} \alpha(\xi q_0\delta_1\delta_2q_1q_2)\e\big({-}\xi h - a_1h\frac{\bar{q_0\lambda_1\lambda_2n_1}}{a_2n_0}\big) \e\Big(a_1h\frac{n_1-n_2}{q_0}\frac{\bar{a_2n_0n_2\delta_1q_1}}{n_1\delta_2q_2}\Big).
\end{align*}
We write~$\cR_1''$ in the form~\eqref{eq:majo-quintilin}, with
\begin{equation}
{\bf c} \gets q_2, \quad \mathbf{d}\gets q_1, \quad \mathbf{n}\gets a_1h\frac{n_1-n_2}{q_0}, \quad \mathbf{r}\gets a_2n_0n_2\delta_1, \quad \mathbf{s}\gets n_1\delta_2,  \quad \mathbf{q}\gets n_0a_2,\label{eq:substitution-R1}
\end{equation}
$$ \mathbf{C} \gets \frac{Q}{q_0\delta_2}, \quad \mathbf{D} \gets \frac{Q}{q_0\delta_1}, \quad \mathbf{N} \gets \frac{|a_1|HN}{q_0n_0} \quad \mathbf{R} \gets a_2\delta_1 N, \quad \mathbf{S} \gets \frac{N\delta_2}{n_0}. $$
Here bold letters denote the ``new'' summation variables in~\eqref{eq:majo-quintilin}. The analogue of the sequence~$b_{\mathbf{n},\mathbf{r},\mathbf{s}}$ is defined through
$$ b_{{\bf n}, {\bf r}, {\bf s}} = \underset{\substack{(n_0n_j, q_0\delta_ja_2)=1 \\ (n_1, n_2)=1 \\ n_1 \equiv n_2\mod{q_0} \\ {\bf r} = a_2n_0n_2\delta_1 \\ {\bf s} = n_1\delta_2}}{\sum_{n_1} \sum_{n_2}} \beta_{n_0n_1}\bar{\beta_{n_0n_2}} \ssum{0<|h|\leq H \\ q_0{\bf n} = a_1 h (n_1-n_2)} \e\Big({-}\xi h - a_1h\frac{\bar{q_0\lambda_1\lambda_2n_1}}{a_2n_0}\Big). $$
Note that this has at most one term since the case~$n_1=n_2$ is prohibited by the conditions~$(n_1, n_2)=1$ and~$N\geq x^\eta$. Note also that it is void unless~$({\bf r}, {\bf s})=1$ (here we use the fact that~$\beta$ is supported on squarefree integers). The quantity~$g({\bf c}, {\bf d}, {\bf n}, {\bf r}, {\bf s})$ in~\eqref{eq:majo-quintilin} is
$$ \gamma(q_0\delta_1{\bf d})\gamma(q_0\delta_2{\bf c}) \alpha(\xi q_0\delta_1\delta_2{\bf c} {\bf d}). $$
The derivative conditions~\eqref{eq:cond-derg-quintilin} are satisfied with~$\ee_0=B\delta$, by virtue of our hypothesis on~$\gamma$. Note that the congruence and coprimality conditions on~$q_1$ and~$q_2$ translate exactly into
$$ {\bf c} \equiv \lambda_2 \bar{\delta_2} \mod{{\bf q}}, \quad {\bf d}\equiv \lambda_1 \bar{\delta_1}\mod{{\bf q}}, \quad ({\bf d}, {\bf c s})=({\bf c}, {\bf r})=1. $$
At this point, we are in a situation analogous to~\cite[formula~(13.2)]{BFI}. Applying Theorem~\ref{thm:quintilin} and estimating the resulting expression as in~\cite[page~241]{BFI}, we obtain
$$ \cR_1'' \ll x^{O(\delta)} \cA^{1/2}\cB^{1/2}, $$
where~$\cA \ll HN^2$ is the contribution coming from~$\|b_{N,R,S}\|_2^2$ in~\eqref{eq:majo-quintilin}, and
$$ \cB \ll Q^2N^2N(H+N) + Q^3N^2\sqrt{H+N} + Q^2HN \ll (QN)^2\{N(H+N) + Q\sqrt{H+N}\}. $$
We have~$H \ll x^{O(\delta)}N$, so that~$\cB\ll Q^2N^2x^{O(\delta)}(N^2+Q\sqrt{N})$ (compare with~\cite[formula~(13.4)]{BFI}). Inserting in~\eqref{eq:majo-R1p}, we obtain
$$ \cR_1' \ll x^{O(\delta)}MN^2(Q^{-1}N^{3/2} + Q^{-1/2}N^{3/4}) \ll x^{-\eta/2+O(\delta)}MN^2 $$
by the hypothesis~$N\leq Q^{2/3-\eta}$. Taking~$\delta$ sufficiently small in terms of~$\eta$, we have the required bound~$O(MN^2x^{-\delta})$.

\subsection{The main terms}

The main terms~$X_1$ and~$X_3$ defined in~\eqref{eq:def-X1} and~\eqref{eq:def-X3} are real numbers. They combine to form
$$ X_1 - X_3 = \sum_{(q_1q_2, a_1a_2)=1} \frac{\gamma(q_1)\gamma(q_2)}{[q_1, q_2]} \ssum{n_1, n_2 \\ (n_j, q_ja_2)=1} \bar{\beta_{n_1}}\beta_{n_2}\cu_R(n_1\bar{n_2} ; (q_1, q_2)). $$
Notice the summands are zero unless~$(q_1, q_2)>R$. We use Möbius inversion
$$ \bfUn_{(n_j, q_j)=1} = \sum_{d_j|(q_j, n_j)} \mu(d_j) $$
to detect the conditions~$(n_j, q_j)=1$, in order to separate the sums over~$n_1$,~$n_2$ from those over~$q_1$,~$q_2$. We insert the definition of~$\cu_R$ in the form
$$ \cu_R(n_1\bar{n_2} ; (q_1, q_2)) = \frac1{\vphi((q_1, q_2))}\ssum{\chi\prim \\ \cond(\chi)>R \\ \cond(\chi) | (q_1, q_2)} \chi(\bar{n_1})\chi(n_2). $$
We can assume~$(d_j, \cond(\chi))=1$ because of the factors~$\chi(n_j)$. Quoting from~\cite[Theorem~I.5.4]{Tenenbaum-Intro} the bound~$\vphi(q)\gg q/\log\log q$, we obtain
$$ X_1 - X_3 \ll (\log \log x) \sum_{R<r\leq Q} \ssum{d_1, d_2 \\ d_j \ll Q/r} \Big( \ssum{q_1, q_2 \\ q_j\asymp Q \\ rd_j | q_j}\frac1{q_1q_2} \Big) \ssum{\chi\prim \\ \chi\mod{r}} \prod_{j=1}^2 \Big|\sum_{(n, a_2)=1} \beta_{d_j n}\chi(n)\Big|. $$
The sum over~$q_1$,~$q_2$ is~$O(1/(r^2d_1d_2))$. By Cauchy--Schwarz, and the symmetry between~$n_1$ and~$n_2$, we obtain
$$ X_1 - X_3 \ll (\log x)^2 \ssum{d \ll N}\frac1{d} \sum_{R<r\leq Q}\frac1{r^2} \ssum{\chi\prim \\ \chi\mod{r}} \Big|\sum_{(n, a_2)=1} \beta_{d n}\chi(n)\Big|^2. $$
For all~$t>R$, the multiplicative large sieve inequality (Lemma~\ref{lemme:GC}) and our hypothesis~\eqref{eq:cond-taille} yields
$$ G(t) := \sum_{R<r \leq t}  \ssum{\chi\prim \\ \chi\mod{r}} \Big|\sum_{(n, a_2)=1} \beta_{d n}\chi(n)\Big|^2 \ll (\log x)^{O(1)}\tau(d)^{2A}(t^2 + N)N $$
after ignoring denominators~$d$. We obtain by partial summation
$$ X_1 - X_3 \ll (\log x)^2 \ssum{d \ll N} \frac1{d} \Big(\frac{G(Q)}{Q^2} + \int_R^Q \frac{G(t)}{t^3}\dd t \Big)
\ll (\log x)^{O(1)}(N + N^2R^{-2}). $$
By hypothesis~$R\leq x^{\delta}$, so we have the desired bound~$X_1 - X_3 \ll N^2R^{-2}(\log x)^{O(1)}$. Given~$\hat{\alpha}(0) \ll M$, our claimed estimate~\eqref{eq:obj-dispersion} is proved, and therefore Proposition~\ref{prop:distrib-convo} as well.

\section{Application to the Titchmarsh divisor problem}\label{sec:appl-titchm-divis}

The aim of this section is to justify Theorems~\ref{thm:TM-cond} and~\ref{thm:TM-uncond}. Recall the definition
$$ T(x) := \sum_{1<n\leq x} \Lambda(n)\tau(n-1) .$$
We let
$$ \psi(x ; q, a) := \ssum{n\leq x\\n\equiv a\mod{q}}\Lambda(n), \qquad \psi_q(x) := \ssum{n\leq x\\(n, q)=1}\Lambda(n), \qquad  \psi(x, \chi) := \sum_{n\leq x}\Lambda(n)\chi(n) .$$
Let us recall the following classical theorem of Page~\cite[Theorems~5.26, 5.28]{IK}.
\begin{lemme}
There is an absolute constant~$b$ such that for all~$Q, T\geq 2$, the following holds. The function~$s\mapsto \prod_{q\leq Q}\prod_{\chi\mod{q}} L(s, \chi)$ has at most one zero~$s=\beta$ satisfying~$\Re(s)>1-b/\log(QT)$ and~$|\Im(s)|\leq T$. If it exists, the zero~$\beta$ is real and it is the zero of a unique function~$L(s, \chit)$ for some primitive real character~$\chit$.
\end{lemme}
Given a large~$x$, we shall say that~$\chit$ is~$x$-exceptional if the above conditions are met with~$Q=T=\e^{\sqrt{\log x}}$. For all~$q\geq 1$ for which~$\qt|q$, we let~$\chit_q$ denote the character~$\mod{q}$ induced by~$\chit$.

\subsection{Primes in arithmetic progressions}

We deduce from the previous sections the following result about equidistribution of primes in arithmetic progressions.
\begin{theoreme}\label{thm:BVP}
Assume the GRH. For some~$\delta>0$, all~$x\geq 1$, $Q\leq x^{1/2+\delta}$ and all integers~$0<|a_1|, |a_2|\leq x^\delta$,
$$ \ssum{q\leq Q \\(q, a_1a_2)=1} \big(\psi(x ; q, a_1\bar{a_2}) - \frac1{\vphi(q)}\psi_q(x) \big) \ll x^{1-\delta} .$$
Unconditionally, under the same assumptions,
$$ \ssum{q\leq Q \\ (q, a_1a_2)=1} \big(\psi(x ; q, a_1\bar{a_2}) - \frac{\psi_q(x) + \bfUn_{\qt|q}\chit(a_2\bar{a_1})\psi(x, \chit_q)}{\vphi(q)}\big) \ll x\e^{-\delta\sqrt{\log x}}, $$
where the term~$\psi(x ; \chit_q)$ is to be taken into account only if the~$x$-exceptional character~$\chit$ exists.
\end{theoreme}

Using the Dirichlet hyperbola method (see in particular section~VII of~\cite{FouvryTit}), it follows that the same estimate holds on the condition~$q\leq x^{1-\ee}$ for any fixed~$\ee>0$ (the implicit constants and~$\delta$ may then depend on~$\ee$). Note however that the symmetry point is at~$q\approx (x|a_2|)^{1/2}$, rather than~$x^{1/2}$ (so the flexibility of taking~$Q$ somewhat larger than~$x^{1/2}$ is not superfluous). We refer to~\cite{Fiorilli} for more explanations on what happens when~$Q$ is very close to~$x$.

As mentioned in the introduction, the uniformity in~$a_1$ and~$a_2$ is an interesting question. At the present state of knowledge, bounds coming from the theory of automorphic forms are typically badly behaved in that aspect. By using a more refined form of the combinatorial decomposition~\eqref{eq:sum-multilin}, Friedlander and Granville~\cite{FG} prove that~$|a_1|\leq x^{1/4-\ee}$ is admissible for all~$\ee>0$ (in the case~$a_2=1$), with a somewhat larger error term.

For the application to the Titchmarsh divisor problem, the following slightly weaker statement suffices.
\begin{prop}\label{prop:majo-q2}
For some~$\delta>0$, all~$x\geq 2$ and~$0<|a|\leq x^\delta$, assuming the GRH, we have
\begin{equation}
\ssum{q\leq \sqrt{x} \\(q, a)=1} \big(\psi(x ; q, a) - \psi(q^2 ; q, a) - \frac{\psi_q(x) - \psi_q(q^2)}{\vphi(q)} \big) \ll x^{1-\delta}.\label{eq:majo-q2-cond}
\end{equation}
Unconditionally,
\begin{equation}
\begin{aligned}
\ssum{q\leq \sqrt{x} \\(q, a)=1} \big(\psi(x ; q, a) - \psi(q^2 ; q, a) - \frac{\psi_q(x) - \psi_q(q^2)}{\vphi(q)}\ -&\ \bfUn_{\qt|q}\bar{\chi(a)}\frac{\psi(x ; \chit_q) - \psi(q^2 ; \chit_q)}{\vphi(q)} \big) \\ &\ll x\e^{-\delta\sqrt{\log x}}.\label{eq:majo-q2-uncond}
\end{aligned}
\end{equation}
\end{prop}
We will focus here on proving Proposition~\ref{prop:majo-q2} only, because the presentation is slightly simpler and addresses all the essential issues.

\begin{proof}[of Proposition~\ref{prop:majo-q2}]
Let~$1\leq R\leq x^{1/10}$ be a parameter. Let
$$ \cS_1 := \ssum{q\leq \sqrt{x} \\ (q, a)=1} \ssum{q^2 < n\leq x \\ n\equiv a\mod{q}} \Lambda(n). $$
By orthogonality of characters,
\begin{equation}
\cS_1 = \ssum{q\leq \sqrt{x} \\ (q, a)=1} \frac1{\vphi(q)}\ssum{\chi\mod{q}}
\ssum{q^2 < n \leq x} \chi(n\bar{a}) \Lambda(n) \label{eq:TM-cond-orthog}
\end{equation}
We decompose~$\cS_1=\cS_1^- + \cS_1^+$ where~$\cS_1^-$ is the contribution of those characters~$\chi$ of conductor at most~$R$, and
$$ \cS_1^+ = \ssum{q\leq \sqrt{x} \\ (q, a)=1} \ssum{q^2 < n \leq x} \Lambda(n) \cu_R(n\bar{a} ; q). $$

We first focus of~$\cS_1^+$. By the Heath-Brown identity~\cite[lemma~5]{BFI} and a dichotomy argument similar to~\cite[Section~2.(a)]{FT-Piltz}, the problem is reduced to showing
\begin{equation}\label{eq:sum-multilin}
\begin{aligned}
\ssum{Q<q\leq2Q\\(q,a)=1}\underset{\substack{(1-\Delta)M_i<m_i\leq \min\{M_i,x^{1/4}\}\\(1-\Delta)N_i<n_i\leq N_i \\1\leq i\leq j}}{\sum\cdots\sum}
\mu(m_1)\cdots\mu(m_j)(\log n_1) &\cu_R(n_1m_1\cdots n_jm_j\bar{a} ; q) \\
&\ll x(\log x)^{O(1)}R^{-1}
\end{aligned}
\end{equation}
where~$j\in\{1, 2, 3, 4\}$,~$x^{-1/10}<\Delta \leq 1/2$, and~$Q, M_i, N_i\geq 1\ (1\leq i\leq j)$ are real numbers such that
\[ Q^2 \leq \prod_i M_iN_i \leq x, \qquad M_i\leq 2x^{1/4} .\]

Let us justify briefly this step. The Heath-Brown identity states that~$\cS_1^+$ is a linear combination of the expression on the left-hand side of~\eqref{eq:sum-multilin} for various values of~$j$, with the conditions~$q\leq \sqrt{x}$, $m_i\leq x^{1/4}$ and~$q^2 < m_1n_1\dotsb m_jn_j \leq x$. We then localize $q$ in dyadic intervals, and each~$n_i$, $m_i$ in intervals~$[(1-\Delta)X, X]$ ($X=N_i$ or~$M_i$). Having done this, the subset of~$(M_i, N_i)$ for which the condition~$q^2< \prod_i m_i n_i \leq x$ is relevant will only concern those indices with~$\prod_i m_i n_i \in [(1-\Delta)^{8}x, (1-\Delta)^{-8}x]$ or~$[(1-\Delta)^{8}q^2, (1-\Delta)^{-8}q^2]$. For those~$M_i, N_i$, we apply Lemma~\ref{lemme:shiu} or a trivial bound (if~$q$ is very small); for the others, the bound~\eqref{eq:sum-multilin} will apply. We deduce respectively
\begin{equation}
\cS_1^+ \ll x \Delta (\log x)^{O(1)} + x\Delta^{-8}(\log x)^{O(1)}R^{-1}\label{eq:rel-S1p-smooth}
\end{equation}
and optimizing~$\Delta$ yields~$\cS_1^+ \ll x (\log x)^{O(1)}R^{-1/9}$.

Let~$\eta>0$ be small. The contribution of tuples such that~$\prod_i M_iN_i \leq x^{1-\eta}$ is trivially bounded by~$O_\ee
(x^{1-\eta + \ee})$ using~Lemma~\ref{lemme:shiu}. Suppose then~$\prod_i M_iN_i > x^{1-\eta}$. For convenience we rename~$x = \prod_iM_i N_i$. Our objective bound for~\eqref{eq:sum-multilin} is~$O(x^{1-\delta})$ and we now have~$M_i\leq x^{1/4+\eta}$ if~$\eta$ is small enough.

Fix~$\eta\in(0, 1/100]$. At least one of the three following cases must hold:
\begin{enumerate}[(a)]
\item there exists an index~$k$ such that~$N_k>x^{1-(2j-1)\eta}$,
\item we have~$\min\{N_k, N_{k'}\}>x^{1/3-\eta}$ for two indices~$k\neq k'$,
\item there exists an index~$k$ such that~$M_k$ or~$N_k$ lies in the interval~$[x^\eta, x^{1/3-\eta}]$.
\end{enumerate}

In case~(a), our sum~\eqref{eq:sum-multilin} is at most
\begin{equation}
\cS_a := x^\ee\ssum{Q<q\leq2Q \\ (q, a)=1}\ssum{M/2 < m\leq M} \Big| \sum_{(1-\Delta)N < n \leq N}\beta_n \cu_R(mn\bar{a} ; q) \Big|\label{eq:convo-a}
\end{equation}
with~$\beta=\bfUn$ or~$\log$,~$MN=x$ and~$N\geq x^{1-7\eta}$. Choose~$\eta<1/30$, for the sum over~$n$, we express~$\cu_R$ as~\eqref{eq:expr-cuR-cu1}. Using
\begin{equation}
\ssum{n\leq z\\n\equiv a\mod{q}} 1 = \frac zq  + O(1) \qquad (z\geq 1, (a, q)\in\bfN^2)\label{eq:comptage-intervalle}
\end{equation}
and partial summation in case~$\beta=\log$, we get that the sum over~$n$ above is
$$  \sum_{(1-\Delta)N < n \leq N}\beta_n \cu_R(mn\bar{a} ; q)  \ll \log x + \frac1{\vphi(q)} \ssum{\chi\mod{q} \\ 1<\cond(\chi)\leq R} \big|\sum_{(1-\Delta)N < n \leq N} \beta_n\chi(n)\big|. $$
For each~$\chi$ in the above, the sum over~$n$ is estimated using Lemma~\ref{lemme:PV} as
$$ \sum_{(1-\Delta)N < n \leq N} \beta_n\chi(n) \ll R^{1/2}(\log x)^2 \tau(q). $$
Dropping the condition~$\cond(\chi)\leq R$, we obtain for~\eqref{eq:convo-a} a crude bound
$$ \cS_a \ll_\ee x^\ee MQR^{1/2} \ll QR^{1/2}x^{8\eta} \ll x^{11/20+8\eta+\delta} $$
which is acceptable.

Consider case~(b). Then the sum on the LHS of~\eqref{eq:sum-multilin} is of the form
\begin{equation}
\cS_b := \ssum{Q<q\leq2Q\\(q, a)=1}\underset{\substack{(1-\Delta)N<n\leq N\\(1-\Delta)M<m\leq M \\ (1-\Delta)^{2j-2}L<\ell\leq L}}{\sum\sum\sum}
\alpha(m)\beta(n)\gamma_\ell\cu_R(mn\ell\bar{a} ; q)\label{eq:convo-b}
\end{equation}
where~$M, N>x^{1/3-\eta}$,~$MNL=x$,~$\alpha$ and~$\beta$ are either~$\bfUn$ or~$\log$, and~$\gamma_\ell$ satisfies
\[ |\gamma_\ell| \leq \tau_{2j-2}(\ell)\log \ell \]
By partial summation and upon rewriting the size restrictions on~$m, n, \ell, q$ as differences of one-sided inequalities,
it suffices to establish the bound
\[ \cS_b' :=  \sum_{\ell\leq L}\Bigg|\ssum{q\leq Q\\(q, a\ell)=1}\sum_{m\leq M}\sum_{n\leq N}\cu_R(mn\ell\bar{a} ; q)\Bigg| \ll x^{1-\delta} \]
whenever~$M, N>x^{1/3-2\eta}$ and~$Q\leq 2\sqrt{x}$. Writing~$\cu_R$ as in~\eqref{eq:expr-cuR-cu1}, we have by the triangle inequality
$$ \cS_b' \ll \cS'_{b1} + \cS'_{b2}, $$
where
$$ \cS'_{b1} = \sum_{\ell\leq L}\Bigg|\ssum{q\leq Q\\(q, a\ell)=1}\sum_{m\leq M}\sum_{n\leq N}\cu_1(mn\ell\bar{a} ; q)\Bigg| ,  $$
$$ \cS'_{b2} = \sum_{\ell\leq L}\sum_{q\leq Q}\frac1{\vphi(q)}\ssum{\chi\mod{q} \\ 1<\cond(\chi)\leq R} \Big|\sum_{m\leq M}\chi(m)\Big| \Big|\sum_{n\leq N}\chi(n)\Big|. $$
Theorem~7 of~\cite{BFI} yields the acceptable bound~$\cS'_{b1}\ll x^{1-\delta}$ as long as~$\eta<1/30$. In~$\cS'_{b2}$, by Lemma~\ref{lemme:PV}, the sums over~$m$ and~$n$ are majorized by~$O(\tau(q)R^{1/2+\ee})$. Dropping the condition~$\cond(\chi)\leq R$, we obtain for~\eqref{eq:convo-b} a bound
$$ \cS'_{b2} \ll_\ee x^\ee LRQ \ll_\ee x^{14/15+5\eta} $$
which is also acceptable.

In case~(c), we write our sum as
\begin{equation}
\cS_c := \ssum{Q<q\leq 2Q \\ (q, a)=1} \underset{\substack{(1-\Delta)^{2j-1}M<m\leq M \\ (1-\Delta)N<n\leq N}}{\sum\sum} \alpha_m \beta_n \cu_R(mn\bar{a} ; q)\label{eq:convo-c}
\end{equation}
where~$x^\eta\leq N\leq x^{1/3-\eta}$, so~$M\geq x^{2/3}$. We may assume that~$R\leq x^{\eta/2}$. If~$Q\leq x^{1/2-\eta/2}$, then Lemma~\ref{lemme:motohashi} is applicable. If on the contrary~$x^{1/2-\eta/2} < Q \leq \sqrt{x}$, then Theorem~\ref{thm:distrib-convo} is applicable with~$\eta\gets \eta/2$ (assuming~$|a|\leq x^{\delta/2}$ as we may). In both cases, we obtain that the quantity~\eqref{eq:convo-c} is majorized by
$$ \cS_c \ll x(\log x)^{O(1)}R^{-1}. $$
Summarizing the above and in view of~\eqref{eq:rel-S1p-smooth}, we have obtained
$$ \cS_1^+ \ll x(\log x)^{O(1)}R^{-1/9}. $$
We consider now~$\cS_1^-$, which we recall is
\begin{equation}
\cS_1^- = \ssum{q\leq \sqrt{x} \\ (q, a)=1} \frac1{\vphi(q)} \ssum{\chi\mod{q} \\ \cond(\chi)\leq R} \ssum{q^2 < n \leq x} \Lambda(n)\chi(n\bar{a}).\label{eq:def-S1moins}
\end{equation}

First let us assume the GRH. Isolating the contribution of the principal character, we write
$$ \cS_1^- = \ssum{q\leq \sqrt{x} \\ (q, a)=1} \frac{\psi_q(x) - \psi_q(q^2)}{\vphi(q)} + \cS_1^\flat, $$
say. For any non-trivial character~$\chi\mod{q}$ with~$q\leq x$, the GRH~\cite[formula~(13.19)]{MontgomeryVaughan} yields
$$ \ssum{q^2 < n \leq x} \chi(n)\Lambda(n) \ll x^{1/2}(\log x)^2. $$
We therefore have
$$ \cS_1^\flat \ll x^{1/2}(\log x)^2 \sum_{q\leq \sqrt{x}} \frac1{\vphi(q)} \ssum{\chi\mod{q} \\ \cond(\chi) \leq R} 1 \ll Rx^{1/2}(\log x)^3 $$
which is acceptable. The choice~$R = x^{\delta}$ for small enough~$\delta$ concludes the proof of~\eqref{eq:majo-q2-cond}.

Unconditionally, for any~$q\leq \e^{\sqrt{\log x}}$ and any non-principal, non $x$-exceptional character~$\chi\mod{q}$, we have by a straightforward adaptation of~\cite[Theorem~11.16]{MontgomeryVaughan} the estimate
$$ \ssum{q^2 < n \leq x} \chi(n)\Lambda(n) \ll x\e^{-c\sqrt{\log x}} $$
for some absolute constant~$c>0$. Choose~$R = \e^{c\sqrt{\log x}/2}$. We extract from~$\cS_1^-$ the contribution from the principal character and the possible $x$-exceptional characters, and write accordingly
$$ \cS_1^- = \ssum{q\leq \sqrt{x} \\ (q, a)=1} \frac{\psi_q(x) - \psi_q(q^2) + \bfUn_{\qt|q}\bar{\chi(a)}(\psi(x ; \chit_q) - \psi(q^2 ; \chit_q))}{\vphi(q)} + \cS_1^{\flat\flat} + O(x\e^{-c\sqrt{\log x}/2}) $$
the error term being there to cover the trivial case when either~$\chit$ was inexistant, or~$\qt>R$. By the same computation as above,
$$ \cS_1^{\flat\flat} \ll Rx(\log x)\e^{-c\sqrt{\log x}} \ll x\e^{-c\sqrt{\log x}/3}. $$
This concludes the proof of~\eqref{eq:majo-q2-uncond} hence of Proposition~\ref{prop:majo-q2}.

\end{proof}

\subsection{Proof of Theorems~\ref{thm:TM-cond} and~\ref{thm:TM-uncond}}\label{sec:deduction-TM}

It is now straightforward to deduce Theorems~\ref{thm:TM-cond} and~\ref{thm:TM-uncond}. By the Dirichlet hyperbola method~\cite[page~45]{FT-Piltz}, we have
$$ T(x) = 2\sum_{q\leq \sqrt{x}} \big(\psi(x ; q, 1) - \psi(q^2 ; q, 1)\big) + O(\sqrt{x}\log x). $$
Assume first the GRH. Then Proposition~\ref{prop:majo-q2} yields
$$ T(x) = 2\sum_{q\leq \sqrt{x}} \frac{\psi_q(x) - \psi_q(q^2)}{\vphi(q)} + O(x^{1-\delta}) $$
The GRH~\cite[formula~(13.19)]{MontgomeryVaughan} allows us to deduce
$$ T(x) = 2\sum_{q\leq \sqrt{x}} \frac{x - q^2}{\vphi(q)} + O(x^{1-\delta}). $$
The main term is computed using~\cite[Lemme~6]{Fouvry82}, which yields the claimed estimate.

\medskip

Unconditionally, from Proposition~\ref{prop:majo-q2}, we merely have to add to our estimate for~$T(x)$ the additional contribution of the $x$-exceptional character (if it exists), which takes the form
\begin{equation}
2\ssum{q\leq \sqrt{x} \\ \qt|q} \frac{\psi(x ; \chit_q) - \psi(q^2 ; \chit_q)}{\vphi(q)} \label{eq:contrib-ap-BV-uncond}
\end{equation}
We have from~\cite[Theorem~11.16]{MontgomeryVaughan}
$$ \psi(x ; \chit_q) = -\frac{x^{\beta}}{\beta} + O(x\e^{-\delta\sqrt{\log x}}) $$
and similarly
$$ \psi(q^2 ; \chit_q) = -\frac{q^{2\beta}}{\beta} + O(x\e^{-\delta\sqrt{\log x}}) $$
at the possible cost of changing the numerical value of~$\delta$. We obtain that~\eqref{eq:contrib-ap-BV-uncond} equals
$$ -\frac{2}{\beta}\ssum{q\leq \sqrt{x} \\ \qt|q} \frac{x^{\beta} - q^{2\beta}}{\vphi(q)} + O(x\e^{-\delta\sqrt{\log x}}). $$
The sums over~$q$ are computed using~\cite[Lemme~6]{Fouvry82} (and partial summation in the form~$x^\beta - q^{2\beta}=\beta \int_{q^2}^xt^{\beta-1}\dd t$), which yields Theorem~\ref{thm:TM-uncond}. Corollary~\ref{coro:TM-partial} is straightforward.

There remains to justify Corollary~\ref{coro:TM-onesided}. Note that~$C_2(\qt)$ is absolutely bounded, while~$\qt \leq \e^{\sqrt{\log x}}$ by definition. Therefore~$x^\beta\to\infty$, and~$\beta\li(x^\beta)/x^\beta \sim (\log x)^{-1}$. We deduce
$$ \frac{\log \qt + C_2(\qt) - \gamma}{x^\beta /(\beta \li(x^\beta))} \underset{x\to\infty}{\longrightarrow} 0$$
in an effective way. For~$x$ large enough, it is less than~$1/3$ and Corollary~\ref{coro:TM-onesided} follows.

\begin{remarque}
If we were to consider~$\tau(n-a)$ instead of~$\tau(n-1)$, for some~$a$ which is not a perfect square, then the Siegel zero contribution (if it 
existed) would have a twist by~$\chi(a)$, which is~{\it a priori} of unpredictable sign.
\end{remarque}

\section{Application to correlation of divisor functions}\label{sec:appl-corr-divis}

In this section, we justify Theorem~\ref{thm:Tk}. The proof has the same structure as that of Theorems~\ref{thm:TM-cond} and~\ref{thm:TM-uncond}, replacing the function~$\Lambda(n)$ by~$\tau_k(n)$.

\subsection{An equidistribution estimate}

The analogue of Theorem~\ref{thm:BVP} is the following:
\begin{theoreme}\label{thm:distrib-tauk}
There exists~$\eta>0$ such that under the conditions~$k\geq 4$, $0<|a|\leq x^{\eta}$ and~$Q\leq x^{1/2+\eta}$,
\begin{equation}
\ssum{q\leq Q\\(q, a)=1}\Big(\ssum{n\leq x \\ n\equiv a\mod{q}} \tau_k(n) - \frac1{\vphi(q)}\ssum{n\leq x \\ (n, q)=1} \tau_k(n)\Big) \ll x^{1-\eta/k}.\label{eq:majo-distrib-Tk}
\end{equation}
If the Lindelöf hypothesis is true for all Dirichlet~$L$-functions, then the right-hand side can be replaced by~$x^{1-\eta}$.
\end{theoreme}
In order to simplify the presentation, we put
$$ \cE = \begin{cases} x &\text{if the generalized Lindelöf hypothesis is assumed,} \\
x^{1/k} & \text{unconditionally.}
\end{cases} $$
To handle the small conductor case, we require the following.
\begin{lemme}\label{lemme:Tk-SW}
For some~$\delta>0$ and any non-principal character~$\chi\mod{q}$ with~$q\leq x$, of conductor~$r\leq \cE^\delta$ we have
$$ \sum_{n\leq x}\tau_k(n)\chi(n) \ll_k x\cE^{-\delta} .$$
\end{lemme}
\begin{proof}
Starting from the representation
$$ \sum_{n\leq x}\tau_k(n)\chi(n) = \frac1{2\pi i}\int_{1+1/(\log x)-i\infty}^{1+1/(\log x)+i\infty} L(s, \chi)^k \frac{x^s\dd s}{s} \qquad (x\not\in\bfN),$$
one may truncate the contour at~$T=x^{\delta/k}$, and shift it to the abscissa~$\Re(s)=1-\delta/k$. The convexity bound~$|L(1-\delta/k+it, \chi)|\ll q^\ee(r(|t|+1))^{c\delta/k+\ee}$ (for some~$c>0$) yields the desired estimate if~$\cE=x^{1/k}$. If the Lindelöf hypothesis~$ L(\tfrac12 + it, \chi) \ll (q(|t|+1))^{\ee}$ is true, then one chooses~$T=x^{\delta}$ and shifts the contour to~$\Re(s)=1-\delta$, where the bound~$L(1-\delta+it, \chi) \ll (q(|t|+1))^\ee$ holds by convexity.
\end{proof}

\subsubsection{Small conductors}

Let~$\cS_0$ denote the quantity in the left-hand side of~\eqref{eq:majo-distrib-Tk}, and let~$R\leq \cE^{\delta}$. The contribution of those characters~$\chi$ having conductors at most~$R$ is
$$ \sum_{1<r\leq R}\ssum{\chi\mod{r} \\ \chi\prim}\bar{\chi(a)}\ssum{q\leq Q \\ (q, a)=1 \\ r|q}\frac1{\vphi(q)}\ssum{n\leq x \\ (n, q)=1} \tau_k(n)\chi(n). $$
By Lemma~\ref{lemme:Tk-SW} applied to the character~$\mod{q}$ induced by~$\chi$, we have a bound
$$ x \cE^{-\delta} \sum_{r\leq R}\ssum{\chi\mod{r} \\ \chi\prim}\ssum{q\leq Q \\ r|q}\frac1{\vphi(q)} \ll x\cE^{-\delta} R(\log x)^2 .$$
Letting~$R=\cE^{\delta/2}$, this is an acceptable error term. There remains to bound
$$ \cS_1 := \ssum{q\leq Q \\ (q, a)=1}\sum_{n\leq x}\tau_k(n)\cu_R(n\bar{a} ; q). $$

\subsubsection{Dyadic decomposition}

We dyadically decompose in~$\cS_1$ the sums over~$q$ and~$n$ in~\eqref{eq:majo-distrib-Tk}, yielding an upper bound
\begin{equation}
\cS_1 \ll (\log x)^2 \sup_{\substack{Q'\leq x^{1/2+\eta} \\ N\leq x}} \Big|\ssum{Q'<q\leq 2Q' \\ (q, a)=1} \ssum{N< n \leq 2N} \tau_k(n) \cu_R(n\bar{a} ; q) \Big| .\label{eq:somme-Tk-dicho}
\end{equation}
Let~$\eta>0$ and assume throughout that~$\delta$ is small with respect to~$\eta$. When~$N \leq x^{1-\eta}$, by the triangle inequality, our trivial bound~\eqref{eq:tv-cuR} and Lemma~\ref{lemme:shiu}, the sum over~$q$ and~$n$ above is~$O_k(x^{1-\eta/2})$, so we may add the restriction~$N>x^{1-\eta}$ in the supremum with an acceptable error. Then we relax the condition~$Q'\leq x^{1/2+\eta}$ into~$Q'\leq N^{1/2+2\eta}$. Renaming~$N$ into~$x$, and expanding out~$\tau_k(n)$, we obtain that it will suffice to prove
\begin{equation}
\cS_2 := \ssum{Q<q\leq 2Q \\ (q, a)=1} \ssum{x< n_1\dotsb n_k \leq 2x} \cu_R(n_1\dotsb n_k \bar{a} ; q) \ll x\cE^{-\eta}\label{eq:Tk-apres-dicho}
\end{equation}
under the constraints~$|a|\leq x^{2\eta}$ and~$Q\leq x^{1/2+2\eta}$.
We decompose the sums over~$n_1, \dotsc, n_k$ dyadically to obtain an upper bound
\begin{equation}
\cS_2 \ll \cS_3 := (\log x)^k \sup_{N_1, \dotsc, N_k \geq 1/2} \Big| \ssum{Q<q\leq 2Q \\ (q, a)=1} \ssum{x< n_1\dotsb n_k \leq 2x \\ N_j <n_j\leq 2N_j} \cu_R(n_1\dotsb n_k \bar{a} ; q) \Big|.\label{eq:Tk-apres-dicho-2}
\end{equation}

\subsubsection{Splitting cases}

Let the parameter~$0<\delta_1<1/100$ be fixed. We separate into two cases according to whether there is a subset~$\cJ\subset\{1, \dotsc, k\}$ such that
$$ \prod_{j\in\cJ} N_j\in (x^{\delta_1}, x^{1/3-\delta_1}], $$
or not. Suppose there is no such subset, and let
$$ \cK := \{j:\ 1\leq j\leq k,\ N_j> x^{1/3-\delta_1} \}. $$
Necessarily~$\card\cK\leq 3$. Since~$N_j\leq x^{\delta_1}$ for each~$j\not\in\cK$, and by assumption there is no subset~$\cL\subset\{1, \dotsc, k\}\smallsetminus\cK$ such that~$\prod_{j\in\cL} N_j \in (x^{\delta_1}, x^{1/3-\delta_1}]$, it is necessarily the case that
$$ \prod_{j\not\in\cK}N_j \leq x^{\delta_1} .$$
This implies~$\card\cK\geq 1$. Define
$$ \cW := \{(u_n)\in\bfC^\bfN :\ |u_n|\leq \tau(n)^k \quad (n\geq 1)\}. $$
Summarizing the above, we have
\begin{equation}
\cS_3 \ll_{k,\ee} x^\ee(\cA + \cB_3 + \cB_2 + \cB_1 ),\label{eq:Tk-S3}
\end{equation}
where
$$ \cA = \sup_{\substack{x^{\delta_1} < N \leq x^{1/3-\delta_1} \\ MN=x \\(\alpha_m), (\beta_n)\in\cW }} \Big| \ssum{Q<q\leq 2Q \\ (q, a)=1} \ssum{N<n\leq 2^kN \\ M2^{-k} <m < 2M \\ x<mn\leq 2x} \alpha_m\beta_n \cu_R(n m \bar{a} ; q) \Big|, $$
$$ \cB_3 = \sup_{\substack{N_1, N_2, N_3 > x^{1/3-\delta_1} \\  MN_1N_2N_3=x \\ (\alpha_m)\in\cW }} \Big| \ssum{Q<q\leq 2Q \\ (q, a)=1} \ssum{N_j<n_j\leq 2N_j \\ M/8 <m \leq 2M \\ x<mn_1n_2n_3 < 2x} \alpha_m \cu_R(n_1n_2n_3 m \bar{a} ; q) \Big|, $$
$$ \cB_2 = \sup_{\substack{N_1, N_2 > x^{1/3-\delta_1} \\ N_1N_2 > x^{1-\delta_1} \\ MN_1N_2=x \\ (\alpha_m)\in\cW }} \Big| \ssum{Q<q\leq 2Q \\ (q, a)=1} \ssum{N_j<n_j\leq 2N_j \\ M/8 <m \leq 2M \\ x<mn_1n_2 < 2x} \alpha_m \cu_R(n_1n_2 m \bar{a} ; q) \Big|, $$
$$ \cB_1 = \sup_{\substack{N > x^{1-\delta_1} \\  MN=x \\ (\alpha_m)\in\cW }} \Big| \ssum{Q<q\leq 2Q \\ (q, a)=1} \ssum{N<n\leq 2 \\ M/8 <m \leq 2M \\ x<mn < 2x} \alpha_m \cu_R(nm \bar{a} ; q) \Big|. $$
We will focus on~$\cA$ and~$\cB_3$, since the treatment of~$\cB_1$ and~$\cB_2$ is analogous to~$\cB_3$ and actually simpler.

\subsubsection{Separation of variables}

Fix another small parameter~$\delta_2>0$. We smoothen the cutoff using a smooth function~$\phi:\bfR\to[0, 1]$ with~$\phi(\xi) = 1$ for~$\xi\in[1, 2]$,~$\phi(\xi)=0$ for~$\xi\not\in[1-\cE^{-\delta_2}, 2+\cE^{-\delta_2}]$, whose derivatives satisfy~$\|\phi^{(j)}\|_\infty \ll_j \cE^{j\delta_2}$. The cost of replacing in~$\cA$ and $\cB_3$ the sharp cutoff condition~$x<nm\leq 2x$ (resp. $x<n_1n_2n_3m\leq2x$) by~$\phi(nm/x)$ (resp.~$\phi(n_1n_2n_3m/x)$) is at most~$O(x\cE^{-\delta_2/2})$, by trivially bounding the contribution of the transition ranges using~Lemma~\ref{lemme:shiu}.

Integration by parts shows that the Mellin transform~${\breve \phi}(s) = \int_0^\infty \phi(\xi)\xi^{s-1}\dd\xi$ satisfies
$$ {\breve \phi}(it) \ll \frac{\cE^{5\delta_2}}{1+|t|^5} \qquad (t\in\bfR). $$
We use the inversion formula~$\phi(\xi) = (2\pi)^{-1}\int_\bfR {\breve \phi}(it) \xi^{-it}\dd t$ at~$\xi = nm/x$ (resp.~$\xi=mn_1n_2n_3/x$) in the case of~$\cA$ (resp.~$\cB_3$), to obtain the upper bounds
\begin{equation}
\cA \ll_k x\cE^{-\delta_2/2} + \cE^{5\delta_2} \sup_{\substack{x^{\delta_1}<N\leq x^{1/3-\delta_1},\\ MN=x\\ (\alpha_m), (\beta_n)\in\cW}}
\Big|\ssum{Q<q\leq 2Q \\ (q, a)=1} \ssum{N<n\leq 2^kN \\ M2^{-k}<m\leq 2M}\alpha_m\beta_n \cu_R(mn\bar{a} ; q) \Big|,\label{eq:Tk-cas1}
\end{equation}
\begin{equation}
\begin{aligned}
\cB_3 \ll_k x\cE^{-\delta_2/2} + \cE^{5\delta_2} & \sup_{\substack{N_1, N_2, N_3>x^{1/3-\eta}, \\ (\alpha_m) \in\cW, \ t\in\bfR}}\frac1{1+|t|^3} \times \\ &\times \Big|\ssum{Q<q\leq 2Q \\ (q, a)=1} \ssum{N_j<n_j\leq 2N_j \\ M/8<m\leq 2M} \alpha_m(n_1n_2n_3)^{it} \cu_R(n_1n_2n_3m\bar{a} ; q)
\Big|.
\end{aligned}\label{eq:Tk-cas2}
\end{equation}

\subsubsection{The case of~$\cA$}

Let~$(\alpha_m)$, $(\beta_n)$ and~$N$ be given as in the supremum in~\eqref{eq:Tk-cas1}. We wish to bound
\begin{equation}
\cS_a := \ssum{Q<q\leq 2Q \\ (q, a)=1} \ssum{N<n\leq 2^kN \\ M2^{-k} < m< 2M} \alpha_m\beta_n \cu_R(mn\bar{a}; q). \label{eq:Tk-cas1-2}
\end{equation}
By dyadic decomposition, enlarging our bound by a factor of~$k^2$, we may assume the conditions are~$N_1<n\leq 2N_1$ and~$M_1<m\leq 2M_1$ for~$M_1N_1 \in[x2^{-k}, x2^{k+1}]$. Suppose first~$Q\geq x^{1/2-\delta_1/2}$. Then Theorem~\ref{thm:distrib-convo} with~$\eta\gets \min\{\delta_1/2, 1/30\}$ gives the existence of~$\delta_3>0$ depending on~$\delta_1$ such that~\eqref{eq:Tk-cas1-2} is majorized by~$O(2^kx\cE^{-\delta_3})$, on the condition that~$|a|\leq 2^{-k}x^{\delta_3}$ and~$Q\leq 2^{-k}x^{1/2+\delta_3}$, which are satisfied assuming~$\eta<\delta_3/4$ and taking~$x$ large enough in terms of~$k$.

If on the contrary~$Q\leq x^{1/2-\delta_1/2}$, we appeal to Lemma~\ref{lemme:motohashi}. We again obtain for~\eqref{eq:Tk-cas1-2} a bound
$$ \cS_a \ll_j 2^k x\cE^{-\delta_3} $$
for some~$\delta_3$ (depending on~$\delta_1$).

Summarizing, we have obtained in any case
\begin{equation}
\cA \ll_k x\cE^{-\delta_2/2}+x\cE^{5\delta_2 - \delta_3}\label{eq:Tk-majo-A}
\end{equation}
for~$\delta_3>0$. Choosing~$\delta_2$ appropriately, it is an acceptable error term once we can prove that~$\delta_1>0$ can be chosen independently of~$k$.

\subsubsection{The case of~$\cB_3$}

Let~$(\alpha_m)$, $N_1, N_2, N_3>x^{1/3-\delta_1}$ and~$t\in\bfR$ be as in supremum in~\eqref{eq:Tk-cas2}. The quantity we wish to bound is at most
$$ \cS_b := \frac1{1+|t|^3}\sum_{M/8\leq m \leq 2M}\ssum{Q<q\leq 2Q\\(q, am)=1}\Big| \ssum{n_1, n_2, n_3 \\ N_j \leq n_j \leq 2N_j} (n_1n_2n_3)^{it} \cu_R(n_1n_2n_3m\bar{a} ; q)\Big| $$
where~$N_1N_2N_3M = x$ and~$M < x^{3\delta_1}$. Writing~$n_j^{it} = (2N_j)^{it} - it\int_{n_j}^{2N_j} z^{it-1}\dd z$, the above is bounded by
\begin{equation}
\cS_b \ll_\ee \sup_{\substack{N_1', N_2', N_3' \\ N_j < N_j' \leq 2N_j}} \sum_{M/8\leq m \leq M}\ssum{Q<q\leq 2Q\\(q, am)=1}\Big| \ssum{n_1, n_2, n_3 \\ N_j \leq n_j \leq N_j'} \cu_R(n_1n_2n_3m\bar{a} ; q)\Big|\label{eq:Tk-cas2-2}
\end{equation}
Fix~$N_1', N_2', N_3'$ as in the supremum. Using~\eqref{eq:expr-cuR-cu1} and the triangle inequality,
$$ \cS_b \leq \cS_b' + \cS_b'', $$
where
\begin{equation}
\cS_b' = \sum_{M/8\leq m \leq M}\ssum{Q<q\leq 2Q\\(q, am)=1}\Big| \ssum{n_1, n_2, n_3 \\ N_j \leq n_j \leq N_j'} \cu_1(n_1n_2n_3m\bar{a})\Big|,\label{eq:Tk-cas2-u1}
\end{equation}
\begin{equation}
\cS_b'' = \sum_{M/8<m\leq M}\sum_{Q < q\leq 2Q} \frac1{\vphi(q)}\ssum{\chi\mod{q} \\ 1<\cond(\chi)\leq R} \prod_{j=1}^3 \Big|\sum_{N_j<n\leq N_j'} \chi(n) \Big|.\label{eq:Tk-cas2-u2}
\end{equation}
To~$\cS_b'$ we apply~\cite[Lemma~2]{BFI2} for each~$q$ individually (note that this is a very deep result~\cite{FrIw3,HB}, relying on Deligne's proof of the Weil conjectures~\cite{Deligne}). For some small, absolute~$\delta_4$, on the condition that~$Q\leq x^{1/2+\delta_4}$ (requiring~$\eta<\delta_4/2$), the quantity~\eqref{eq:Tk-cas2-u1} is bounded by
\begin{equation}
\cS_b' \ll Mx^{1-\delta_4}  \leq x^{1-\delta_4+3\delta_1}.\label{eq:Tk-cas2-contribu1}
\end{equation}
Consider then~$\cS_b''$. By Lemma~\ref{lemme:PV}, each sum over~$n$ is bounded by~$O_\ee(x^\ee R^{1/2})$, and so we obtain a bound
$$ \cS_b'' \ll_\ee x^\ee R^{5/2}M $$
which is absorbed in the term~\eqref{eq:Tk-cas2-contribu1}. Inserting in~\eqref{eq:Tk-cas2}, we have obtained for~$\cB_3$ a bound
\begin{equation}
\cB_3 \ll x\cE^{-\delta_2} + \cE^{5\delta_2}x^{1-\delta_4 + 3\delta_1}.\label{eq:Tk-majo-B}
\end{equation}
The terms~$\cB_2$ and~$\cB_1$ are shown in the same way to satisfy the same bound with~$\delta_4>0$ absolute and small enough. Choosing our parameters adequately, we can choose absolute constants~$\delta_1$,~$\delta_2$,~$\delta_3$ in such a way that both bounds~\eqref{eq:Tk-majo-B} and~\eqref{eq:Tk-majo-A} are true and~$O(x \cE^{-\eta})$. Inserting back into~\eqref{eq:Tk-S3} and~\eqref{eq:Tk-apres-dicho-2}, we obtain the claimed bound~\eqref{eq:Tk-apres-dicho}.

\subsection{Proof of Theorems~\ref{thm:Tk} and~\ref{thm:Tk-cond}}

As a last step, we deduce from Theorem~\ref{thm:distrib-tauk} the estimate
\begin{equation}
\ssum{q\leq \sqrt{x} \\ (q, a)=1} \Big( \ssum{n\leq q^2 \\ n\equiv a\mod{q}}\tau_k(n) - \frac1{\vphi(q)}\ssum{n\leq  q^2 \\ (n, q)=1} \tau_k(n)\Big) \ll_k x\cE^{-\eta} \qquad (0<|a|\leq x^{\eta}) \label{eq:Tk-q2}
\end{equation}
where as before~$\cE=x$ if the generalized Lindelöf is true and~$\cE=x^{1/k}$ otherwise. Let~$\Delta\in(0, 1/10)$ be fixed and decompose the sums over~$q$ and~$n$ into intervals~$((1+\Delta)^{-1}Q, Q]$ and~$((1+\Delta)^{-1}N, N]$. Calling~$\cS'_1$ the left-hand side of~\eqref{eq:Tk-q2}, we have
$$ \cS'_1 \ll \ssum{j_0, j_1 \geq 0 \\ Q = (1+\Delta)^{-j_0}\sqrt{x}\\ N = (1+\Delta)^{-j_1}x } \Big|\ssum{(1+\Delta)^{-1}Q<q\leq Q} \ssum{(1+\Delta)^{-1}N<n\leq N \\ n\leq q^2} \tau_k(n)\cu_1(n\bar{a} ; q)\Big|, $$
where we used the notation~\eqref{eq:def-cu1}. The inner sums are void if~$Q^2\leq N$ and the condition~$n\leq q^2$ is automatically satisfied if~$N \leq Q^2(1+\Delta)^{-2}$. The contribution of~$j_0, j_1$ such that~$(1+\Delta)^{-2}Q^2\leq N\leq Q^2$ is at most
$$ \ssum{q\leq \sqrt{x} \\ (q, a)=1} \sum_{q^2(1+\Delta)^{-3} \leq n\leq q^2(1+\Delta)^2} \tau_k(n)|\cu_1(n\bar{a} ; q)| \ll \Delta x (\log x)^k $$
by virtue of Lemma~\ref{lemme:shiu}. Therefore
$$ \cS'_1 \ll \Delta x (\log x)^k + (\log x)^2\Delta^{-2} \sup_{\substack{Q\leq \sqrt{x} \\ N\leq Q^2}} \Big| \ssum{(1+\Delta)^{-1}Q<q\leq Q} \ssum{(1+\Delta)^{-1}N<n\leq N} \tau_k(n)\cu_1(n\bar{a} ; q)\Big|. $$
Let~$Q$,~$N$ be as in the supremum, and let~$\eta>0$ be the real number given by Theorem~\ref{thm:distrib-tauk}. Lemma~\ref{lemme:shiu} gives the bound
$$ \Big| \ssum{(1+\Delta)^{-1}Q<q\leq Q} \ssum{(1+\Delta)^{-1}N<n\leq N} \tau_k(n)\cu_1(n\bar{a} ; q)\Big| \ll_\ee x^\ee N $$
which is acceptable if~$N\leq x^{1-\eta/10}$. Suppose~$N\geq x^{1-\eta/10}$, then Theorem~\ref{thm:distrib-tauk} applies with~$x\gets N$ and yields a bound~$O(x\cE^{-\eta/10})$ for~$|a|\leq x^{\eta/10}$. Therefore,
$$ \cS'_1 \ll_{\ee, k} x^{1+\ee}\Delta + \Delta^{-2}x^{1+\ee}\cE^{-\eta/10} .$$
Taking \textit{e.g.}~$\Delta = \cE^{-\eta/30}$ and reinterpreting~$\eta$, we have the claimed estimate~\eqref{eq:Tk-q2}.

\medskip

From the Dirichlet hyperbola method, Theorem~\ref{thm:distrib-tauk} and estimate~\eqref{eq:Tk-q2}, we deduce
\begin{align*}
\cT_k(x) &\ = 2\sum_{q\leq \sqrt{x}} \ssum{q^2< n\leq x \\ n\equiv -1\mod{q}}\tau_k(n) + O_\ee(x^{1/2+\ee}) \\
&\ = 2\ssum{q\leq \sqrt{x}} \frac{1}{\vphi(q)}\ssum{q^2< n\leq x \\ (n, q)=1}\tau_k(n) + O(x\cE^{-\delta})
\end{align*}
The main terms are computed in~\cite[Théorème~2]{FT-Piltz}, with an error term~$O(x^{1-\delta/k})$ (unconditionally). If one assumes the generalized Lindelöf hypothesis, then the proof is adapted in the following way. Under the hypotheses and in the notations of~\cite[Lemma~6]{FT-Piltz}, there holds~$|\theta(p^\nu)|\leq Cp^{-\delta}{k \choose \floor{k/2}}$ (\cite[first display page~52]{FT-Piltz}). Therefore the series~$F_k(s)$ in~\cite[Lemma~7]{FT-Piltz} is bounded in terms of~$k$ only in the half-plane~$\Re(s)\geq 1-\delta/2$. In the proof of~\cite[Lemma~7]{FT-Piltz}, one chooses~$T=x^{\delta/2}$ and shift the contour to~$\Re(s)=1-\delta/2$, where the Lindelöf hypothesis implies~$\zeta(s)\ll t^\ee$ by convexity, to produce the conclusion
$$ \sum_{n\leq x}\Psi(n)\tau_k(n) = xQ_{k-1}(\log x) + O_{\ee, k}(x^{1-\delta/2+\ee}). $$
The rest of the argument in Corollaries~1-2 of Lemma~7, and Corollary of Lemma~8 of~\cite{FT-Piltz} are transposed \textit{verbatim} to yield
$$ 2\ssum{q\leq \sqrt{x}} \frac{1}{\vphi(q)}\ssum{q^2< n\leq x \\ (n, q)=1}\tau_k(n) = xP_k(\log x) + O_k(x^{1-c}) $$
for some~$c>0$, as claimed.

\subsection{Remark on the uniformity in~$a$}
\label{sec:remark-uniformity-a}

If we were to replace the shift~$\tau(n+1)$ by~$\tau(n+a)$, $0<|a|\leq x^{\delta}$, then the deduction of an asymptotic formula analogous to~\eqref{eq:estim-Tk} from Theorem~\ref{thm:distrib-tauk} would go along similar lines. We briefly indicate how one reduces to our previous setting. From Dirichlet's hyperbola method, the problem reduces to the evaluation of
$$ \cS_{k, a}(x) = 2\sum_{q\leq \sqrt{x}} \ssum{q^2 \leq n\leq x \\ n\equiv -a \mod{q}} \tau_k(n). $$
Extracting the largest factor~$d_1|a^\infty$ from~$n$, we rewrite this as
$$ \cS_{k, a}(x) = 2\sum_{d_1|a^\infty} \tau_k(d_1) \sum_{q\leq \sqrt{x}} \ssum{q^2/d_1\leq n\leq x/d_1 \\ (n, a)=1 \\  nd_1 \equiv -a \mod{q}} \tau_k(n). $$
Writing~$d_2 := (q, d_1)$, the congruence condition is equivalent to~$d_2|a$ and
$$ n\equiv -(a/d_2)\bar{(d_1/d_2)}\mod{q/d_2}. $$
We therefore have
$$ \cS_{k, a}(x) = 2\sum_{d_1|a^\infty} \tau_k(d_1) \sum_{d_2|(d_1, a)} \ssum{q\leq \sqrt{x}/d_2 \\ (q, d_1/d_2)=(q, a/d_2) = 1} \ssum{q^2/d_1\leq n\leq x/d_1 \\ (n, a)=1 \\ n \equiv -(a/d_2)\bar{(d_1/d_2)} \mod{q}} \tau_k(n). $$
Summing for each~$d_j$ individually, the contribution of~$d_1 > x^\delta$ is bounded trivially using Lemma~\ref{lemme:shiu}. When~$d_1\leq x^\delta$, the sum over~$n$ and~$q$ is handled by an adequate generalization of Theorem~\ref{thm:distrib-tauk}, involving a congruence of the type~$n\equiv b_1\bar{b_2}\mod{q}$, as well as an additional coprimality condition~$(n, b_3)=1$, for integers~$|b_j|\leq x^\delta$. Our arguments readily adapt to account for both these modifications. Note however that it is now important that the method is able to handle values of the modulus~$q$ up to~$x^{1/2+\delta}$, with~$\delta$ independent of~$k$ (\textit{cf.} the statement of Theorem~\ref{thm:distrib-tauk}).

\bibliographystyle{amsalpha}
\bibliography{divtm-siegel}

\newcommand{\etalchar}[1]{$^{#1}$}
\providecommand{\bysame}{\leavevmode\hbox to3em{\hrulefill}\thinspace}
\providecommand{\MR}{\relax\ifhmode\unskip\space\fi MR }
\providecommand{\MRhref}[2]{%
  \href{http://www.ams.org/mathscinet-getitem?mr=#1}{#2}
}
\providecommand{\href}[2]{#2}
\begin{thebibliography}{BHM07b}

\bibitem[ABSR15]{ABR}
J.~C. Andrade, L.~Bary-Soroker, and Z.~Rudnick, \emph{Shifted convolution and
  the {T}itchmarsh divisor problem over~{${\mathbb F}_q[t]$}}, Philos. Trans. A
  Roy. Soc. London \textbf{373} (2015), no.~2040, to appear.

\bibitem[BFI86]{BFI}
E.~Bombieri, J.~B. Friedlander, and H.~Iwaniec, \emph{Primes in arithmetic
  progressions to large moduli}, Acta Math. \textbf{156} (1986), no.~3-4,
  203--251.

\bibitem[BFI87]{BFI2}
\bysame, \emph{Primes in arithmetic progressions to large moduli. {II}}, Math.
  Ann. \textbf{277} (1987), no.~3, 361--393.

\bibitem[BHM07a]{BHM}
V.~Blomer, G.~Harcos, and P.~Michel, \emph{Bounds for modular {$L$}-functions
  in the level aspect}, Ann. Sci. \'Ecole Norm. Sup. (4) \textbf{40} (2007),
  no.~5, 697--740.

\bibitem[BHM07b]{BHM-Burgess}
\bysame, \emph{A {B}urgess-like subconvex bound for twisted {$L$}-functions},
  Forum Math. \textbf{19} (2007), no.~1, 61--105, Appendix 2 by Z. Mao.

\bibitem[BM15a]{BM}
V.~Blomer and D.~Mili{\'c}evi{\'c}, \emph{Kloosterman sums in residue classes},
  J. Eur. Math. Soc. \textbf{17} (2015), no.~1, 51--69.

\bibitem[BM15b]{BM2}
\bysame, \emph{The second moment of twisted modular {$L$}-functions}, Geom.
  Func. Anal. (2015), to appear.

\bibitem[BV69]{BarbanVekhov}
M.~B. Barban and P.~P. Vehov, \emph{Summation of multiplicative functions of
  polynomials}, Mat. Zametki \textbf{5} (1969), 669--680.

\bibitem[BV87]{BV}
V.~A. Bykovski{\u\i} and A.~I. Vinogradov, \emph{Inhomogeneous convolutions},
  Zap. Nauchn. Sem. Leningrad. Otdel. Mat. Inst. Steklov. (LOMI) \textbf{160}
  (1987), no.~Anal. Teor. Chisel i Teor. Funktsii. 8, 16--30, 296.

\bibitem[CG01]{CG}
J.~B. Conrey and S.~M. Gonek, \emph{High moments of the {R}iemann
  zeta-function}, Duke Math. J. \textbf{107} (2001), no.~3, 577--604.

\bibitem[dB62]{deBruijn}
N.~G. de~Bruijn, \emph{On the number of integers {$\leq x$} whose prime factors
  divide {$n$}.}, Illinois J. Math. \textbf{6} (1962), 137--141.

\bibitem[Del74]{Deligne}
P.~Deligne, \emph{La conjecture de {W}eil. {I}}, Inst. Hautes \'Etudes Sci.
  Publ. Math. (1974), no.~43, 273--307.

\bibitem[Des82]{Deshouillers}
J.-M. Deshouillers, \emph{Majorations en moyenne de sommes de {K}loosterman},
  Seminar on {N}umber {T}heory, 1981/1982, Univ. Bordeaux I, Talence, 1982,
  pp.~Exp. No. 3, 5.

\bibitem[DFI02]{DFI}
W.~Duke, J.~B. Friedlander, and H.~Iwaniec, \emph{The subconvexity problem for
  {A}rtin {$L$}-functions}, Invent. Math. \textbf{149} (2002), no.~3, 489--577.

\bibitem[DI82a]{DI-BinDiv}
J.-M. Deshouillers and H.~Iwaniec, \emph{An additive divisor problem}, J.
  London Math. Soc. (2) \textbf{26} (1982), no.~1, 1--14.

\bibitem[DI82b]{DI}
\bysame, \emph{Kloosterman sums and fourier coefficients of cusp forms},
  Invent. Math. \textbf{70} (1982), no.~2, 219--288.

\bibitem[Dra15]{Drappeau}
S.~Drappeau, \emph{Théorèmes de type {F}ouvry--{I}waniec pour les entiers
  friables}, Compos. Math. (2015), to appear.

\bibitem[ES34]{ErdosSzekeres}
P.~Erd{\H{o}}s and G.~Szekeres, \emph{{\"Uber die Anzahl der Abelschen Gruppen
  gegebener Ordnung und \"uber ein verwandtes zahlentheoretisches Problem}},
  {Acta Litt. Sci. Szeged} \textbf{7} (1934), 95--102.

\bibitem[Fel12]{Felix}
A.~T. Felix, \emph{Generalizing the {T}itchmarsh divisor problem}, Int. J.
  Number Theory \textbf{8} (2012), no.~3, 613--629.

\bibitem[FG92]{FG}
J.~B. Friedlander and A.~Granville, \emph{Relevance of the residue class to the
  abundance of primes}, Proceedings of the {A}malfi {C}onference on {A}nalytic
  {N}umber {T}heory ({M}aiori, 1989), Univ. Salerno, Salerno, 1992,
  pp.~95--103.

\bibitem[FI83]{FI}
É.~Fouvry and H.~Iwaniec, \emph{Primes in arithmetic progressions}, Acta Arith.
  \textbf{42} (1983), no.~2, 197--218.

\bibitem[FI85]{FrIw3}
J.~B. Friedlander and H.~Iwaniec, \emph{Incomplete {K}loosterman sums and a
  divisor problem}, Ann. of Math. (2) \textbf{121} (1985), no.~2, 319--350,
  With an appendix by B. J. Birch and E. Bombieri.

\bibitem[FI05]{FrIw2}
\bysame, \emph{Summation formulae for coefficients of {$L$}-functions}, Canad.
  J. Math. \textbf{57} (2005), no.~3, 494--505.

\bibitem[Fio12a]{Fiorilli2}
D.~Fiorilli, \emph{On a theorem of {B}ombieri, {F}riedlander, and {I}waniec},
  Canad. J. Math. \textbf{64} (2012), no.~5, 1019--1035.

\bibitem[Fio12b]{Fiorilli}
\bysame, \emph{Residue classes containing an unexpected number of primes}, Duke
  Math. J. \textbf{161} (2012), no.~15, 2923--2943.

\bibitem[Fou82]{Fouvry82}
\'E. Fouvry, \emph{R\'epartition des suites dans les progressions
  arithmétiques}, Acta Arith. \textbf{41} (1982), no.~4, 359 382.

\bibitem[Fou85]{FouvryTit}
{\'E}.~Fouvry, \emph{Sur le probl\`eme des diviseurs de {T}itchmarsh}, J. Reine
  Angew. Math. \textbf{357} (1985), 51--76.

\bibitem[FT85]{FT-Piltz}
{\'E}.~Fouvry and G.~Tenenbaum, \emph{Sur la corr\'elation des fonctions de
  {P}iltz}, Rev. Mat. Iberoamericana \textbf{1} (1985), no.~3, 43--54.

\bibitem[GR07]{GR}
I.~S. Gradshteyn and I.~M. Ryzhik, \emph{Table of integrals, series, and
  products}, seventh ed., Elsevier/Academic Press, Amsterdam, 2007, Translated
  from the Russian, Translation edited and with a preface by A. Jeffrey and D.
  Zwillinger.

\bibitem[Hal67]{Halberstam}
H.~Halberstam, \emph{Footnote to the {T}itchmarsh-{L}innik divisor problem},
  Proc. Amer. Math. Soc. \textbf{18} (1967), 187--188.

\bibitem[Har11]{Harper-AB}
A.~J. Harper, \emph{On finding many solutions to {S}-unit equations by solving
  linear equations on average}, preprint (2011), arXiv,
  \url{http://arxiv.org/abs/1108.3819}.

\bibitem[HB82]{HB3}
D.~R. Heath-Brown, \emph{Prime numbers in short intervals and a generalized
  {V}aughan identity}, Canad. J. Math. \textbf{34} (1982), no.~6, 1365--1377.

\bibitem[HB86]{HB}
\bysame, \emph{The divisor function {$d_3(n)$} in arithmetic progressions},
  Acta Arith. \textbf{47} (1986), no.~1, 29--56.

\bibitem[Hen12]{Henriot}
K.~Henriot, \emph{Nair-{T}enenbaum bounds uniform with respect to the
  discriminant}, Math. Proc. Cambridge Philos. Soc. \textbf{152} (2012), no.~3,
  405--424.

\bibitem[IK04]{IK}
H.~Iwaniec and E.~Kowalski, \emph{Analytic number theory}, vol.~53, Cambridge
  Univ. Press, 2004.

\bibitem[Iwa82]{Iwaniec-MV}
H.~Iwaniec, \emph{Mean values for {F}ourier coefficients of cusp forms and sums
  of {K}loosterman sums}, Number theory days, 1980 ({E}xeter, 1980), London
  Math. Soc. Lecture Note Ser., vol.~56, Cambridge Univ. Press, Cambridge-New
  York, 1982, pp.~306--321.

\bibitem[Iwa95]{Iwaniec-Book-Automorphic}
\bysame, \emph{Introduction to the spectral theory of automorphic forms},
  Biblioteca de la Revista Matem\'atica Iberoamericana., Revista Matem\'atica
  Iberoamericana, Madrid, 1995.

\bibitem[Iwa97]{Iwaniec-Book-Topics}
\bysame, \emph{Topics in classical automorphic forms}, Graduate Studies in
  Mathematics, vol.~17, American Mathematical Society, Providence, RI, 1997.

\bibitem[Iwa02]{Iwaniec-Book-Spectral}
\bysame, \emph{Spectral methods of automorphic forms}, second ed., Graduate
  Studies in Mathematics, vol.~53, American Mathematical Society, Providence,
  RI; Revista Matem\'atica Iberoamericana, Madrid, 2002.

\bibitem[Kim03]{Kim}
H.~H. Kim, \emph{Functoriality for the exterior square of {${\rm GL}_4$} and
  the symmetric fourth of {${\rm GL}_2$}}, J. Amer. Math. Soc. \textbf{16}
  (2003), no.~1, 139--183, With appendix 1 by D. Ramakrishnan and appendix 2 by
  H. H. Kim and P. Sarnak.

\bibitem[KL13]{KL}
A.~Knightly and C.~Li, \emph{Kuznetsov's trace formula and the {H}ecke
  eigenvalues of {M}aass forms}, Mem. Amer. Math. Soc. \textbf{224} (2013),
  no.~1055, vi+132.

\bibitem[Lin63]{Linnik}
Ju.~V. Linnik, \emph{The dispersion method in binary additive problems},
  Translated by S. Schuur, American Mathematical Society, Providence, R.I.,
  1963.

\bibitem[May15]{Maynard}
J.~Maynard, \emph{Small gaps between primes}, Ann. of Math. (2) \textbf{181}
  (2015), no.~1, 383--413.

\bibitem[Mot76]{Motohashi}
Y.~Motohashi, \emph{An induction principle for the generalization of
  {B}ombieri's prime number theorem}, Proc. Japan Acad. \textbf{52} (1976),
  no.~6, 273--275.

\bibitem[MV07]{MontgomeryVaughan}
H.~L. Montgomery and R.~C. Vaughan, \emph{Multiplicative number theory. {I}.
  {C}lassical theory}, Cambridge Studies in Advanced Mathematics, vol.~97,
  Cambridge University Press, Cambridge, 2007.

\bibitem[PCF{\etalchar{+}}14]{Polymath}
D.~H.~J. Polymath, W.~Castryck, {\'E}.~Fouvry, G.~Harcos, E.~Kowalski,
  P.~Michel, P.~Nelson, E.~Paldi, J.~Pintz, A.~V. Sutherland, T.~Tao, and X.-F.
  Xie, \emph{New equidistribution estimates of {Z}hang type}, Algebra Number
  Theory \textbf{8} (2014), no.~9, 2067--2199.

\bibitem[Pit13]{Pitt}
N.~J.~E. Pitt, \emph{On an analogue of {T}itchmarsh's divisor problem for
  holomorphic cusp forms}, J. Amer. Math. Soc. \textbf{26} (2013), no.~3,
  735--776.

\bibitem[Pro03]{Proskurin}
N.~V. Proskurin, \emph{On general {K}loosterman sums}, Zap. Nauchn. Sem.
  S.-Peterburg. Otdel. Mat. Inst. Steklov. (POMI) \textbf{302} (2003),
  no.~Anal. Teor. Chisel i Teor. Funkts. 19, 107--134, 200.

\bibitem[Rod65]{Rodriquez}
G.~Rodriquez, \emph{Sul problema dei divisori di {T}itchmarsh}, Boll. Un. Mat.
  Ital. (3) \textbf{20} (1965), 358--366.

\bibitem[Sar95]{Sarnak}
P.~Sarnak, \emph{Selberg's eigenvalue conjecture}, Notices Amer. Math. Soc.
  \textbf{42} (1995), no.~11, 1272--1277.

\bibitem[Shi80]{Shiu}
P.~Shiu, \emph{A {B}run-{T}itchmarsh theorem for multiplicative functions}, J.
  Reine Angew. Math. \textbf{313} (1980), 161--170.

\bibitem[Ten95]{Tenenbaum-Intro}
G.~Tenenbaum, \emph{Introduction to analytic and probabilistic number theory},
  Cambridge Studies in Advanced Mathematics, vol.~46, Cambridge University
  Press, Cambridge, 1995, Translated from the second French edition (1995) by
  C. B. Thomas.

\bibitem[Tit30]{Titchmarsh}
E.~C. Titchmarsh, \emph{{A divisor problem.}}, {Rend. Circ. Mat. Palermo}
  \textbf{54} (1930), 414--429.

\bibitem[Tit86]{TM-zeta}
\bysame, \emph{The theory of the {R}iemann zeta-function}, second ed., The
  Clarendon Press, Oxford University Press, New York, 1986, Edited and with a
  preface by D. R. Heath-Brown.

\bibitem[Top15]{Topacogullari}
B.~Topacogullari, \emph{The shifted convolution of divisor functions}, preprint
  (2015), arXiv, \url{http://arxiv.org/abs/1506.02608v1}.

\bibitem[Wei48]{Weil}
A.~Weil, \emph{On some exponential sums}, Proc. Nat. Acad. Sci. U. S. A.
  \textbf{34} (1948), 204--207.

\bibitem[Zha14]{Zhang}
Y.~Zhang, \emph{Bounded gaps between primes}, Ann. of Math. (2) \textbf{179}
  (2014), no.~3, 1121--1174.

\end{thebibliography}

\affiliationone{I2M (UMR 7373) \\
Aix-Marseille Université \\
163 av. de Luminy, Case 901 \\
13009 Marseille (France)
\email{sary-aurelien.drappeau@univ-amu.fr}}

\end{document}